\documentclass[a4paper,11pt]{article}

\usepackage{amsmath}
\usepackage{amsthm}
\usepackage{amssymb}
\usepackage{mathtools}
\counterwithin{equation}{section}
\usepackage{mathrsfs}
\usepackage{physics}
\usepackage{stmaryrd} % double brackets
\usepackage{xspace}
\usepackage[margin=1.0in]{geometry}
\usepackage{dsfont, bbm}
\usepackage{graphicx}
\usepackage{xcolor}

\usepackage[hidelinks]{hyperref}

\numberwithin{equation}{section}
\newtheorem{theorem}{Theorem}[section]
\newtheorem{lemma}[theorem]{Lemma}
\newtheorem{proposition}[theorem]{Proposition}
\newtheorem{corollary}[theorem]{Corollary}

\newtheorem{definition}[theorem]{Definition}

\theoremstyle{remark}
\newtheorem{remark}[theorem]{Remark}
\newtheorem{theorem*}[theorem]{Theorem}
\newtheorem{lemma*}[theorem]{Lemma}
\newtheorem{proposition*}[theorem]{Proposition}
\newtheorem{corollary*}[theorem]{Corollary}
\newtheorem{assumption*}[theorem]{Assumption}
	
	\renewcommand{\a}{\alpha}
	
	\newcommand{\g}{\gamma}

	\newcommand{\z}{\zeta}

	\newcommand{\s}{\sigma}

	\renewcommand{\o}{\omega}

	\newcommand{\E}{{\mathbb E}} 
	  
	\newcommand{\N}{{\mathbb N}}  
	\renewcommand{\P}{{\mathbb P}} 
	
	\newcommand{\R}{{\mathbb R}} 
	\newcommand{\T}{{\mathbb T}}

	\let\cal=\mathcal

	\newcommand{\FF}{{\cal F}}
	
	\newcommand{\HH}{{\cal H}}

	\newcommand{\cB}{{\cal B}}

	\newcommand{\cF}{{\cal F}}

	\newcommand{\cR}{{\cal R}}

	\newcommand{\bz}{{\boldsymbol z}}

        \newcommand{\1}{\mathbbm{1}}
	\renewcommand{\log}{\ln} 
	
	\newcommand{\eee}{\mathrm{e}}
	\renewcommand{\d}{\mathrm{d}}
	\renewcommand{\dd}{\mathtt{d}}

	\newcommand{\kl}[1]{\left(#1\right)}
	\newcommand{\bkl}[1]{\big(#1\big)}
	\newcommand{\bbkl}[1]{\Big(#1\Big)}
	\newcommand{\gkl}[1]{\left\{#1\right\}}

    \newcommand{\ga}[1]{
        \langle #1 \rangle
    }
    \newcommand{\bga}[1]{
        \big\langle #1 \big\rangle
    }
    \newcommand{\bbga}[1]{
        \Big\langle #1 \Big\rangle
    }

        \newcommand*\iid{i.i.d.\ }
	\newcommand*\wrt{w.r.t.\ }

\newcommand{\PR}[1]{\P\left(#1\right)}

\newcommand{\EX}[1]{\E\left[#1\right]}
\newcommand{\cEX}[2]{\E\left[#1 \,\middle|\, #2 \right]}

\newcommand\brac[1]{\left(#1\right)}
\newcommand\bbrac[1]{\big(#1\big)}
\newcommand\cbrac[1]{\left\{#1\right\}}

\newcommand{\q}{\mathsf{q}}
\newcommand{\qm}{Q^{(M)}}

\newcommand{\qml}{\qm_\leq}
\newcommand{\qmt}{\tilde{Q}^{(M)}}
\newcommand{\qmeq}{\qm_{\text{equidist}}}
\newcommand{\pp}{\mathsf{p}}
\newcommand{\yy}{\mathsf{y}}
\newcommand{\HHH}{\mathsf{H}}

\newcommand\leaves[1]{\cbrac{-1,1}^{#1}}

\def\sfrac#1#2{{\textstyle{\frac{#1}{#2}}}}

\newcommand{\restr}[2]{#1\vert_{#2}}

\newcommand\ar{A_{\mathrm{ext.}}}

\allowdisplaybreaks

\begin{document}
	\title{The Free Energy of an Enriched Continuous Random Energy Model in the Weak Correlation Regime}
	\author{
		Alexander~Alban\thanks{
			Institut f\"{u}r Mathematik,
			Johannes Gutenberg-Universit\"{a}t,
			GERMANY.
			\href{mailto:aalban@uni-mainz.de}
			\texttt{aalban@uni-mainz.de}
		}
		\and
		Fu-Hsuan~Ho\thanks{
			Department of Mathematics, Weizmann Institute of Science, ISRAEL.
			\href{mailto:fu-hsuan.ho@weizmann.ac.il}{\texttt{fu-hsuan.ho@weizmann.ac.il}}
		}
		\and
		Justin~Ko\thanks{
Department of Statistics and Actuarial Science, 
			University of Waterloo, CANADA.
            \href{mailto:justin.ko@uwaterloo.ca}{\texttt{justin.ko@uwaterloo.ca}}
		}
	}
	\date{\today}
	\maketitle
	\begin{abstract}
         We revisit the proof of the limiting free energy of the continuous random energy model (CREM) using the Hamilton--Jacobi approach for mean-field disordered systems. To achieve this, we introduce an enriched model that interpolates between the CREM and the Ruelle probability cascade. We focus on the weak correlation regime, where the CREM’s covariance function $A$ is bounded above by the identity function.
		
		In the weak correlation regime, we show that the free energy is given by the Hopf formula. The resulting expression is independent of $A$, confirming that in this regime the free energy does not depend on the precise form of the covariance function. 
        Outside of the weak correlation regime, the Hamilton–Jacobi framework no longer applies.
        Moreover, we provide an example where a formal application of the associated variational principle fails to yield the correct free energy. 
	\end{abstract}
	
	%\tableofcontents

	\section{Introduction} \label{sec:intro}

	The continuous random energy model (CREM) was introduced in \cite{BK2} as an extension of the generalized random energy model (GREM) in \cite{ derrida1985generalization,GD86b}.  Thanks to the tractability of these models, the limiting free energy of the GREM was computed in \cite{capocaccia1987existence} and then extended to the CREM in \cite{BK2}.
    Both of these models were introduced as an analytically tractable model of spin glasses  \cite{guerra2003broken, PUltra,parisi1979infinite,Talagrand}. In contrast to the other spin glass models, the limits of the free energies in the CREM and GREM have a simple form of the limit.

	More recently, a PDE approach to understanding mean-field disordered systems was developed to shed new insights on the free energies of these models. While most of the work in these areas has been focused on spin glasses \cite{Mourrat2021basis,Mourrat2022parisi,mourrat2020extending} or statistical inference \cite{Chen2022a,Chen2022c,Chen2022b,Chen2023,dominguez2024mutual,Mourrat2020finiterank,Mourrat2021basis}, the tools developed are quite general and can be adapted to study other class of models. In this work, we apply these approaches to a model where the spin configurations do not have independent coordinates. In particular, we revisit the CREM free energy from this point of view and show that the limiting free energy can be expressed as the solution to a particular Hamilton--Jacobi equation. 
    
    We study a generalization of the free energy of the CREM with an enrichment and prove a Hopf variational formula for the limiting free energy in the weak correlation regime. We recall that the weak correlation regime corresponds to covariance functions that are bounded by the identity. This formula reduces to the case of the classical CREM without enrichment. The Hopf formula can be solved explicitly for the CREM, and we obtain a one-parameter variational representation of the free energy. We get as an immediate corollary that the free energy is universal for all models in the weak correlation regime. 
	
	The case outside the weak correlation regime is trickier, since it does not fall under the existing theory developed for monotonic non-linearities \cite{chen_hamilton-jacobi_2023}. This regime is interesting because it is a one-dimensional example of a model non-convex covariance structure which has been a recent topic of great interest in \cite{bates2025balancedmultispeciesspinglasses,chen2024freeenergynonconvexmultispecies, ChenMourrat2023nonconvexVectorSpin,JCnonconvex}. We show that a naive extension of the Hopf formula does not agree with the known formula of the limiting free energy outside of the weak correlation regime. Hopefully, this work provides the groundwork for a simpler toy model, which can be used as an analytically tractable model to understand the non-convex mean field models in future work.

	% ----------------------------------------------------------------------------------
	\subsection{Notation and Definitions} \label{sec:notation and def}
	% ----------------------------------------------------------------------------------
	
	Let $A\colon[0,1]\to[0,1]$ be a right-continuous, increasing function such that $A(0)=0$ and $A(1)=1$.
	For $N\in\N$, the \emph{continuous random energy model} (CREM)  with the covariance function $A$ is a centered Gaussian process $\kl{H_N^A(\sigma)}_{\s\in\leaves{N}}$ with covariance
	\begin{align}\label{eq:3DefCREM}
		\EX{H_N^A(\s)\, H_N^A(\tilde\s)}
		= N A\kl{\sfrac{\s \wedge \tilde\s}{N}},  \qquad \s,\tilde\s \in \{-1,1 \}^N.
	\end{align}
	The \emph{overlap} $\s\wedge\tilde\s$ in \eqref{eq:3DefCREM} is defined as
	\begin{align}\label{eq:3DefOverlap}
		\s \wedge \tilde \s \coloneqq \max \gkl{i=0,\dots,N\colon \restr{\s}{i}=\restr{\tilde\s}{i}},
	\end{align} 
	where for $\s = (\s_1,\dots,\s_N) \in \leaves{N}$, we write $\restr{\s}{0}=\varnothing$ and $\restr{\s}{i}=(\s_1,\dots,\s_i)$ for $i=1,\dots,N$.
	
    We now define the enrichment. For $M\in\N$, define $Q^{(M)}$ to be the subset of non-negative and increasing step functions defined on $[0,1)$ with at most $M$ jumps by
	\begin{align}\label{eq:3qPiecewise}
		Q^{(M)}
		=
		\left\{
		\q_M = \sum_{k=0}^M q_k \1_{[\z_k, \z_{k+1})}
		\,\bigg|\,
		\begin{array}{l}
			0 = \zeta_0 < \zeta_1  < \dots < \zeta_{M - 1} < \zeta_M < \zeta_{M  + 1} = 1 \\
			0 = q_{-1} \leq q_0 \leq \dots \leq q_{M- 1} \leq q_M < \infty
		\end{array}
		\right\}.
	\end{align}
    Given $\q_M = \sum_{k=0}^M q_k \1_{[\z_k, \z_{k+1})} \in \qm$,
	let $\kl{Y_{\q_M}(\sigma,\alpha)}_{\s\in\leaves{N}\!,\, \a\in\N^M}$ be a centered Gaussian process with the covariance function
	\begin{align}\label{eq:3YCovariance}
		\E [Y_{\q_M}(\s, \a) Y_{\q_M}(\tilde\s, \tilde\a)] 
		= \kl{\s \wedge \tilde\s} q_{\a \wedge \tilde\a}, \qquad \forall\: \s,\tilde\s \in \{-1,1 \}^N\! , \, \a,\tilde\a \in \N^M,
	\end{align}    
	where $\a \wedge \tilde\a \coloneqq \max \gkl{k=0,\dots,M\colon \restr{\a}{k}=\restr{\tilde\a}{k}}$.
	Note that
	\begin{align}\label{eq:3YwrtZ}
		\kl{Y_{\q_M}(\sigma,\alpha)}_{\s\in\leaves{N}\!,\, \a\in\N^M} \overset{\d}{=} \kl{\textstyle\sum_{i = 1}^N \sum_{k = 0}^M (q_k - q_{k - 1} )^{1/2} z_{\restr{\s}{i},\restr{\a}{k}}}_{\s\in\leaves{N}\!,\, \a\in\N^M},
	\end{align}
	where each $z_{\restr{\s}{i},\restr{\a}{k}}$ is from a family of \iid standard Gaussian random variables. 

    The law of the $\alpha$ is distributed according to the weights of the \emph{Ruelle cascades} $\kl{v_\a}_{\a\in\N^M}$ with parameters $\kl{\z_k}_{k=1,\dots,M}$, where \begin{align}
		0 &= \zeta_0 < \zeta_1 < \zeta_2 < \dots < \zeta_{M - 1} < \zeta_M < \zeta_{M  + 1} = 1.
	\end{align}
    These weights are constructed in the following way:
	Let
	$u_{(1)} > u_{(2)} > \dots$
	be the ordered atoms of a Poisson point process on $\R_{>0}$ with intensity $\z_1 y^{-1-\z_1}\, \d y$.
	For each $k \in {1,\dots,M-1}$ and each $\g\in\N^k$ independently, 
	$
	u_{(\g,1)} > u_{(\g,2)} > \dots
	$
	are sampled as the ordered atoms of a Poisson point process on $\R_{>0}$ with intensity $\z_{k+1} y^{-1-\z_{k+1}}\, \d y$.
	Recalling the notation $\restr{\a}{k}=(\a_1,\dots,\a_k)$ for $k=1,\dots,M$, for $\a= (\a_1,\dots, \a_M)\in\N^M$, we set
	\begin{align}\label{eq:3DefValphaWalpha}
		w_\a \coloneqq \prod_{k =1}^M u_{\restr{\tilde\a}{k}},
		\quad\text{and}\quad
		v_\a \coloneqq  
		\frac{ 
			w_\a
		}{
			\sum_{\tilde\a \in \N^M} w_\a
		}.
	\end{align} 
	We call $(w_\a)_{\a\in\N^M}$ the \emph{unnormalized weights} of $(v_\a)_{\a\in\N^M}$.
	Since $\sum_{\tilde\a \in \N^M} w_\a$ is finite with probability $1$ (see for example \cite[Lemma~5.23]{JCBook}), $\kl{v_\a}_{\a\in\N^M}$ is well-defined.
	
	Assuming that the processes $\kl{H_N^A(\sigma)}_{\s\in\leaves{N}}$, $\kl{Y_\q(\sigma,\alpha)}_{\s\in\leaves{N}\!,\, \a\in\N^M}$, and $\kl{v_\a}_{\a\in\N^M}$ are independent, the \emph{enriched Hamiltonian} of the CREM with covariance $A$ is defined by 
	\begin{align*}
		H_N(t,\q_M,\s,\a)&\coloneqq\sqrt{2t}\, H_N^A(\sigma) - Nt  + \sqrt{2}\,Y_{\q_M}(\sigma,\alpha) - N q_M.
	\end{align*}
	Also, we define the \emph{enriched free energy} by
	\begin{align}\label{eq:enrichedFE}
		F_N(t,\q_M) 
		&\coloneqq - \frac{1}{N} \EX{\textstyle
			\log\kl{ \sum_{\alpha \in \N^M} v_\a \sum_{\sigma \in \{-1,1\}^N}   
				\exp\bkl{H_N(t,\q_M,\s,\a)}}
		}.
	\end{align}
    \begin{remark}
        This choice of enrichment is motivated by the one introduced to study the SK model (see for example \cite[Chapter~6]{JCBook}) as it encodes the richest possible overlap structure for these models. Such an enrichment requires no assumptions on the properties of the support of the overlaps under the Gibbs measure a priori. Furthermore, this form of enrichment may become necessary when studying the non-convex models.
    \end{remark}

    We now introduce the following path spaces:
	\begin{align}
		Q &\coloneqq \gkl{\q\colon [0,1) \to \R_{\geq 0}; \: \q \text{ is right-continuous and increasing}},\notag\\
		Q_p &\coloneqq Q \cap L_p([0,1), \R), \qquad \forall\:p \in [1, \infty].
	\end{align}
    The following proposition asserts that we can extend the discrete paths $\q_M$ in the domain of $F_N$ to $\R_{\geq 0} \times Q_1$.
	\begin{proposition}\label{prop:LipschitzFn}
		Let $F_N$ be the enriched free energy of the CREM defined in \eqref{eq:enrichedFE}.
		For all $\q_{M_1},\tilde\q_{M_2} \in \bigcup_{M=1}^\infty\qml$ and all $t \geq 0$,
		\begin{align}\label{eq:FNLipschitzClaim}
			\abs{F_N(t,\q_{M_1}) - F_N(t,\tilde\q_{M_2})} \leq \| \q_{M_1} - \tilde\q_{M_2}\|_1.
		\end{align}
		In particular, there is a unique extension of $F_N$ to the space $\R_{\geq 0} \times Q_1$, which satisfies \eqref{eq:FNLipschitzClaim} on $\R_{\geq 0} \times Q_1$.
	\end{proposition}
	Our main result shows that the limiting free energy has a variational formula for almost all covariance functions $A$ in the weak correlation regime.
	\begin{theorem}\label{thm:main}
		Suppose that $A(x)\leq x$ has a finite left-derivative at $1$.
		For all $t>0$ and $\q\in Q_1$,
		\begin{align}
			f(t,\q)
			\coloneqq
			\lim_{N\uparrow \infty} F_N(t,\q) 
			= 
			\sup_{\lambda\in [0,1)}
			\left\{
			\lambda t 
			+ \int_{1-\lambda}^1 \q(u) \d u
			- \frac{\log 2}{1-\lambda}
			\right\}
			. 
			\label{eq:main}
		\end{align}
		In particular, we have
		\begin{align}
			f(t,0) = \begin{cases}
				- \log 2, & \text{if } t\leq \log 2,\\
				t -2 \sqrt{t \log 2}, & \text{if } t> \log 2.
			\end{cases}
			\label{eq:main.q=0}
		\end{align}
	\end{theorem}
	\begin{remark}\label{rem:reparam}
		The sign conventions and normalizations in \eqref{eq:enrichedFE} are adopted to facilitate the analysis in the remainder of the paper.
		To remain consistent with standard conventions in the literature on continuous random energy models (CREM), we note that the inverse temperature is given by \(\beta = \sqrt{2t}\). 
		To recover the usual free energy formula $f_{\mathrm{REM}}(\beta)$ in the weak correlation regime, one should look at $-f(\beta^2/2,0)+\frac{\beta^2}{2}$ which gives
	\begin{align}\label{eq:3FECREMFormulaConvexCase}
			f_{\mathrm{REM}}(\beta)
			=
			-f\left(\frac{\beta^2}{2},0\right)+\frac{\beta^2}{2}
			=
			\begin{cases}
				\log 2 + \frac{\beta^2}{2}, & \text{if } \beta \leq \sqrt{2\log 2},\\
				\sqrt{2 \log 2}\, \beta , & \text{if } \beta > \sqrt{2\log 2}.
			\end{cases}
		\end{align}
	\end{remark}
	\begin{remark}
		The condition that $A$ has a finite left-derivative at $1$ can be removed as long as we can show the formula \eqref{eq:main} for the REM case, that is, when $A(x)=\1_{x=1}$. This holds for $\q = 0$ (see \cite{derrida1981random}). For the case when $\q \neq 0$, we refer the readers to Section~\ref{sec:rem case} for the details.
	\end{remark}
	
	\begin{remark}
		Outside of the weak correlation regime, we provide an example which shows that a naive extension of the variational formula does not recover the free energy. We refer the readers to Section~\ref{sec:outside} for the explicit construction of the example.
	\end{remark}
	
	\subsection{Previous Works} \label{sec:previous works}
	
	The limiting free energy of the CREM was previously proven in \cite{BK1,BK2} by extending the formula for the GREM (see \cite{capocaccia1987existence}). We first state a simplified version of the limiting CREM free energy as it appears in \cite[Appendix~A]{HoCREMGibbs}. 
    For $t\geq 0$, recall that for any covariance function $A$, the limiting free energy is given by 
	\begin{align}\label{eq:3FECREMFormula}
		\lim_{N\uparrow\infty} F_N^{A}(t) 
		&= \lim_{N\uparrow\infty} -\frac{1}{N}\EX{\log \kl{\textstyle\sum_{\s\in\leaves{N}} \exp\brac{\sqrt{2t}\, H_N^A (\s) - N t}}} \notag\\
		&= t \hat A(x(t)) - (1-x(t)) \log 2 - 2 \sqrt{t \log 2} \int_0^{x(t)} \sqrt{\hat A'(x)}\, \d x,
	\end{align}
	where $x(t) = \sup\kl{x\in (0,1)\colon \hat A'(x) \geq \frac{\log 2}{t}}$, $\hat {A}$ is the concave hull of $A$ and $\hat {A}'$ is its right derivative.
    Note that \eqref{eq:3FECREMFormula} is slightly different than what appears in the literature, but the formula is equivalent up to a parameterization of $t$ as in Remark~\ref{rem:reparam}.
    If we are in the weak correlation regime when $A(x) \leq x$ for all $x\in[0,1]$ (and in particular if $A$ is a convex covariance function), then $x(t) = \1_{t\geq\log 2}(t)$, which recovers \eqref{eq:main.q=0}.
	Our goal is to compute a generalization of the CREM free energy, using a new set of techniques motivated by the PDE techniques in \cite{JCBook}, which we now introduce.
	
	In \cite{Mourrat2022parisi}, Mourrat found a link between the free energy of mean field spin glasses and the solution of an infinite-dimensional Hamilton--Jacobi equation. The existence and uniqueness of the solutions to these equations were established in \cite{chenHamiltonJacobiEquations2025, issaUniquenessSemiconcaveWeak2024, issaWeakStrongUniquenessPrinciple2024}. These techniques were further developed in a series of works \cite{Mourrat2021basis,Mourrat2022parisi,mourrat2020extending} to generalize these techniques to broader models, culminating in a formula for the limiting free energy of vector spin models with possibly non-convex interactions \cite{ChenMourrat2023nonconvexVectorSpin}.  New insights from the PDE approach have also led to new variational formulae for the limiting free energies \cite{issaHOPF,JCuninvert}. A textbook introduction for these techniques is available in \cite{JCBook}. One of the motivations for this work is to study the CREM to improve our understanding of models with non-convex interactions.
	
	A key requirement in these works is a condition that the covariance function is \emph{proper} (see, for example \cite[Section~4.4]{chen_hamilton-jacobi_2023} and Lemma~3.6.5 \cite{alban_influence_2025}), which in one dimension is equivalent to the covariance function being convex. In \cite[Proposition~6.6]{Mourrat2023FreeEnergyUpperBound}, it is shown that all such covariance functions corresponding to vector spin models always satisfy this condition, which is not the case for models such as the CREM, which is well-defined even when $A$ is not convex. As such, a first step in this analysis is understanding the CREM with convex covariance functions, i.e., in the weak correlation regime. This will provide a starting point to study the more general variational formulae from which the non-convex models are derived. 
	
	We emphasize that although the CREM was originally introduced as a toy model for spin glasses, it is not an immediate consequence of the variational formulae derived for the mean-field spin glass models on Hilbert spaces \cite{chen_hamilton-jacobi_2023}. For example, Chen and Mourrat proved in  \cite[Theorem~1.1]{ChenMourrat2023nonconvexVectorSpin} that equality as in \eqref{eq:viscosQ0_2} below holds in the setting of $\R^{D\times N}$ with $D, N \in \N$, where the reference measure $P_N$ is a product measure, i.e., $P_N=P_1^{\otimes N}$, where $P_1$ is a probability measure on $\R^D$. However, one can show that the covariance structure in the CREM cannot be embedded into a class of vector spin models with a reference measure that can be written as a product measure \cite[Section~3.7]{alban_influence_2025}, so some work has to be done to establish the variational formula for the limiting free energy. 
	
	Classically, the REM and its variants were introduced as a simplified model of spin glass to study generic phenomena that one expects to see in such models. However, from the point of view of Hamilton--Jacobi techniques, these models introduced several technical difficulties due to the stronger geometric structure of the configuration spaces, which were not seen previously in inference and spin glass applications. Namely, the spins have a hierarchal structure so a simple application of the cavity method fails for such models, since our Hamiltonian depends on $\restr{\tilde\s}{i}$ which includes its ancestors on the binary tree, so we cannot use independence to easily split the Hamiltonian into its cavity fields like in the mixed $p$-spin models (see \cite[Section~1.3]{Pan-book}). We exploited convexity in a crucial way to get around the cavity computations and synchronization \cite[Chapter~6]{JCBook}.  
    
    Furthermore, the initial condition, which was a trivial consequence of independence for spin glasses, becomes much more difficult in this setting. So the computation of the initial conditional, although explicit, does not factorize nicely since it couples two types of spins ($\sigma$ and $\alpha$) with competing geometry. In Theorem~\ref{thm:InitCond}, we compute the limiting initial condition by computing the recursive expected values using truncations. The resulting initial condition is explicit, so the variational formula can be explicitly solved. This allows us to draw several concrete conclusions about the behavior of the CREM, providing independent proof of the universality of the CREM in the weak correlation regime and evidence of the fact that it is 1RSB, which was previously proven in \cite[Theorem~3.6]{BK2}.

	Surprisingly, from the point of view of Hamilton--Jacobi equations, the CREM has some rich structure that was not seen in previous models. Most interestingly, in the one-dimensional case, the covariance function of mean field spin glasses is proper and, in particular, is always convex. On the other hand, in the CREM, there exist covariance functions that violate this condition, which was previously only seen in models with two or more vector spin dimensions, such as the bipartite models. The limiting free energy for the CREM has an explicit formula, and we hope that by studying this model, we have a more tractable tool to verify the validity of \cite{JCnonconvex}. In this paper, we establish the variational formula for the fully enriched free energies in the convex case. We leave the rigorous extension to non-convex covariance functions to follow-up work. We hope that the formulae here can be used as a toy model to understand the more complex behaviors for non-convex mean field models. 
	
	\subsection{Outline of the Paper} \label{sec:HJ}
	
    In this section, we outline the main steps to prove  Theorem~\ref{thm:main} and the structure of the paper. In Section~\ref{sec:properties}, we prove several preliminary facts about the free energy that will be used throughout the rest of the paper. Section~\ref{sec:FnLipschitz} is devoted to the proof of Proposition~\ref{prop:LipschitzFn}, which shows that $F_N(t,\q_M)$ can be extended to continuous $\q \in Q_1$. 
	
	To derive Theorem~\ref{thm:main}, we want to argue that the limiting free energy is the viscosity solution to a Hamilton--Jacobi equation under the framework of \cite{chen_hamilton-jacobi_2023}. In particular, the variational formula in \eqref{eq:main} is derived from the variational form of solutions to the following infinite-dimensional Hamilton-Jacobi equation:
	\begin{align}\tag{$HJE[\q]$}\label{eq:HJX}
		\begin{cases}
			\displaystyle
			\partial_t f(t,\q) 
			- \int_0^1
			A\brac{
				\brac{\nabla_\q f(t,\q)} (u)
			} \, 
			\d u 
			= 0, 
			& \forall\, (t,\q)\in\R_{+}\times Q_2, \\
			f(0,\q) = \Psi(\q), & \q\in Q_2,
		\end{cases}
	\end{align}
where $\nabla_\q$ is the Fr\'echet derivative on $L^2([0,1] ; \R)$. A precise definitions of all the terms appearing in the above can be found in \cite{chen_hamilton-jacobi_2023}. It can be interpreted as the limit of finite-dimensional approximations of this equation, which corresponds to the cases when $\q$ is a step function. In fact, in this paper, we only consider the finite-dimensional versions of \eqref{eq:HJX} and use \cite[Theorem~4.7]{chenHamiltonJacobiEquations2025} to recover a variational formula corresponding to the unique solution of \eqref{eq:HJX}. 

The solution to equation \eqref{eq:HJX}, which is often expressed as a variational formula,  gives us the limit of the free energy. That is, the unique viscosity solution $f$ of \eqref{eq:HJX} with the initial condition $f(0,\cdot) = \Psi$ satisfies
	\begin{align}\label{eq:viscosQ0_2}
		f(t,\q) &= \lim_{N\uparrow\infty} F_N^A(t,\q),
	\end{align}
	for all $t\geq 0$ and $\q\in Q_2$. Thanks to the tractability of our model, we can further calculate the explicit formula of $\Psi$, which is usually impossible in other mean-field spin glasses. 
	
	\begin{theorem}\label{thm:InitCond}
		Let $F_N\colon \R_{\geq 0}\times Q_1 \to \R$ be as in \eqref{eq:enrichedFE}. For each $\q\in Q_1$, we have
		\begin{align}\label{eq:InitialConditionClaim2}
			\Psi(\q) \coloneqq \lim_{N\uparrow \infty} F_N(0,\q) = 
			-\log 2 + \int_0^1 \brac{\q(u) - \sfrac{\log 2}{u^2}}_+ \d u. 
		\end{align}
	\end{theorem}
	The proof of Theorem~\ref{thm:InitCond} is done in Section~\ref{sec:initalCondition}. First, we show that, via an iterative formula (Proposition~\ref{prop:RecursiveCompValpha}), one can rewrite the integration over the Ruelle cascade as a recursive conditional expectation of a partition function involving $M$ copies of i.i.d. branching random walks. The second step is to compute this recursive conditional expectation by determining the typical behavior of the partition function.

	The proof of Theorem~\ref{thm:main} is done in Section~\ref{sec:finitedim} and consists of the following steps. First, we argue that the limiting free energy is a supersolution of \eqref{eq:HJX} by following Mourrat’s argument. This gives a lower bound on the free energy in terms of a variational formula. Moreover, we show that this variational formula is independent of $A$ as long as $A(x)\leq x$. For the upper bound, we apply a Gaussian comparison argument, which implies that $F_N^A\leq F_N^{\mathrm id}$. Finally, we show that the free energy $F_N^{\mathrm id}$ is upper bounded by the solution of a transport equation, and its solution matches the lower bounding variational formula.

	Moreover, Theorem~\ref{thm:main} allows us to write the limiting free energy as a variation formula which depends only on the limit of the initial condition $\Psi(\q)=\lim_{N\uparrow \infty}F_N(0,\q)$. Due to the explicit form of the initial condition in  Theorem~\ref{thm:InitCond}, we are able to solve the variational formula explicitly giving the simple one parameter variational formula in \eqref{eq:main}. At $\q = 0$, we recover the free energy formula for the CREM as expected.
	
	Lastly, in Section~\ref{sec:outlook}, we discuss a possible generalization of Theorem~\ref{thm:main} to non-convex covariance functions. In particular, we show that a naive extension of the variational formula to non-convex covariance functions gives an incorrect limit except in the high-temperature phase. 
	
	\section{Properties of the Free Energy} \label{sec:properties}
	
	In this section, we prove several properties of the free energy.
	
	\subsection{Recursive Formula} \label{sec:recursive}
	
	We begin by stating a recursive formula for $F_N$ in terms of the branching random walks. This is an essential step in computing the initial condition. Before proceeding, we define this object as follows.
	
	\begin{definition}\label{def:BRWGaussian}
		The \emph{branching random walk} (BRW) on the $N$-level binary tree with standard Gaussian increments is a centered Gaussian process $(z(\s))_{\s\in\cbrac{-1,1}^N}$ with covariances 
		\begin{align}
			\EX{z(\s)\, z(\tilde\s)} = \s\wedge\tilde\s, \qquad \forall\: \s,\tilde\s\in\leaves{N}
		\end{align}
		recalling the definition of the overlap $\s\wedge\tilde\s$ in \eqref{eq:3DefOverlap}.
	\end{definition}
    
	Note that $(z(\s))_{\s\in\leaves{N}}$ is a CREM whose covariance function $A$ is the identity function, i.e., it satisfies $A(x) = x$ for all $x\in[0,1]$. We have
	\begin{align}\label{eq:zBRWRewrite}
		(z(\s))_{\s\in\leaves{N}} 
		\overset{\d}{=} 
		\kl{\textstyle\sum_{i=1}^N z_{\restr{\s}{i}}}_{\s\in\leaves{N}},
	\end{align}
	where each $z_{\restr{\s}{i}}$ is from a family of \iid standard Gaussian random variables.

    We first recall an iterative formula that allows us to compute averages with respect to the Ruelle cascades recursively. 
    \begin{proposition}[See e.g.\ {\cite[Theorem~5.25]{JCBook}}]
    \label{prop:RecursiveCompValpha}
		Let $M\in \N$ and  let $(v_\a)_{\a\in\N^M}$ be the Ruelle cascades with parameters \begin{align}
			0 &= \zeta_0 < \zeta_1 < \zeta_2 < \dots < \zeta_{M - 1} < \zeta_M < \zeta_{M  + 1} = 1.
		\end{align}
		Independent of $(v_\a)_{\a\in\N^M}$, let 
        \begin{equation}\label{eq:unif}
        \big(\o_\g : \g\in\{\varnothing\}\cup \textstyle\bigcup_{j=1}^M \N^j
        \big)
        \quad \text{and} \quad
        (\tilde\o_j)_{j=0,\dots,M}
        \end{equation}
        be two independent families of \iid uniform random variables on $[0,1]$.
		Let $X_M\colon [0,1]^{M+1}\to\R$ be a measurable function.
		We set
		\begin{align}
			X_k 
            \coloneqq 
            X_k(\tilde\o_0,\dots,\tilde\o_k) 
			\coloneqq 
            \sfrac{1}{\z_{k+1}} 
            \log \cEX{\exp\bkl{\z_{k+1} X_{k+1}(\tilde\o_0,\dots,\tilde\o_{k+1})}}{\tilde\o_0,\dots,\tilde\o_{k}},
		\end{align}
		recursively for $k=M-1, M-2, \dots, 0$. 
		Then,
		\begin{align}\label{eq:IRuelleRecursive}
			\EX{
				\ln\kl{
					\textstyle\sum_{\a\in\N^M} v_\a
					\exp\bbkl{
						X_M(\o_{\varnothing},\o_{\restr{a}{1}}, \dots, \o_{\restr{a}{M-1}}, \o_\a)
					}
				}
			}
			=
			\EX{X_0(\tilde\o_0)}.
		\end{align}
	\end{proposition}
    \begin{remark}
        One can view $\{\varnothing\}\cup \textstyle\bigcup_{j=1}^M \N^j$ as the vertices of an infinite tree $\T^M$. Then, the random variables introduced \eqref{eq:unif} encode the randomness on the vertices of $\T^M$ and levels of $\T^M$, respectively.
    \end{remark}
	Note that the right-hand side of \eqref{eq:IRuelleRecursive} can be expanded to 
	\begin{align}
		\EX{X_0(\tilde\o_0)}
		=
		\E\Bigg[
		\log\Bigg(
		\E\bigg[
		\E\bigg[ 
		\dots
		%\E\Big[
		\E\big[
		\exp\kl{\z_M X_M(\tilde\o_0,\dots, \tilde\o_M)}
		\big\vert\,
		\FF_{M-1}
		\big]^{\frac{\z_{M-1}}{\z_M}}
		\dots
		\bigg\vert\,\FF_1
		\bigg]^{\frac{\z_1}{\z_2}}
		\bigg\vert\, \FF_0
		\bigg]^{\frac{1}{\z_1}}
		\Bigg)
		\Bigg],
	\end{align}
	where we write $\FF_k = \sigma(\tilde\o_0,\tilde\o_1,\dots,\tilde\o_k)$ for k = $0,\dots, M-1$.
    
	Now, we apply Proposition~\ref{prop:RecursiveCompValpha} to the enriched free energy of the CREM. 
	\begin{lemma}\label{lem:RecursiveAveragingFreeEnergy}
		Let $M,N\in \N$, $t\geq 0$ and  $\q = \sum_{k=0}^M q_k \1_{[\z_k, \z_{k+1})} \in \qm$.
		Let $F_N$ be as in \eqref{eq:enrichedFE}.
		Then,
		\begin{align}\label{eq:enrichedFERuelleAverage}
			&F_N(t,\q)
			\nonumber \\
			&= t + q_M -\frac{1}{N}\E\Bigg[
			\log\Bigg(
			\E\bigg[
			\E\bigg[ 
			\dots
			\E\Big[
			Z_{t,q}(z_0,\dots, z_M)^{\z_M}
			\Big\vert\,
			\FF_{M-1}
			\Big]^{\frac{\z_{M-1}}{\z_M}}
			\dots
			\bigg\vert\,\FF_1
			\bigg]^{\frac{\z_1}{\z_2}}
			\bigg\vert\, \FF_0
			\bigg]^{\frac{1}{\z_1}}
			\Bigg)
			\Bigg],
		\end{align}
		where 
		\begin{align}
			Z_{t,q}(z_0,\dots, z_M) 
			&\coloneqq \sum_{\s\in\leaves{N}} 
			\exp\brac{
				\sqrt{2t}\, H_N^A(\s) +  \sqrt{2}\, \textstyle\sum_{j=0}^M (q_j - q_{j-1})^{\frac{1}{2}} z_j(\sigma)
			},  \label{eq:InitCondNotation} \\
			\FF_k 
			&\coloneqq \s(z_0,\dots,z_k), \quad k=0,\dots, M-1.
			\label{eq:FF_k}
		\end{align}
		Furthermore, independent of $\kl{H_N^A(\sigma)}_{\s\in\leaves{N}}$, the processes $z_1 = (z_1(\s))_{\s\in\cbrac{-1,1}^N}, \dots, z_M = (z_M(\s))_{\s\in\cbrac{-1,1}^N}$ are i.i.d.\ copies of $z_0 \coloneqq (z_0(\s))_{\s\in\cbrac{-1,1}^N}$, a BRW defined as in Definition~\ref{def:BRWGaussian}. 
	\end{lemma}
	\begin{proof}
		By the tower property,
		\begin{align}\label{eq:RecAv1}
			& F_N(t,\q) \notag\\
			&=	-\frac{1}{N}\EX{\log\kl{\textstyle\sum_{\alpha \in \N^M} v_\a \sum_{\sigma \in \{-1,1\}^N}   
					\exp\bbkl{H_N\bkl{t,\q,\s,\a}}}}\notag\\
			&=
			-\frac{1}{N}\E\Big[
				\E\Big[
					\log\kl{\textstyle\sum_{\alpha \in \N^M} v_\a \sum_{\sigma \in \{-1,1\}^N}   
						\exp\bbkl{ 
							\sqrt{2t}\, H_N^A(\sigma) - Nt  + \sqrt{2}\,Y_\q(\sigma,\alpha) - N q_M
						}
					}
                    \nonumber \\
                    &\qquad\qquad\qquad\qquad\qquad\qquad\qquad\qquad\qquad\qquad\qquad\qquad\qquad\qquad\qquad
				\Big\vert
					\kl{H_N^A(\sigma)}_{\s\in\leaves{N}}
				\Big]
			\Big]\notag\\
			&=
			t+q_M \nonumber \\
                &
			-\frac{1}{N}
			\E\Big[
				\E\Big[
					\log\kl{\textstyle\sum_{\alpha \in \N^M} v_\a \sum_{\sigma \in \{-1,1\}^N}   
						\exp\bbkl{
							\sqrt{2t}\, H_N^A(\sigma) + \sqrt{2}\,Y_\q(\sigma,\alpha) 
						}
				} 
                \Big\vert
			\kl{H_N^A(\sigma)}_{\s\in\leaves{N}}
				\Big]
			\Big].
		\end{align}
		To apply Proposition~\ref{prop:RecursiveCompValpha} to the last conditional expectation in \eqref{eq:RecAv1}, we construct a measurable function $G_M\colon [0,1]^{M+1} \to \R$ so that conditionally on $\kl{H_N^A(\sigma)}_{\s\in\leaves{N}}$,
		\begin{align}\label{eq:RecAv1.2}
			&\bbkl{
				\exp\kl{G_M(\o_{\varnothing},\o_{\restr{\a}{1}}, \dots, \o_{\restr{\a}{M-1}}, \o_\a)}
				\!}_{\a\in\N^M}
                \nonumber \\
			&
            \qquad\qquad\qquad
            \overset{\d}{=}
			\kl{ \textstyle\sum_{\sigma \in \{-1,1\}^N}   
				\exp\bbkl{
					\sqrt{2t}\, H_N^A(\sigma) + \  \sqrt{2}\,Y_\q(\sigma,\alpha) 
				}
			}_{\a\in\N^M},
		\end{align}
		where $\kl{\o_{\restr{\a}{k}}}_{\a\in\N^M, k=0,\dots,M}$ is a family of \iid uniform distributions on $[0,1]$. Throughout the rest of this proof, the notation $\overset{\d}{=}$ is understood to mean equal in distribution conditionally on $\kl{H_N^A(\sigma)}_{\s\in\leaves{N}}$.
        
        We follow the same construction as in Step~1 of the proof of Proposition~6.3 in \cite{JCBook} which requires the following steps:
		Let $\o$ be a uniformly distributed random variable on $[0,1]$.
		\begin{enumerate}
			\item  There exists a measurable bijection $\eta_N\colon [0,1] \to [0,1]^{2^N}$ so that $\eta_N(\o)$ is uniformly distributed on $[0,1]^{2^N}$. One can construct $\eta_N$, e.g., by splitting the binary representation of $x\in[0,1]$ into $2^N$ parts. 
			\item Let $\Phi$ be the cumulative distribution function of the standard Gaussian distribution. Let $\Phi_N^{-1}\colon [0,1]^{2^N} \to \R^{2^N}$ be the map which applies $\Phi^{-1}$ to each entry of an element of $[0,1]^{2^N}$. Then, $\Phi_N^{-1}\circ\eta_N(\o)$ is a $2^N$-dimensional standard Gaussian vector.
			\item \label{it:RecAvStep3}
			Let $\kl{z(\s)}_{\s\in\leaves{N}}$ be a BRW on the $N$-level binary tree with standard Gaussian increments as defined in \eqref{def:BRWGaussian}. 
			Let $\varphi\colon \{1,\ldots,2^N\} \to \leaves{N}$ be any bijection. 
			Let $\mathsf{S} = \kl{\varphi\kl{i} \wedge \varphi\kl{j}}_{i,j=1,\dots, 2^N}$ 
			be the covariance matrix of the Gaussian process $\kl{z(\varphi(i))}_{i=1,\dots, 2^N}$.
			Thus, there exists $\mathsf{R}\in\R^{2^N\times 2^N}$  so that  $\mathsf{R} \mathsf{R}^T = \mathsf{S}$. We also denote the corresponding linear map on $\R^{2^N}$ by $\mathsf{R}$. Then, $\mathsf{R}\circ\Phi_N^{-1}\circ\eta_N(\o)\overset{\d}{=} \kl{z(\varphi(i))}_{i=1,\dots, 2^N}$.
			\item We set \vspace{-10pt}
			\begin{align}
				\tilde G_M\colon [0,1]^{M+1} &\to \R^{2^N},\notag\\
				(u_0,\dots, u_M) &\mapsto \sum_{k=0}^M (q_k - q_{k-1})^{1/2} \mathsf{R}\circ\Phi_N^{-1}\circ\eta_N(u_k).
			\end{align}
			Then, for a family of \iid uniform distributions $\kl{\o_{\restr{\a}{k}}}_{\a\in\N^M, k=0,\dots,M}$  on $[0,1]$, we have 
			\begin{align}\label{eq:RecAv2}
				\kl{
					\tilde G_M(\o_{\varnothing},\o_{\restr{\a}{1}}, \dots, \o_{\restr{\a}{M-1}}, \o_\a)
				}_{\a\in\N^M}  
				&=
				\kl{
					\textstyle\sum_{k=0}^M (q_k - q_{k-1})^{1/2} \mathsf{R}\circ\Phi_N^{-1}\circ\eta_N(\o_{\restr{\a}{k}})
				}_{\a\in\N^M} \notag\\
				&\overset{\d}{=} 
				\kl{
					\textstyle\sum_{k=0}^M (q_k - q_{k-1})^{1/2} \kl{z_{\restr{\a}{k}}(\varphi(i))}_{i=1,\dots, 2^N}
				}_{\a\in\N^M}, 
			\end{align}
			where each $\bkl{z_{\restr{\a}{k}}(\s)}_{\s\in\leaves{N}}$ is from a family of \iid copies of the BRW $\kl{z(\s)}_{\s\in\leaves{N}}$ from Step~\ref{it:RecAvStep3}.
			By \eqref{eq:3YwrtZ} and \eqref{eq:zBRWRewrite}, recalling that each $z_{\restr{\s}{i},\restr{\a}{k}}$ denotes an element of a family of \iid standard Gaussian random variables,
			\begin{align}\label{eq:RecAv3}
				\kl{Y_\q(\sigma,\alpha)}_{\s\in\leaves{N}\!,\, \a\in\N^M} 
				&\overset{\d}{=} \kl{\textstyle\sum_{i = 1}^N \sum_{k = 0}^M (q_k - q_{k - 1} )^{1/2} z_{\restr{\s}{i},\restr{\a}{k}}}_{\s\in\leaves{N}\!,\, \a\in\N^M}\notag\\
				&\overset{\d}{=} \kl{\textstyle \sum_{k = 0}^M (q_k - q_{k - 1} )^{1/2} z_{\restr{\a}{k}}(\s)}_{\s\in\leaves{N}\!,\, \a\in\N^M}\notag\\
				&\overset{\d}{=} \kl{\pi_{\varphi^{-1}(\s)}\circ\tilde G_M(\o_{\varnothing},\o_{\restr{\a}{1}}, \dots, \o_{\restr{\a}{M-1}}, \o_\a)}_{\s\in\leaves{N}\!,\, \a\in\N^M}, 
			\end{align}
			where $\pi_i\colon \R^{2^N} \to \R$, $i=1,\dots,2^N$, is the projection to the $i$-th entry.
			\item We set
			\begin{align}\label{eq:RecAv3.5}
				G_M\colon [0,1]^{M+1} &\to \R,\notag\\
				u  &\mapsto \log \kl{\textstyle\sum_{i=1}^{2^N} \exp\bbkl{
						\sqrt{2t}\, H_N^A(\varphi(i)) + \sqrt{2}\,\pi_i\circ \tilde{G}_M(u)  
				} }.
			\end{align}
			By \eqref{eq:RecAv3}, $G_M$ satisfies \eqref{eq:RecAv1.2}.
		\end{enumerate}
		
		We apply Proposition~\ref{prop:RecursiveCompValpha} to the conditional expectation on the right-hand side of \eqref{eq:RecAv1} to get that
		\begin{align}\label{eq:RecAv4}
			&\cEX{
					\log\kl{\textstyle\sum_{\alpha \in \N^M} v_\a \sum_{\sigma \in \{-1,1\}^N}   
						\exp\bbkl{
							\sqrt{2t}\, H_N^A(\sigma) + \sqrt{2}\,Y_\q(\sigma,\alpha) 
						}
				}}{
					\kl{H_N^A(\sigma)}_{\s\in\leaves{N}}
				}
			\notag\\
			&=
			\cEX{
					\log\kl{\textstyle\sum_{\alpha \in \N^M} v_\a \exp\kl{G_M(\o_{\varnothing},\o_{\restr{\a}{1}}, \dots, \o_{\restr{\a}{M-1}}, \o_\a)}
				}}{
					\kl{H_N^A(\sigma)}_{\s\in\leaves{N}}
				}
			\notag\\
			&=
			\E\bigg[
			\log\bigg(
			\E\bigg[
			\E\bigg[ 
			\dots
			\E\bigg[
			\E\Big[
			\exp\kl{\z_M G_M(\tilde\o_0,\dots, \tilde\o_M)}
			\Big\vert\,
			\tilde\FF_{M-1}
			\Big]^{\frac{\z_{M-1}}{\z_M}}
			\,\bigg\vert\,
			\tilde\FF_{M-2}
			\bigg]^{\frac{\z_{M-2}}{\z_{M-1}}}
			\dots
			\notag\\
			&\qquad\qquad\qquad\qquad\qquad\dots
			\bigg\vert\,\tilde\FF_1
			\bigg]^{\frac{\z_1}{\z_2}}
			\bigg\vert\, \tilde\FF_0
			\bigg]^{\frac{1}{\z_1}}
			\bigg)
			\bigg\vert\,
			\kl{H_N^A(\sigma)}_{\s\in\leaves{N}}
			\bigg],
		\end{align}
		where $\tilde\o_0,\dots, \tilde\o_M$ are \iid uniformly distributed random variables on $[0,1]$ and we write for $k=0,\dots,M-1$, $\tilde \FF_k = \s(\tilde\o_0,\dots,\tilde\o_k)$.
		By \eqref{eq:RecAv2} and \eqref{eq:RecAv3.5}, 
		\begin{align}\label{eq:RecAv4.5}
			\exp\kl{\z_M G_M(\tilde\o_0,\dots, \tilde\o_M)} \overset{\d}{=}
			Z_{t,q}(z_0,\dots, z_M)^{\z_M},
		\end{align}
		recalling the notation in \eqref{eq:InitCondNotation}.
        
        Inserting \eqref{eq:RecAv4.5} into \eqref{eq:RecAv4} gives
		\begin{align}\label{eq:RecAv5}
			&\cEX{
					\log\kl{\textstyle\sum_{\alpha \in \N^M} v_\a \sum_{\sigma \in \{-1,1\}^N}   
						\exp\bbkl{
							\sqrt{2t}\, H_N^A(\sigma) + \sqrt{2}\,Y_\q(\sigma,\alpha) 
						}
				}}{
					\kl{H_N^A(\sigma)}_{\s\in\leaves{N}}
				}\notag\\
			&=
            \E\bigg[
			\log\Bigg(
			\E\bigg[
			\E\bigg[ 
			\dots
			\E\bigg[
			\E\Big[
			Z_{t,q}(z_0,\dots, z_M)^{\z_M}
			\Big\vert\,
			\FF_{M-1}
			\Big]^{\frac{\z_{M-1}}{\z_M}}
			\,\bigg\vert\,
			\FF_{M-2}
			\bigg]^{\frac{\z_{M-2}}{\z_{M-1}}}
            \nonumber \\
            &
            \qquad\qquad\qquad\qquad\qquad\qquad\qquad\qquad\qquad\qquad
			\dots
			\bigg\vert\,\FF_1
			\bigg]^{\frac{\z_1}{\z_2}}
			\bigg\vert\, \FF_0
			\bigg]^{\frac{1}{\z_1}}
			\Bigg)
            \bigg|
            \kl{H_N^A(\sigma)}_{\s\in\leaves{N}}
            \bigg],
		\end{align}
		recalling that $\FF_k = \s(z_0,\dots,z_k)$ for $k=0,\dots, M-1$.
		Thus, inserting \eqref{eq:RecAv5} into \eqref{eq:RecAv1} and applying the tower property completes the proof.
	\end{proof}
	
	% ----------------------------------------------------------------------------------
	\subsection{Derivatives of the Free Energy}\label{sec:derivative}
	% ---------------------------------------------------------------------------------- 
	
	Throughout this section, we fix $M,N\in\N$ and 
	$0 = \zeta_0 < \zeta_1 < \zeta_2 < \dots < \zeta_{M - 1} < \zeta_M < \zeta_{M  + 1} = 1$. We define the convex cone 
	\begin{align}\label{eq:DefConeRm}
		C_\leq^{(M)} \coloneqq \cbrac{q=(q_0,\dots,q_M)\in \R^{M+1} \colon 0 \leq q_0 \leq \dots \leq q_M}.
	\end{align}
	Its interior is
	\begin{align}\label{eq:DefConeInteriorRm}
		C_<^{(M)} \coloneqq \cbrac{q=(q_0,\dots,q_M)\in \R^{M+1} \colon 0 < q_0 < \dots < q_M}.
	\end{align}
	Then we can regard 
    \begin{align}
        F_N(t,q)=F_N(t,\q_M)
        \quad \text{where} \quad
        \q_M = \sum_{k=0}^M q_k \1_{[\z_k, \z_{k+1})}
        \label{eq:F_N(t,q)}
    \end{align}
    as a function defined on $\R_{\geq 0}\times C_\leq^{(M)}$.
	The goal of this section is to record several useful facts about the derivatives of the free energy $F_N(t,q)$.
	
	Denote by $\mathcal{S}_{N,M} =\leaves{N}\times \N^M$ the state space of our coupled spins $\s$ and $\alpha$. For $(t,q)\in\R_{\geq 0}\times C_<^{(M)}$, denote by
	\begin{align}
		\mu_{t,q}(\hat \s, \hat \a) = 
		\frac{
			v_{\hat \a} \exp\bkl{H_N(t,q,\hat\s,\hat\a)}
		}{
			\sum_{\a\in\N^M} v_\a \sum_{\s\in\leaves{N}}  \exp\bkl{H_N(t,q,\s,\a)}
		} 
	\end{align}
	the weight of the (random) Gibbs measure $\mu_{t,q}$ at $(\hat \s, \hat \a) \in \mathcal{S}_{N,M}$, where $H_N(t,q,\hat\s,\hat\a) = H_N(t,\q_M,\hat\s,\hat\a)$ with $\q_M = \sum_{k=0}^M q_k \1_{[\z_k, \z_{k+1})}$.
	Note that replacing $\kl{v_\a}_{\a\in\N^M}$ by the unnormalized weights $\kl{w_\a}_{\a\in\N^M}$ does not change \eqref{eq:GibbsAvgSigmaAlpha}. 
	For any $H_N$ and $v_\alpha$, the $1$-fold Gibbs average on the spin space $\mathcal{S}_{N,M}$ \wrt $H_N(t,q,\cdot,\cdot)$ and $\kl{v_\a}_{\a\in\N^M}$ of a bounded and measurable function $g\colon \mathcal{S}_{N,M}\to\R$ is denoted by
	\begin{align}\label{eq:GibbsAvgSigmaAlpha}
		\ga{g(\s,\a)}_{t,q} 
		\coloneqq
		\sum_{\a\in\N^M} \sum_{\s\in\leaves{N}} g(\s,\a)\mu_{t,q}(\s, \a). 
	\end{align}
	For $n\geq 2$ and $g\colon \mathcal{S}_{N,M}^n \to\R$ bounded and measurable, the $n$-fold Gibbs average 
	\begin{align}\label{eq:GibbsAvgSigmaAlphan}
		\ga{g(\s^{(1)},\a^{(1)},\dots, \s^{(n)},\a^{(n)})}_{t,q,n}
	\end{align}
	is defined as the average of~$g$ w.r.t.~$\bkl{\mu_{t,q}}^{\otimes n}$. The subscripts $t$, $q$ and $n$ will be omitted when the dependency is clear in the context.
	
	The first lemma shows that the partial derivatives of $F_N(t,q)$ with respect to the $q$'s are non-negative, and they are monotone.
	\begin{lemma}\label{lem:increasing}
		Let $N,M \geq 1$ and 
		$0 = \zeta_0 < \zeta_1 < \zeta_2 < \dots < \zeta_{M - 1} < \zeta_M < \zeta_{M  + 1} = 1$.
		For all $(t,q)\in\R_{\geq 0}\times C_<^{(M)}$ and $\ell,\ell'\in\{0,\ldots,M\}$ with $\ell \leq \ell'$, we have
		\begin{equation}\label{eq:dF>=0}
			\partial_{q_\ell} F_N(t,q) \geq 0
		\end{equation}
		and
		\begin{equation}\label{eq:dFinC}
			(\zeta_{\ell + 1} - \zeta_\ell)^{-1} \partial_{q_\ell} F_N(t,q) \leq 	(\zeta_{\ell' + 1} - \zeta_{\ell'})^{-1} \partial_{q_{\ell'}}  F_N(t,q).
		\end{equation}
		%where $F_N(t,q) = F_N(t,\q_M)$ with
		%$\q_M = \sum_{k=0}^M q_k \1_{[\z_k, \z_{k+1})}$.
	\end{lemma}
	
	\begin{proof}
		The proof follows the same strategy as Lemma 2.4 in \cite{JCnonconvex}. The key difference is that the enriched term has a slightly different form, which makes the integration by parts more complicated. We first define an embedding of $\sigma \in \{ -1,1 \}^N$, into a larger space which will allow us to write the enriched term in a form that resembles \cite[Lemma~2.4]{JCnonconvex}.
		
		We identify the first $2^{N+1}-2$ unit vectors $e_1,\dots,e_{2^{n+1}-2} \in \R^{2^{N+1}}$ by $(e_z)_{z\in \bigcup_{i=1}^N \leaves{i}}$ in an arbitrary but fixed order. The map
		\begin{align}
			\iota_N\colon \leaves{N}&\to \HH_N = \R^{2^{N+1}},\notag\\
			\sigma \mapsto \sum_{i=1}^N e_{\restr{\s}{i}},
		\end{align}
		defines an embedding of $\leaves{N}$ in $\R^{2^{N+1}}$ so that for $\s,\tilde\s\in\leaves{N}$,
		\begin{align}\label{eq:generalFEEmbedding}
			\iota_N(\s)\cdot\iota_N(\tilde\s) = \sum_{i,j=1}^N e_{\restr{\s}{i}} \cdot e_{\restr{\tilde\s}{j}}
			= \sum_{i,j=1}^N \1_{\restr{\s}{i} = \restr{\tilde\s}{j}}(\s,\tilde\s)
			= \sum_{i=1}^N \1_{\restr{\s}{i} = \restr{\tilde\s}{i}}(\s,\tilde\s)
			= \s\wedge \tilde\s.
		\end{align}
		Let $z_0,\ldots,z_M$ be a sequence of i.i.d. $2^{N + 1}$-dimensional standard Gaussian vectors. Then, we have that the partition function defined in \eqref{eq:InitCondNotation} can be equivalently written as 
		\begin{align}
			Z_{t,q}(z_0,\dots, z_M) 
			&= \sum_{\s\in\leaves{N}} 
			\exp\brac{
				\sqrt{2t}\, H_N^A(\s) +  \sqrt{2}\, \textstyle\sum_{j=0}^M (q_j - q_{j-1})^{\frac{1}{2}} z_j\cdot \iota_N(\sigma)}
		\end{align}
		and
		\[
		X_M = \log Z_{t,q} (z_0,\ldots,z_M).
		\]	
		For $\ell \leq m \leq M$, we write
		\[
		D_{\ell,m} := \frac{ e^{ \zeta_\ell X_\ell + \dots + \zeta_M X_m } }{ \cEX{e^{\zeta_\ell X_\ell}}{\cF_{\ell-1}} \cdots \cEX{e^{\zeta_\ell X_m}}{\cF_{m-1}} }.
		\]
		We claim that 
		\begin{align}
			\partial_{q_\ell} F_N
			=
			\frac{\zeta_{\ell+1}-\zeta_{\ell}}{N}
			\sum_{i=1}^{2^{N + 1}}
			\EX{
				\cEX{
					\bga{
						\iota_{N,i}(\sigma)
					}
					D_{\ell+1,M}
				}{\cF_\ell}^2
				D_{1,\ell}
			}
			\label{eq:quadratic}
		\end{align}
		where $ \iota_{N,i}(\sigma)$ denotes the $i$th entry of $\iota_N(\sigma) \in \R^{2^N +1}$. 
		Suppose \eqref{eq:quadratic} holds. Then we adopt the argument that appears in Step 2 of the proof of Proposition 3.6 in \cite{Mourrat2023FreeEnergyUpperBound}.
		Since $D_{1,\ell}  \geq 0$, \eqref{eq:quadratic} implies \eqref{eq:dF>=0}. To prove \eqref{eq:dFinC}, it suffices to show that for $i\in\{1,\ldots,N\}$ and $\ell \leq \ell'$, 
		\[
		\EX{
			\cEX{
				\bga{
					\iota_{N,i}(\sigma)
				}
				D_{\ell+1,M}
			}{\cF_{\ell}}^2
			D_{1,\ell}
		} 
		\leq  
		\EX{
			\cEX{
				\bga{
					\iota_{N,i}(\sigma)
				}
				D_{\ell'+1,M}
			}{\cF_{\ell'}}^2
			D_{1,\ell'}
		}.
		\]
		By the tower property and applying Jensen's inequality to $\cEX{(\,\cdot\,)D_{\ell+1,\ell'}}{\cF_\ell}$,
		\begin{align*}
			\cEX{
				\bga{
					\iota_{N,i}(\sigma)
				}
				D_{\ell+1,M}
			}{\cF_{\ell}}^2
			&=
			\cEX{
				\cEX{
					\bga{
						\iota_{N,i}(\sigma)
					}
					D_{\ell'+1,M}
				}
				{\cF_{\ell'}}
				D_{\ell+1,\ell'}
			}{\cF_{\ell}}^2
			\nonumber \\
			&\leq 
			\cEX{
				\cEX{
					\bga{
						\iota_{N,i}(\sigma)
					}
					D_{\ell'+1,M}
				}
				{\cF_{\ell'}}^2
				D_{\ell+1,\ell'}
			}{\cF_{\ell}},
		\end{align*}
		so
		\begin{align*}
			\EX{
				\cEX{
					\bga{
						\iota_{N,i}(\sigma)
					}
					D_{\ell+1,M}
				}{\cF_{\ell}}^2
				D_{1,\ell}
			} 
			&\leq 
			\EX{
				\cEX{
					\cEX{
						\bga{
							\iota_{N,i}(\sigma)
						}
						D_{\ell'+1,M}
					}
					{\cF_{\ell'}}^2
					D_{\ell+1,\ell'}
				}{\cF_{\ell}}
				D_{1,\ell}
			}
			\nonumber \\
			&=
			\EX{
				\cEX{
					\bga{
						\iota_{N,i}(\sigma)
					}
					D_{\ell'+1,M}
				}
				{\cF_{\ell'}}^2
				D_{1,\ell'}
			}
		\end{align*}
		which proves the desired inequality. 
		
		It remains to prove \eqref{eq:quadratic}. 
		The proof mostly follows the proof of Lemma 2.2 in \cite{Mourrat2022parisi}, except that the inner product between the $N$-dimensional Gaussian vector and the spin needs to be replaced due to the covariance structure of our model.
		For $\ell,m\in\{0,\ldots,M\}$, by a decreasing induction on $\ell$, we can show that
		\begin{align}
			\partial_{q_m} X_{\ell - 1} = \cEX{( \partial_{q_m} X_M ) D_{\ell, M}}{\cF_{\ell-1}}.
			\label{eq:dqX.0}
		\end{align}
		For any $m\in\{0,\ldots,M\}$, by differentiation, we can show that
		\begin{align}
			\partial_{q_m} X_{M} &= \bga{ (2 q_m - 2 q_{m-1})^{-\frac{1}{2}} z_m\cdot\iota_N(\sigma)  -  (2 q_{m+1} - 2 q_{m})^{-\frac{1}{2}} z_{m+1}\cdot\iota_N(\sigma) } \notag
			\\&= \sum_{i = 1}^{2^{N + 1}} 
			(2 q_m - 2 q_{m-1})^{-\frac{1}{2}}
			z_{m,i}
			\bga{
				\iota_{N,i}(\sigma)
			}  -  
			(2 q_{m+1} - 2 q_{m})^{-\frac{1}{2}} 
			z_{m,i}
			\bga{
				\iota_{N,i + 1}(\sigma)}.
			\label{eq:inc1}
		\end{align}
		Plugging \eqref{eq:inc1} back to \eqref{eq:dqX.0} yields
		\begin{align}
			\partial_{q_m} X_{\ell - 1}
			&= 
			\bga{ (2 q_m - 2 q_{m-1})^{-\frac{1}{2}} z_m\cdot\iota_N(\sigma)  -  (2 q_{m+1} - 2 q_{m})^{-\frac{1}{2}} z_{m+1}\cdot\iota_N(\sigma) } \notag
			\\&= \sum_{i = 1}^{2^{N + 1}} 
			(2 q_m - 2 q_{m-1})^{-\frac{1}{2}}
			\cEX{
				z_{m,i}
				\bga{
					\iota_{N,i}(\sigma)
				}
				D_{\ell,M}
			}{\cF_{\ell-1}}
			\nonumber \\
			&
			\qquad\qquad
			-
			(2 q_{m+1} - 2 q_{m})^{-\frac{1}{2}}
			\cEX{
				z_{m,i}
				\bga{
					\iota_{N,i+1}(\sigma)}
				D_{\ell,M}
			}
			{\cF_{\ell-1}}. \label{eq:dqX}
		\end{align}
		In the following, we fix $\ell=1$.
		Gaussian integration by parts yields
		\begin{align}
			&\cEX{
				z_{m,i}
				\bga{
					\iota_{N,i + 1}(\sigma)}
				D_{1,M}
			}
			{\cF_{0}}
			\nonumber \\
			&=
			\cEX{
				\partial_{z_{m,i}}
				(
				\bga{
					\iota_{N,i + 1}(\sigma)}
				D_{1,M}
				)
			}
			{\cF_{0}}
			\nonumber \\
			&=
			\cEX{
				\partial_{z_{m,i}}
				(
				\bga{
					\iota_{N,i + 1}(\sigma)}
				)
				D_{1,M}
			}
			{\cF_{0}}
			+
			\cEX{
				\bga{
					\iota_{N,i + 1}(\sigma)}
				\partial_{z_{m,i}}
				D_{1,M}
			}
			{\cF_{0}}.
			\label{eq:Leibniz}
		\end{align}
		We have
		\begin{align}
			\partial_{z_{m,i}}
			\bga{
				\iota_{N,i + 1}(\sigma)
			}
			=
			(2q_m-2q_{m-1})^{1/2}
			\left(
			\bga{
				\iota_{N,i + 1}(\sigma)
			}
			-
			\bga{
				\iota_{N,i + 1}(\sigma)
			}^2
			\right)
			\label{eq:dz1}
		\end{align}
		and 
		\begin{align}
			\partial_{z_{m,i}} D_{1,M}
			=
			\sum_{\ell=m}^M \zeta_{\ell} (\partial_{z_{m,i}} X_{\ell}) D_{1,M}
			-
			\sum_{\ell=m+1}^M \frac{
				\zeta_{\ell}
				\cEX{(\partial_{z_{m,i}}X_\ell)\exp(\zeta_\ell X_\ell)}{\cF_{\ell-1}}
			}{
				\cEX{\exp(\zeta_\ell X_\ell)}{\cF_{\ell-1}}
			} D_{1,M}.
			\label{eq:dzD}
		\end{align}
		By differentiation, we have
		\begin{align}
			\partial_{z_{m,i}} X_{\ell}
			=
			(2q_m-2q_{m-1})^{1/2}
			\cEX{
				D_{\ell+1,M}
				\bga{
					\iota_{N,i}(\sigma)
				}
			}
			{\cF_{\ell}},
			\label{eq:dzX}
		\end{align}
		and plugging \eqref{eq:dzX} back to \eqref{eq:dzD} yields
		\begin{align}
			&\partial_{z_{m,i}} D_{1,M}
			\nonumber \\
			&=
			(2q_m-2q_{m-1})^{1/2}
                \nonumber \\
                &\qquad\qquad
			\left(
			\sum_{\ell=m}^M \zeta_{\ell} 
			\cEX{
				D_{\ell+1,M}
				\bga{
					\iota_{N,i}(\sigma)
				}
			}
			{\cF_{\ell}} 
			-
			\sum_{\ell=m+1}^M 
			\zeta_\ell
			\cEX{
				D_{\ell,M}
				\bga{
					\iota_{N,i}(\sigma)
				}
			}
			{\cF_{\ell-1}}
			\right)
			D_{1,M} 
			\nonumber \\
			&=
			(2q_m-2q_{m-1})^{1/2}
			\left(
			\bga{
				\iota_{N,i}(\sigma)
			}
			-
			\sum_{\ell=m}^M 
			(\zeta_{\ell+1}-\zeta_\ell)
			\cEX{
				D_{\ell,M}
				\bga{
					\iota_{N,i}(\sigma)
				}
			}
			{\cF_{\ell-1}}
			\right)
			D_{1,M}, \label{eq:dzD.2}
		\end{align}
		where $D_{M+1,M}=1$ in the first equality above. Plugging \eqref{eq:dzD.2} and \eqref{eq:dz1} back to \eqref{eq:Leibniz}, we obtain
		\begin{align}
			&\cEX{
				z_{m,i}
				\bga{
					\iota_{N,i+1}(\sigma)
				}
				D_{1,M}
			}
			{\cF_{0}}
			\nonumber \\
			&=
			(2q_m-2q_{m-1})^{1/2}
                \nonumber \\
                &\qquad
			\cEX{
				\left(
				\bga{\iota_{N,i}(\sigma)}
				-
				\bga{
					\iota_{N,i}(\sigma)
				}
				\sum_{\ell=m}^M
				(\zeta_{\ell+1}-\zeta_\ell)
				\cEX{
					\bga{\iota_{N,i}(\sigma)}
					D_{\ell+1,M}
				}
				{\cF_{\ell-1}}
				\right)
				D_{1,M}
			}{\cF_0}.
			\label{eq:z1D}
		\end{align}
		Applying \eqref{eq:z1D} to \eqref{eq:dqX} with $\ell=1$ yields
		\begin{align*}
			\partial_{q_m} X_{-1}
			=
			\EX{\partial_{q_m} X_{0}}
			&=
			-(\zeta_{m+1}-\zeta_m) \sum_{i=1}^N
			\EX{
				\cEX{
					\bga{
						\iota_{N,i}(\sigma)
					}
					D_{\ell+1,M}
				}{\cF_{\ell}}
				\bga{
					\iota_{N,i}(\sigma)
				}
				D_{1,M}
			}
			\nonumber \\
			&=
			-(\zeta_{m+1}-\zeta_m) \sum_{i=1}^N
			\EX{
				\cEX{
					\bga{
						\iota_{N,i}(\sigma)
					}
					D_{\ell+1,M}
				}{\cF_{\ell}}^2
				D_{1,\ell}
			}
		\end{align*}
		which yields \eqref{eq:quadratic}.
	\end{proof}
	
	The next lemma gives a useful expression of the partial derivatives of $F_N(t,q)$ with respect to $t$ and the $q$'s.
	\begin{proposition}\label{prop:PartialDeriv}
		Let $M,N \in \N$ and $0 = \zeta_0 < \zeta_1 < \zeta_2 < \dots < \zeta_{M - 1} < \zeta_M < \zeta_{M  + 1} = 1$.
		For all $(t,q)\in\R_{\geq 0}\times C_<^{(M)}$ and $k = 0,\dots, M$, 
		we have
		\begin{align}\label{eq:PartialDerivFn}
			\partial_t
			F_N(t,q) &= \E\bbga{A\kl{\sfrac{\s \wedge \tilde\s}{N}}},\notag\\
			\partial_{q_k}
			F_N(t,q) &= \E\bbga{ \1_{\a\wedge \tilde\a = k}(\a,\tilde\a)\, \sfrac{\s \wedge \tilde\s}{N}}.
		\end{align}
	\end{proposition}
	\begin{proof}
		We differentiate $F_N$ to get that for any $t > 0$, \begin{align}\label{eq:PartialDerivComp06} 
			\partial_t
			F_N(t,q)
			&=
			-\frac{1}{N} \partial_t \EX{
				\log\kl{\textstyle \sum_{\alpha \in \N^M} v_\a \sum_{\sigma \in \{-1,1\}^N}   
					\exp\bbkl{H_N\bkl{t,\q,\s,\a}}
			}}\notag\\
			&=
			-\frac{1}{N}  \EX{\partial_t
				\log\kl{\textstyle \sum_{\alpha \in \N^M} v_\a \sum_{\sigma \in \{-1,1\}^N}   
					\exp\bbkl{H_N\bkl{t,\q,\s,\a}}
			}}\notag\\
			&=-\frac{1}{N}  \EX{\ga{(2t)^{-1/2} H_N^A(\s) -N}_{t,q}} = 1 - \frac{1}{N}  (2t)^{-1/2}\E\ga{ H_N^A(\s)}_{t,q}.
		\end{align}
		By Gaussian integration by parts,
		\begin{align}\label{eq:PartialDerivComp07}
			(2t)^{-1/2}\E\ga{ H_N^A(\s)}_{t,q}
			&=(2t)^{-1/2}\E\bbga{\EX{H_N(t,q,\s,\a) H_N^A(\s)} - \EX{H_N(t,q,\s,\a) H_N^A(\tilde\s)}}_{t,q, 2}\notag\\
			&=
			N -N\E\bbga{A\kl{\sfrac{\s \wedge \tilde\s}{N}}}_{t,q, 2},
		\end{align}
		where we used in the last step  that $
		(2t)^{-1/2} \EX{H_N(t,q,\s,\a) H_N^A(\tilde\s)} = NA\kl{\sfrac{\s \wedge \tilde\s}{N}}
		$
		for $\s,\tilde\s\in\leaves{N}$ and $\a\in\N^M$. Inserting \eqref{eq:PartialDerivComp07} into \eqref{eq:PartialDerivComp06} gives
		\begin{align}
			\partial_t
			F_N(t,q) = \E\bbga{A\kl{\sfrac{\s \wedge \tilde\s}{N}}}_{t,q, 2}.
		\end{align}
        By continuity, we can extend this derivative to $t = 0$.
		
		To compute $\partial_{q_k} F_N(t,q)$, one justifies the exchange of derivation and summation/expectation as above for $\partial_{q_k} F_N(t,q)$. This gives
		for $k=0,\dots, M$ that\vspace{-5pt}
		\begin{align}\label{eq:PartialDerivComp2}
			\partial_{q_k}
			F_N(t,q)
			&=
			-\frac{1}{N} \partial_{q_k} \EX{\textstyle
				\log\kl{ \sum_{\alpha \in \N^M} v_\a \sum_{\sigma \in \{-1,1\}^N}   
					\exp\bbkl{H_N\bkl{t,q,\s,\a}}
			}}\notag\\
			&\quad -\frac{1}{N}  
			\E\bbga{\partial_{q_k} H_N\bkl{t,q,\s,\a}}_{t,q}\notag\\
			&=
			\1_{k=M}(k)-\frac{\sqrt{2}}{N}  
			\E\bbga{\partial_{q_k} Y_q(\s,\a)}_{t,q} 
		\end{align}
		where $Y_q(\sigma,\alpha)=Y_{\q_M}(\sigma,\alpha)$ with $\q_M = \sum_{k=0}^M q_k \1_{[\z_k, \z_{k+1})}$ and $Y_{\q_M}(\sigma,\alpha)$ was defined in \eqref{eq:3YwrtZ}.
		Fixing $k = 0,\dots, M-1$, we have
		\begin{align}\label{eq:PartialDerivComp3}
			&\E\bbga{\partial_{q_k} Y_q(\s,\a)}_{t,q}
                \nonumber \\
			&=
			\E\bbga{
				\partial_{q_k} 
				\textstyle
				\sum_{i = 1}^N \sum_{j = 0}^M (q_j - q_{j - 1} )^{1/2} z_{\restr{\s}{i},\restr{\a}{j}}
			}_{t,q}\notag\\
			&=
			\sfrac{1}{2} \E\bbga{
				(q_k-q_{k-1})^{-1/2} \textstyle \sum_{i=1}^N z_{\restr{\s}{i},\restr{\a}{k}}
			}_{t,q}
			-
			\sfrac{1}{2} \E\bbga{
				(q_{k+1}-q_{k})^{-1/2} \textstyle \sum_{i=1}^N z_{\restr{\s}{i},\restr{\a}{k+1}}
			}_{t,q}\notag\\
			&=
			-\sfrac{1}{\sqrt{2}}
			\E\bbga{
				(\s\wedge\tilde\s) 
				\bbkl{
					\1_{\a\wedge\tilde\a \geq k}(\a,\tilde\a)
					-
					\1_{\a\wedge\tilde\a \geq k+1}(\a,\tilde\a)
				}
			}_{t,q, 2}\notag\\
			&=
			-\sfrac{1}{\sqrt{2}}
			\E\bbga{
				(\s\wedge\tilde\s) 
				\1_{\a\wedge\tilde\a = k}(\a,\tilde\a)
			}_{t,q, 2},
		\end{align}
		using that 
		\begin{align}\label{eq:PartialDerivComp4}
			(q_j-q_{j-1})^{-1/2}\EX{ H_N(t,q,\s,\a) \textstyle \sum_{i=1}^N z_{\restr{\tilde\s}{i},\restr{\tilde\a}{j}}} = \sqrt{2}\,  (\s\wedge\tilde\s) \1_{\a\wedge\tilde\a \geq j}(\a,\tilde\a),
		\end{align}
		for $\s,\tilde\s\in\leaves{N}$, $\a,\tilde\a\in\N^M$ and $j=0,\dots, M$.
		In the case $k=0$, we used the convention $q_{-1}=0$.
		Inserting \eqref{eq:PartialDerivComp3} into \eqref{eq:PartialDerivComp2} gives\vspace{-5pt}
		\begin{align}
			\partial_{q_k}
			F_N(t,q)
			&=
			\E\bbga{
				\1_{\a\wedge\tilde\a = k}(\a,\tilde\a)
				\sfrac{\s\wedge\tilde\s}{N}
			}_{t,q, 2},
		\end{align}
		for $k=0,\dots,M-1$.
		Analogously, by \eqref{eq:PartialDerivComp4},
		\begin{align}
			\E\bbga{\partial_{q_M} Y_q(\s,\a)}
			&=
			\sfrac{1}{2} \E\bbga{
				(q_M-q_{M-1})^{-1/2} \textstyle \sum_{i=1}^N z_{\restr{\s}{i},\restr{\a}{M}}
			}_{t,q}\notag\\
			&=
			\sfrac{N}{\sqrt{2}}
			\kl{
				1-
				\E\bbga{
					\sfrac{\s\wedge\tilde\s}{N}
					\1_{\a\wedge\tilde\a = M}(\a,\tilde\a)
				}_{t,q, 2}
			},
		\end{align}
		so
		\begin{align}
			\partial_{q_M}
			F_N(t,q)
			&=
			\E\bbga{
				\1_{\a\wedge\tilde\a = M}(\a,\tilde\a)
				\sfrac{\s\wedge\tilde\s}{N}
			}_{t,q, 2}.\qedhere
		\end{align}
	\end{proof}
	The next proposition asserts that for each $t>0$, $F_N(t,\cdot)$ is continuously differentiable on $C_<^{(M)}$.
	\begin{proposition}\label{prop:FNC1}
		Fix $M,N \in \N$ and $0 = \zeta_0 < \zeta_1 < \zeta_2 < \dots < \zeta_{M - 1} < \zeta_M < \zeta_{M  + 1} = 1$. 
		For all $t>0$, the function $q\mapsto F_N(t,q)$ is continuously differentiable on $C_<^{(M)}$. 
	\end{proposition}
	\begin{proof}
		By Proposition~\ref{prop:PartialDeriv}, $F_N(t,\cdot)$ is partially differentiable in $q$ with partial derivatives
		\begin{align}\label{eq:PartialDerivFn2}
			\partial_{q_k}
			F_N(t,q) &= \E\bbga{ \1_{\a\wedge \tilde\a = k}(\a,\tilde\a)\, \sfrac{\s \wedge \tilde\s}{N}}_{t,q,2}, \qquad k = 0,\dots, M.
		\end{align}
		It remains to show that for each $k \in \gkl{0,\dots, M}$, the partial derivatives in \eqref{eq:PartialDerivFn2} are continuous in $q$.
		For this purpose, let $q^{(n)}\in C_<^{(M)}$ for $n\in\N$ with $\lim_{n\uparrow\infty}\, \norm{\q-\q^{(n)}}_1 = 0$, where $\q = \sum_{k=0}^M q_k^{(n)} \1_{[\z_k, \z_{k+1})}$ and $\q^{(n)}=\sum_{k=0}^M q_k \1_{[\z_k, \z_{k+1})}$. 
		
		For all $t\geq 0$, all $n\in \N$  and all $k \in \gkl{0,\dots, M}$,
		$\bga{
			\1_{\a\wedge\tilde\a = k}(\a,\tilde\a)
			\sfrac{\s\wedge\tilde\s}{N}
		}_{t,q^{(n)},2}$ is bounded from above by $1$. Thus, by the dominated convergence theorem, for all $k \in \gkl{0,\dots, M}$,
		\begin{align}
			\lim_{n\uparrow\infty}
			\partial_{q_k}
			F_N(t,q^{(n)})
			&= \lim_{n\uparrow\infty} \E\bbga{
				\1_{\a\wedge\tilde\a = k}(\a,\tilde\a)
				\sfrac{\s\wedge\tilde\s}{N}
			}_{t,q^{(n)},2}\notag\\
			&= \EX{\lim_{n\uparrow\infty} \bbga{
					\1_{\a\wedge\tilde\a = k}(\a,\tilde\a)
					\sfrac{\s\wedge\tilde\s}{N}
				}_{t,q^{(n)},2}}\notag\\
			&=
			\E\bbga{
				\1_{\a\wedge\tilde\a = k}(\a,\tilde\a)
				\sfrac{\s\wedge\tilde\s}{N}
			}_{t,q,2}\notag\\
			&=
			\partial_{q_k}
			F_N(t,q).
		\end{align}
		In the second last step, we used that the map
		\begin{align}
			C_<^{(M)} &\to \R,\notag\\
			p=(p_0,\dots,p_M)&\mapsto \bbga{
				\1_{\a\wedge\tilde\a = k}(\a,\tilde\a)
				\sfrac{\s\wedge\tilde\s}{N}
			}_{t,p,2}
		\end{align}
		is continuous for almost all the randomness contained in the Gibbs average $\ga{\cdot}_{t,p,2}$. Note that by \eqref{eq:3YwrtZ}, this randomness can be described by  
        \[
        \kl{v_\a}_{\a\in\N^M}, \quad
        \kl{H_N^A(\sigma)}_{\s\in\leaves{N}}
        \quad \text{and} \quad 
        \kl{z_{\restr{\s}{i},\restr{\a}{j}}}_{\s\in\leaves{N};\,\a\in\N^M;\,i=1,\dots, N; j = 0,\dots,M},\] 
        so it does not depend on $p$.
		\end{proof}
		
		\section{
			Proof of Proposition~\ref{prop:LipschitzFn}}\label{sec:FnLipschitz}
		
		The goal of this section is to prove Proposition~\ref{prop:LipschitzFn}. Assuming \eqref{eq:FNLipschitzClaim}, the following lemma yields that $F_N(t,\cdot)$ admits a unique Lipschitz extension to~$Q_1$.
		\begin{lemma}\label{lem:DensityQM}
			Fix $r\in [1,\infty)$. For any $\q\in Q_r$, the sequence $\q_M$ defined by 
			\[
			\q_M
			\coloneqq
			\sum_{j=0}^{M}
            \q\left(\sfrac{j}{M+1}\right)
            \1_{\left[\frac{j}{M+1}, \frac{j+1}{M+1}\right)}
			\]
			converges to $\q$ in $\norm{\cdot}_r$. In particular, defining 
			\begin{align}\label{eq:Qequidist}
				\qmeq \coloneqq \gkl{\q\in Q \colon \text{ there exist } 0 \leq q_0 \leq \dots \leq q_{M} \text{ so that } \q = 		\textstyle\sum_{j=0}^{M} q_j \1_{\left[\frac{j}{M+1}, \frac{j+1}{M+1}\right)} },
			\end{align}
			this implies that both 
			$\bigcup_{M=0}^\infty\qm$ and $\bigcup_{M=0}^\infty\qmeq$ are dense in $Q_r$ \wrt $\norm{\cdot}_r$. 
		\end{lemma}
		Indeed, for any $\q\in Q_1$, by choosing a sequence $(\q_n)\in \bigcup_{M=0}^\infty\qm$ that converges to $\q$ with respect to $\norm{\cdot}_1$, \eqref{eq:FNLipschitzClaim} and Lemma~\ref{lem:DensityQM} implies that the unique Lipschitz extension to $Q_1$ exists. 
		\begin{proof}[Proof of Lemma~\ref{lem:DensityQM}]
			%	Let $\q \in Q_1$. It suffices to show that $\q$ is Riemann-integrable. Then each step function determined by any Riemann sum of $\q$ lies in $\bigcup_{M=0}^\infty Q^{(M)}$, since $\q$ is increasing. Therefore, $\q$ can be approximated pointwise by $(\q_N)_{N\in\N} \subset \bigcup_{M=0}^\infty Q^{(M)}$ with $\q_N \leq \q$ for all $N\in\N$. Convergence of $(\q_N)_{N\in\N}$ to $\q$ in  $L_1\bkl{[0,1)}$ then follows from the dominated convergence theorem.
			%	
			The proof is standard. 
			Let $\q \in Q_r$. It suffices to show that $\q_M$ defined in the statement of the lemma converges to $\q$ in $\norm{\cdot}$ because $\q_M\in \qmeq\subseteq \bigcup_{M=0}^\infty\qmeq \subseteq \bigcup_{M=0}^\infty\qm$. First of all, $\q_M$ converges to $\q$ at every continuity point of $\q$, which is all but countably many points of $[0,1)$. Moreover, since $\q$ is increasing and non-negative, we have $\abs{\q-\q_M}= \q-\q_M\leq \q$. Therefore, by the assumption that $\q\in L_r([0,1))$, the Lebesgue dominated convergence theorem yields that $\q_M\rightarrow\q$ in $L_r$ as $M$ goes to infinity.
		\end{proof}
        \begin{remark}
            Lemma~\ref{lem:DensityQM} also holds if we replace the inequalities in \eqref{eq:Qequidist} with strict inequalities. The same continuity proof extends to this scenario. 
        \end{remark}
        
		The remainder of the section is devoted to the proof of Proposition~\ref{prop:LipschitzFn}.
		\begin{proof}[Proof of Proposition~\ref{prop:LipschitzFn}.]
			This proof has the same structure as that of \cite[Proposition~6.3]{JCBook}, the analogous result for the SK model.
			Let $M\in\N$.
			
			\textit{Step 1: Averaging over the Ruelle cascades.} For $\q = \sum_{k=0}^M q_k \1_{[\z_k, \z_{k+1})} \in \qm$, we have by Lemma~\ref{lem:RecursiveAveragingFreeEnergy} that
			\begin{align}
				&F_N(t,\q)
				\nonumber \\
				&=  q_M + t 
				-\frac{1}{N}\E\Bigg[
				\log\Bigg(
				\E\bigg[
				\E\bigg[ 
				\dots
				\E\Big[
				Z_{t,q}(z_0,\dots, z_M)^{\z_M}
				\Big\vert\,
				\FF_{M-1}
				\Big]^{\frac{\z_{M-1}}{\z_M}}
				\dots
				\bigg\vert\,\FF_1
				\bigg]^{\frac{\z_1}{\z_2}}
				\bigg\vert\, \FF_0
				\bigg]^{\frac{1}{\z_1}}
				\Bigg)
				\Bigg],\label{eq:InitCondPsi1Copy}
			\end{align}
			recalling
			\begin{align}
				Z_{t,q}(z_0,\dots, z_M) 
				&\coloneqq \sum_{\s\in\leaves{N}} 
				\exp\brac{
					\sqrt{2t}\, H_N^A(\s) +  \sqrt{2}\, \textstyle\sum_{j=0}^M (q_j - q_{j-1})^{\frac{1}{2}} z_j(\sigma)
				},\notag\\
				\FF_k 
				&\coloneqq \s(z_0,\dots,z_k), \quad k=0,\dots, M-1.
				\label{eq:ZNestedNotation}
			\end{align}
			Furthermore recall that, independent of $\kl{H_N^A(\sigma)}_{\s\in\leaves{N}}$, the processes 
            \[
            z_1 = (z_1(\s))_{\s\in\cbrac{-1,1}^N}, \dots,z_M = (z_M(\s))_{\s\in\cbrac{-1,1}^N}
            \] 
            are i.i.d.\ copies of $z_0 \coloneqq (z_0(\s))_{\s\in\cbrac{-1,1}^N}$, a BRW on the $N$-level binary tree with standard Gaussian increments.

			\textit{Step 2: Repetitions of parameters.} In \eqref{eq:enrichedFE}, we have defined $F_N$ for step functions $\q\in\qm$ that can have repetition in the $q$ parameters, but in fact, we show that one can assume that the $q$'s are distinct by taking another step function with possibly fewer jumps. This will be important to perform the derivatives.
			
			Suppose that $\hat \q = \sum_{k=0}^M \hat q_k \1_{[\z_k, \z_{k+1})}\in \qm$ has repeated parameters, i.e., there exists $k=0,\dots, M-1$ with $\hat q_k=\hat q_{k+1}$. Then, there exists $\hat\pp = \sum_{k=0}^{M'} \hat p_k \1_{[\z_k', \z_{k+1}')}\in Q^{(M')}$ where $(p_0,\ldots,p_{M'})\in C^{(M')}_<$
			such that $\hat\pp(u) = \hat\q(u)$ for all $u\in [0,1]$. 
			If we define $F_N(t,\hat \q)$ as in \eqref{eq:InitCondPsi1Copy}, replacing $\q$ by $\hat \q$, then the factor $(\hat q_{k+1}-\hat q_k)^{\frac{1}{2}}$, which appears in $Z_{t,\hat \q}(z_0,\dots, z_M)$, is zero. Thus, $z_{k+1}$ does not affect the right-hand side of \eqref{eq:InitCondPsi1Copy} and we get
			\begin{align}
				&\E\bigg[ 
				\E\bigg[ 
				\dots
				\E\Big[
				\E\Big[
				Z_{t,\hat \q}(z_0,\dots, z_M)^{\z_M}
				\Big\vert\,
				\FF_{M-1}
				\Big]^{\frac{\z_{M-1}}{\z_M}}
				\,\Big\vert\,
				\FF_{M-2}
				\Big]^{\frac{\z_{M-2}}{\z_{M-1}}}
				\dots
				\bigg\vert\,\FF_{k+1}
				\bigg]^{\frac{\z_{k+1}}{\z_{k+2}}}
				\bigg\vert\,\FF_{k}
				\bigg]^{\frac{\z_{k}}{\z_{k+1}}}\notag\\
				&=
				\E\bigg[ 
				\dots
				\E\Big[
				\E\Big[
				Z_{t,\hat \q}(z_0,\dots, z_M)^{\z_M}
				\Big\vert\,
				\FF_{M-1}
				\Big]^{\frac{\z_{M-1}}{\z_M}}
				\,\Big\vert\,
				\FF_{M-2}
				\Big]^{\frac{\z_{M-2}}{\z_{M-1}}}
				\dots
				\bigg\vert\,\FF_{k+1}
				\bigg]^{\frac{\z_{k}}{\z_{k+2}}}.
			\end{align}
			Repeating this procedure for all repetitions in the parameters of $\hat\q$ shows that $F_N(t,\hat \q)=F_N(t,\hat\pp)$ for all $t\geq 0$. 
            
			\textit{Step 3: Extending the partial derivatives.}
			Recall the definitions of $C_\leq^{(M)}$ and $C_<^{(M)}$ in \eqref{eq:DefConeRm}, \eqref{eq:DefConeInteriorRm}, respectively.
			We have computed the partial derivatives of $F_N(t,q)$ in Proposition~\ref{prop:PartialDeriv} and have shown in Proposition~\ref{prop:FNC1} that $F_N(t,\cdot)$  is continuously differentiable on $C_<^{(M)}$ for each $t> 0$. Since $C_<^{(M)}$ is convex, the partial derivatives extend to the boundary, so $F_N(t,\cdot) $ is continuously differentiable on $C_\leq^{(M)}$ for each $t\geq 0$. This is a classical result from analysis, see e.g.\ \cite{ExtendC1}. In particular, the formula \eqref{eq:PartialDerivFn} for the partial derivatives of $F_N$ in Proposition~\ref{prop:PartialDeriv} is also true for $t\geq 0$ and $q\in C_\leq^{(M)}$. 
			
			\textit{Step 4: Lipschitz continuity.}
			We now compare $F_N(t,\q)$ and $F_N(t,\tilde\q)$ for $\q,\tilde\q \in \bigcup_{M=1}^\infty\qml$ by taking $M\in\N$ and 
			\begin{align}
				0 &= \zeta_0 < \zeta_1 < \zeta_2 < \dots < \zeta_{M - 1} < \zeta_M < \zeta_{M  + 1} = 1,\notag\\
				0&=q_{-1} \leq q_0 \leq \dots \leq q_M < \infty, \notag\\
				0&=\tilde q_{-1} \leq \tilde q_0 \leq \dots \leq \tilde q_M < \infty,
			\end{align}
			so that
			\begin{align}
				\q = \sum_{k=0}^M q_k \1_{[\z_k, \z_{k+1})} \qquad \text{and} \qquad \tilde\q = \sum_{k=0}^M \tilde q_k \1_{[\z_k, \z_{k+1})}.
			\end{align}
			For $\lambda \in [0,1]$, we define
			\begin{align}
				\q^{(\lambda)} &\coloneqq \lambda \q + (1-\lambda) \tilde\q = \sum_{k=0}^M (\lambda q_k + (1-\lambda) \tilde q_k) \1_{[\z_k, \z_{k+1})}.
			\end{align}
			Note that $\q^{(\lambda)}\in\qml$.
			By the fundamental theorem of calculus,
			\begin{align}\label{eq:Lipsch0}
				\abs{F_N(t,\q) - F_N(t,\tilde\q)}
				=
				\abs{\int_0^1 \partial_\lambda F_N \bkl{t,\q^{(\tilde\lambda)}} \,\d\tilde\lambda\,}
				\leq \sup_{\tilde\lambda \in[0,1]} \abs{\partial_\lambda F_N\bkl{t,\q^{(\tilde\lambda)}}}.
			\end{align}
			For each $\tilde\lambda \in [0,1]$, the chain rule implies that 
			\begin{align}\label{eq:Lipsch1}
				\partial_\lambda F_N\kl{t,\q^{(\tilde\lambda)}}
				= \sum_{k=0}^M(q_k - \tilde q_k) \partial_{q_k} F_N \kl{t,q^{(\lambda)}},
			\end{align}
			where 
			\begin{align}
				q^{(\tilde\lambda)} &\coloneqq \bbrac{\tilde\lambda q_0 + (1-\tilde\lambda) \tilde q_0, \dots, \tilde\lambda q_M + (1-\tilde\lambda) \tilde q_M}.
			\end{align}
			By \eqref{eq:PartialDerivFn} and an application of \cite[Equation~(2.81)]{Pan-book} to compute the averages of $\1_{\a\wedge\tilde\a=k}\kl{\a ,\tilde\a}$,
			\begin{align}\label{eq:Lipsch2}
				\partial_{q_k} F_N (t,q^{(\tilde\lambda)}) = \E\bbga{\1_{\a\wedge\tilde\a=k}\kl{\a ,\tilde\a} \sfrac{\s\wedge\tilde\s}{N}}_{t,q^{(\tilde\lambda)},2} \leq \E\bbga{\1_{\a\wedge\tilde\a=k}\kl{\a ,\tilde\a}}_{t,q^{(\tilde\lambda)},2} = \zeta_{k+1} - \zeta_k,
			\end{align}
			where $\ga{\cdot}_{t,q^{(\lambda)},2}$ is defined in \eqref{eq:GibbsAvgSigmaAlphan} with $n=2$.
			Inserting \eqref{eq:Lipsch1} and \eqref{eq:Lipsch2} into \eqref{eq:Lipsch0} gives
			\begin{align}\label{eq:Lipsch3}
				\abs{F_N(t,\q) - F_N(t,\tilde\q)} \leq \sum_{k=0}^M \abs{\zeta_{k+1} - \zeta_k} \abs{q_k - \tilde q_k} = \norm{\q-\tilde\q}_1.
			\end{align}
			Since by Lemma~\ref{lem:DensityQM}, $\bigcup_{M=1}^\infty\qm$ is dense in $Q_1$, \eqref{eq:Lipsch3} implies the existence of a unique extension of $F_N$ to $\R_{\geq 0} \times Q_1$. This extension satisfies \eqref{eq:Lipsch3} on its domain.
		\end{proof}
	
	\section{Proof of Theorem~\ref{thm:InitCond}}\label{sec:initalCondition}
	Proposition~\ref{prop:LipschitzFn} asserts that $F_N$ can be defined on $\R_{\geq 0}\times Q_1$.
	The central result of this section is Theorem~\ref{thm:InitCond} which states that for $\q\in Q_1$, 
	\begin{align}\label{eq:InitialConditionClaim}
		\Psi(\q) \coloneqq \lim_{N\uparrow \infty} F_N(0,\q) = 
		-\log 2 + \int_0^1 \brac{\q(u) - \sfrac{\log 2}{u^2}}_+ \d u. 
	\end{align}
	We use \eqref{eq:InitialConditionClaim} to prove an explicit formula for $\Psi^*$ of $\Psi$ in Proposition~\ref{prop:HJCremPsiCOnvexDual}.
	Moreover, \eqref{eq:InitialConditionClaim} directly implies the following corollary regarding the regularity of $\Psi$.
	\begin{corollary}\label{cor:convexLipschitz}
		$\Psi$ is convex and Lipschitz continuous on $Q_1$ with Lipschitz constant $1$.
	\end{corollary}
	\begin{proof}
		We start by proving Lipschitz continuity. Let $\q,\tilde\q\in Q_1$ and assume w.l.o.g.\ that
		\begin{align}
			\zeta \coloneqq \inf \cbrac{s\in(0,1)\colon \q(s) > \sfrac{\log 2}{u^2}} \leq \inf\cbrac{s\in(0,1)\colon \tilde\q(s) > \sfrac{\log 2}{u^2}} \eqqcolon \tilde\zeta.
		\end{align}
		Then, since $\tilde\q(u) < \sfrac{\log 2}{u^2}$ for $u \in [0,\tilde\zeta)$,
		\begin{align}
			\abs{\Psi(\q) -\Psi(\tilde\q)}
			= \abs{
				\int_\zeta^{\tilde\zeta} \q(u) - \sfrac{\log 2}{u^2}\, \d u 
				+ 
				\int_{\tilde\zeta}^1 \q(u) - \tilde\q(u)\, \d u
			}
			\leq 
			\int_{\zeta}^1 \abs{\q(u) -\tilde\q(u)}\, \d u \leq \norm{\q-\tilde\q}_1.
		\end{align}
		To prove convexity, note that for any $\lambda \in [0,1]$, we have 
		$\lambda \q + (1-\lambda)\tilde\q \in Q_1$. Since $s\mapsto (s)_+$ is convex,
		\begin{align}
			\Psi(\lambda \q + (1-\lambda)\tilde\q) 
			&\leq - \log 2 + \int_0^1 \lambda \brac{\q(u) - \sfrac{\log 2}{u^2}}_+ + (1-\lambda) \brac{\tilde\q(u) - \sfrac{\log 2}{u^2}}_+ \d u\notag\\
			&= \lambda \Psi(\q) + (1-\lambda) \Psi (\tilde\q).\qedhere
		\end{align}
	\end{proof}
	For $M \in \N$, we define 
	\begin{align}
		\qmt \coloneqq \Big\{\q \in \qm \colon &\q = \textstyle\sum_{j=0}^M q_j \1_{[\zeta_j, \zeta_{j+1})} 
		\text{ 
			with } 0 \leq q_0 \leq \dots\leq q_M < \infty \notag\\
		&\text{and there exists } k_0 \in \cbrac{1,\dots,M} \text{ with } q_{k_0} = \sfrac{\log 2}{\zeta_{k_0}^2} \Big\}.
        \label{eq:QMT}
	\end{align}
	\begin{lemma}\label{lem:DensityQMM}
		For each $M\in\N$, $\tilde{Q}^{(M+1)}$ is dense in $\qm$ \wrt $\norm{\cdot}_1$.
	\end{lemma}
	\begin{proof}
		Let $\q\in Q^{(M)}$ and $\varepsilon>0$ be sufficiently small. Without loss of generality, we can assume that $\q$ has distinct $q$'s.
		We distinguish the following three cases, see Figure~\ref{fig:q1} for an illustration:
		\begin{figure}
			\centering
			\includegraphics[width=0.98\linewidth]{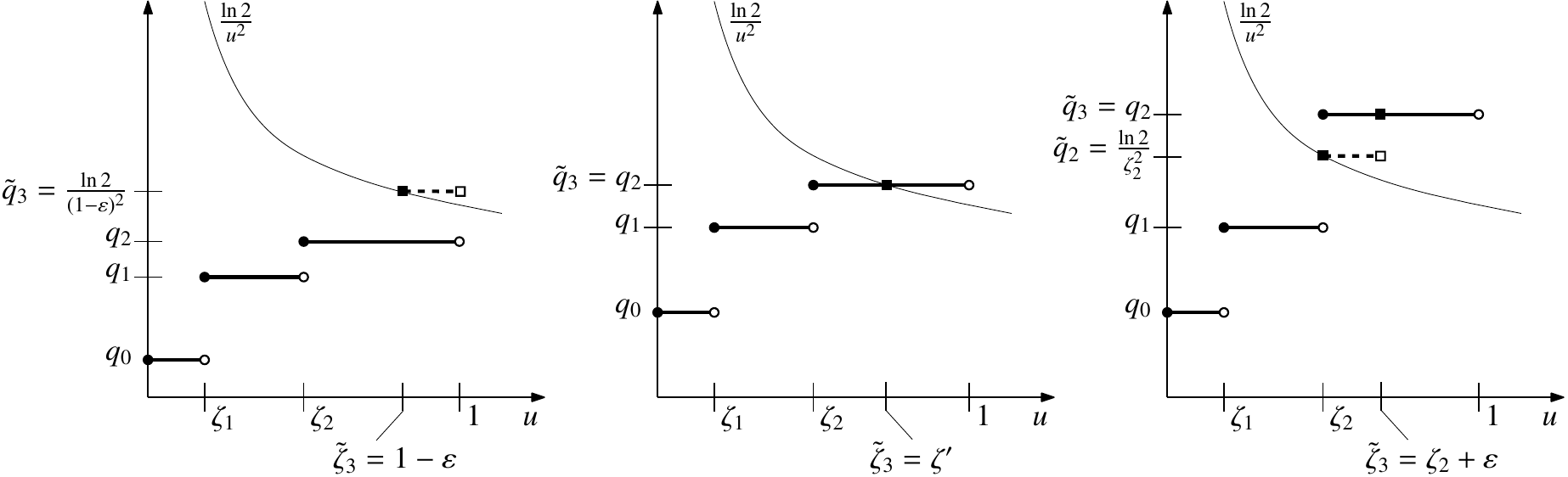}
			\caption{Examples of Case~\ref{it:enhancedq1} to Case~\ref{it:enhancedq3} for $M=2$ from left to right. Newly added path segments of $\tilde \q$ are dashed with boxes at their boundary. Shared path segments of $\q$ and $\tilde \q$ are thick with circles at their boundaries.}
			\label{fig:q1}
		\end{figure}
		\begin{enumerate}
			\item \label{it:enhancedq1} $q_M < \log 2$. In this case, we define $\tilde\q = \sum_{j=0}^{M+1} \tilde{q}_j \1_{[\tilde{\zeta}_j,\tilde{\zeta}_{j+1})}\in \tilde{Q}^{(M+1)}$, where
			\begin{align}
				\tilde{q}_j 
				=
				\begin{cases}
					q_j, & j=0,\ldots, M, \\
					\sfrac{\log 2}{(1-\varepsilon)^2}, & j=M+1,
				\end{cases}
			\end{align}
			and 
			\vspace{-5pt}
			\begin{align}
				\tilde{\zeta}_j
				=
				\begin{cases}
					\zeta_j, & j=0,\ldots,M, \\
					1-\varepsilon, & j=M+1, \\
					1, & j=M+2.
				\end{cases}
			\end{align}
			Then, we have that $\norm{\tilde{\q}-\q}_1\leq \kl{\frac{\log 2}{(1-\varepsilon)^2} - q_M}\varepsilon$.
			\item There exists $k_0\in \llbracket 0,M-1\rrbracket$ such that $\q(\zeta_{k_0})<\frac{\log 2}{\zeta_{k_0}^2}$ and $\q(\zeta_{k_0+1}) >\frac{\log 2}{\zeta_{k_0+1}^2}$, and there exists $\hat\zeta\in [\zeta_{k_0},\zeta_{k_0+1})$ such that $q_{k_0} = \frac{\log 2}{{\hat\zeta}^2}$. 
			Then  
			we define $\tilde\q = \sum_{j=0}^{M+1} \tilde{q}_j \1_{[\tilde{\zeta}_j,\tilde{\zeta}_{j+1})}\in \tilde{Q}^{(M+1)}$, where
			\begin{align}
				\tilde{q}_j 
				=
				\begin{cases}
					q_j, & j=0,\ldots,k_0, \\
					\sfrac{\log 2}{\tilde\zeta_{k_0+1}^2}, & j=k_0+1, \\
					q_{j-1}, & j= k_0+2,\ldots,M, 
				\end{cases}
			\end{align}
			and \vspace{-5pt}
			\begin{align}
				\tilde{\zeta}_j
				=
				\begin{cases}
					\zeta_j, & j= 0,\ldots,k, \\
					\hat\zeta, & j=k_0+1, \\
					\zeta_{j-1}, & j= k_0+2,\ldots,M.
				\end{cases}
			\end{align}
			Then, we have that $\norm{\tilde{\q}-\q}_1=0$.
			\item \label{it:enhancedq3} There exists $k_0\in \llbracket 0,M-1\rrbracket$ such that $\q(\zeta_{k_0})<\frac{\log 2}{\zeta_{k_0}^2}$ and $\q(\zeta_{k_0+1}) >\frac{\log 2}{\zeta_{k_0+1}^2}$, and $q_{k_0} < \frac{\log 2}{{\zeta}^2}$ for all $\zeta\in [\zeta_{k_0},\zeta_{k_0+1})$. 
			Then 
			we define $\tilde\q = \sum_{j=0}^{M+1} \tilde{q}_j \1_{[\tilde{\zeta}_j,\tilde{\zeta}_{j+1})}\in \tilde{Q}^{(M+1)}$, where
			\begin{align}
				\tilde{q}_j 
				=
				\begin{cases}
					q_j, & j=0,\ldots,k_0, \\
					\sfrac{\log 2}{\tilde\zeta_{k_0+1}^2}, & j=k_0+1, \\
					q_{j-1}, & j= k_0+2,\ldots,M, 
				\end{cases}
			\end{align}
			and \vspace{-5pt}
			\begin{align}
				\tilde{\zeta}_j
				=
				\begin{cases}
					\zeta_j, & j= 0,\ldots,k_0+1, \\
					\zeta_{k_0+1}+\varepsilon, & j=k_0+2, \\
					\zeta_{j-1}, & j= k_0+3,\ldots,M.
				\end{cases}
			\end{align}
			Then, we have that $\norm{\tilde{\q}-\q}_1\leq (q_{k_0+1} - q_{k_0})\varepsilon$.
		\end{enumerate}
		Thus, we have shown that $\bigcup_{M=0}^\infty \tilde{Q}^{(M+1)}$ is dense in $\bigcup_{M=0}^\infty Q^{(M)}$. 
	\end{proof}

	Now we prove the central result of this section.
	\begin{proof}[Proof of Theorem~\ref{thm:InitCond}]
		Let $M\in\N$. By Lemma~\ref{lem:DensityQM} and Lemma~\ref{lem:DensityQMM}, $\bigcup_{M=0}^\infty \qmt$ is dense in $Q_1$ \wrt $\norm{\cdot}_1$. This and the Lipschitz continuity of the right-hand side of \eqref{eq:InitialConditionClaim} imply that it suffices to prove \eqref{eq:InitialConditionClaim} for $q\in \qmt$. 
        Fix $\q \in \qmt$. By the definition of $\qmt$ in \eqref{eq:QMT}, there exists $k_0 \in \cbrac{1,\dots,M}$ such that $\q_{k_0} = \sfrac{\log 2}{\zeta_{k_0}^2}$. In this setting, since $\q$ is increasing,
		\begin{align}
			\int_0^1 \brac{\q(u) - \sfrac{\log 2}{u^2}}_+ \d u
			= \int_{\zeta_{k_0}}^1 \q(u) - \sfrac{\log 2}{u^2}\, \d u
			= \log 2 - \sfrac{\log 2}{\zeta_{k_0}} + \int_{\zeta_{k_0}}^1 \q(u)\, \d u.
		\end{align}
		Thus, it remains to prove that 
		\begin{align}
			\Psi(\q) = - \sfrac{\log 2}{\zeta_{k_0}} + \int_{\zeta_{k_0}}^1 \q(u)\, \d u.
			\label{eq:InitCondClaim2}
		\end{align}
		Recall that
		\begin{align}
			Y_\q(\sigma,\alpha) \overset{\d}{=} \sum_{i = 1}^N \sum_{k = 0}^M (q_{k} - q_{k - 1} )^{1/2} z_{\restr{\s}{i},\restr{\a}{k}},
		\end{align}
		where each $z_{\restr{\s}{i},\restr{\a}{k}}$ is from a family of i.i.d.\ standard Gaussians.
		For $\q\in\qmt$, we have by Lemma~\ref{lem:RecursiveAveragingFreeEnergy} that
		\begin{align}
			\Psi(\q)
			&= \lim_{N\uparrow \infty} F_N(0,\q)\notag\\
			&=  q_M - \lim_{N\uparrow \infty}\frac{1}{N}
			\E\Bigg[
			\log\Bigg(
			\E\bigg[
			\E\bigg[ 
			\dots
			\E\big[
			Z_{0,\q}(z_0,\dots, z_M)^{\z_M}
			\big\vert\,
			\FF_{M-1}
			\big]^{\frac{\z_{M-1}}{\z_M}}
			\dots
			\bigg\vert\,\FF_1
			\bigg]^{\frac{\z_1}{\z_2}}
			\bigg\vert\, \FF_0
			\bigg]^{\frac{1}{\z_1}}
			\Bigg)
			\Bigg],\label{eq:InitCondPsi1}
		\end{align}
		recalling that
		\begin{align}
			Z_{0,\q}(z_0,\dots, z_M) &= \sum_{\s\in\cbrac{-1,1}^N} \exp\brac{\sqrt{2}\, \textstyle\sum_{j=0}^M (q_j - q_{j-1})^{\frac{1}{2}} z_j(\sigma)},\notag\\
			\FF_k &= \s(z_0,\dots,z_k), \quad k=0,\dots, M-1.
			\label{eq:InitCondNotationX}
		\end{align}
		Furthermore, $z_1 = (z_1(\s))_{\s\in\cbrac{-1,1}^N}, \dots, z_M = (z_M(\s))_{\s\in\cbrac{-1,1}^N}$ are i.i.d.\ copies of $z_0 \coloneqq (z_0(\s))_{\s\in\cbrac{-1,1}^N}$, a branching random walk on the $N$-level binary tree with standard Gaussian increments. 
		We prove \eqref{eq:InitCondClaim2} by presenting matching upper and lower bounds of $\Psi$.
		
		\noindent \textbf{Lower bound:} 
		Since the function $x \mapsto x^{\z_M}$ is concave and thus subadditive, 
		\begin{align}
			&\cEX{
				Z_{0,\q}(z_0, \dots, z_M)^{\zeta_{M}}
			}{
				\FF_{M-1}
			}\notag\\
			&=
			\cEX{
				\brac{\textstyle\sum_{\s\in\cbrac{-1,1}^N} \exp\brac{\textstyle \sqrt{2} \sum_{j=0}^{M} (q_j-q_{j-1})^{1/2} z_j(\s)}}^{\zeta_{M}}
			}{
				\FF_{M-1}
			}\notag\\
			&\leq 
			\cEX{
				\textstyle\sum_{\s\in\cbrac{-1,1}^N} \exp\brac{\textstyle\sqrt{2}\,\z_M\sum_{j=0}^{M} (q_j-q_{j-1})^{1/2} z_j(\s)}
			}{
				\FF_{M-1}
			}\notag\\
			&=\sum_{\s\in\cbrac{-1,1}^N}
			\exp\brac{\textstyle\sqrt{2}\,\z_M\sum_{j=0}^{M-1} (q_j-q_{j-1})^{1/2} z_j(\s)}
			\EX{\sqrt{2}\,\z_M(q_M-q_{M-1})^{1/2} z_M(\s)}\notag\\
			&=\sum_{\s\in\cbrac{-1,1}^N}
			\exp\brac{\textstyle\sqrt{2}\,\z_M\sum_{j=0}^{M-1} (q_j-q_{j-1})^{1/2} z_j(\s)}
			\exp\brac{\z_M^2(q_M-q_{M-1})N} \nonumber \\
            &=
            Z_{0,\zeta_M^2\q}(z_0,\ldots,z_{M-1}) \exp\brac{\z_M^2(q_M-q_{M-1})N},
			\label{eq:InitCondUpper1}
		\end{align}
		where \eqref{eq:InitCondLower1} follows from the fact that $z_M(\sigma)\sim \N(0,N)$ with respect to $\EX{\cdot}$.
		Inserting \eqref{eq:InitCondUpper1} into the iterating expectation  in \eqref{eq:InitCondPsi1} and then proceeding as in \eqref{eq:InitCondUpper1} up to level $k_0$, we see that 
		\begin{align}
			&\frac{1}{N}
			\E\Bigg[
			\log\Bigg(
			\E\bigg[
			\E\bigg[ 
			\dots
			\E\Big[
			\E\big[
			Z_{0,\q}(z_0,\dots, z_M)^{\z_M}
			\big\vert\,
			\FF_{M-1}
			\big]^{\frac{\z_{M-1}}{\z_M}}
			\,\Big\vert\,
			\FF_{M-2}
			\Big]^{\frac{\z_{M-2}}{\z_{M-1}}}
			\dots
			\bigg\vert\,\FF_1
			\bigg]^{\frac{\z_1}{\z_2}}
			\bigg\vert\, \FF_0
			\bigg]^{\frac{1}{\z_1}}
			\Bigg)
			\Bigg]\notag\\
			&\leq
			\z_M(q_M-q_{M-1})
                \nonumber \\
                &
                +\frac{1}{N}
			\E\Bigg[
			\log\Bigg(
			\E\bigg[
			\E\bigg[ 
			\dots
            % \nonumber \\
            % &\qquad\qquad\qquad\qquad
			\E\Big[
			\brac{
				% \textstyle\sum_{\s\in\cbrac{-1,1}^N}
				% \exp\brac{\sqrt{2}\,\z_M\sum_{j=0}^{M-1} (q_j-q_{j-1})^{1/2} z_j(\s)}
                Z_{0,\zeta_M^2\q}(z_0,\ldots,z_{M-1})
			}^{\frac{\z_{M-1}}{\z_M}}
			\,\Big\vert\,
			\FF_{M-2}
			\Big]^{\frac{\z_{M-2}}{\z_{M-1}}}
			% \nonumber \\
			% &\qquad\qquad\qquad\qquad\qquad\qquad\qquad\qquad\qquad\qquad\qquad\qquad\qquad\qquad\qquad\qquad   
			\dots
			\bigg\vert\,\FF_1
			\bigg]^{\frac{\z_1}{\z_2}}
			\bigg\vert\, \FF_0
			\bigg]^{\frac{1}{\z_1}}
			\Bigg)
			\Bigg]\notag\\
			&\leq
			\sum_{j=k_0+1}^M \z_j(q_j-q_{j-1})\notag\\
			&+\frac{1}{N}
			\E\Bigg[
			\log\Bigg(
			\E\bigg[
			\E\bigg[ 
			\dots
            % \nonumber \\
            % &\qquad
			\E\Big[
			\brac{
				% \textstyle\sum_{\s\in\cbrac{-1,1}^N}
				% \exp\brac{\sqrt{2}\,\z_{k_0+1}\sum_{j=0}^{k_0} (q_j-q_{j-1})^{1/2} z_j(\s)}
                Z_{0,\zeta_{k_0+1}^2\q}(z_0,\ldots,z_{k_0})
			}^{\frac{\z_{k_0}}{\z_{k_0+1}}}
                % \nonumber \\
                % &\qquad\qquad\qquad\qquad\qquad\qquad\qquad\qquad\qquad\qquad\qquad\qquad\qquad
                \,\Big\vert\,
			\FF_{k_0-1}
			\Big]^{\frac{\z_{k_0-1}}{\z_{k_0}}}
			\dots
			\bigg\vert\,\FF_1
			\bigg]^{\frac{\z_1}{\z_2}}
			\bigg\vert\, \FF_0
			\bigg]^{\frac{1}{\z_1}}
			\Bigg)
			\Bigg].\label{eq:InitCondLower2first}  
		\end{align}
		By subadditivity of $x\mapsto x^{\frac{\z_{k_0}}{\z_{k_0+1}}}$ and then Jensen's inequality (applied to the concave functions $x\mapsto x^{\frac{\z_{k_0-1}}{\z_{k_0}}}$, $x\mapsto x^{\frac{\z_{k_0-2}}{\z_{k_0}}}$, \dots, $x\mapsto x^{\frac{\z_{1}}{\z_{k_0}}}$ and $x\mapsto\log x$),
		\begin{align}
			&\E\Bigg[
			\log\Bigg(
			\E\bigg[
			\E\bigg[ 
			\dots
            % \nonumber \\
            % &\qquad
			\E\Big[
			\brac{
				% \textstyle\sum_{\s\in\cbrac{-1,1}^N}
				% \exp\brac{\sqrt{2}\,\z_{k_0+1}\sum_{j=0}^{k_0} (q_j-q_{j-1})^{1/2} z_j(\s)}
                Z_{0,\zeta_{k_0+1}^2\q}(z_0,\ldots,z_{k_0})
			}^{\frac{\z_{k_0}}{\z_{k_0+1}}}
			\,\Big\vert\,
			\FF_{k_0-1}
			\Big]^{\frac{\z_{k_0-1}}{\z_{k_0}}}
            % \nonumber \\
            % &\qquad\qquad\qquad\qquad\qquad\qquad\qquad\qquad\qquad\qquad\qquad\qquad\qquad\qquad\qquad
			\dots
			\bigg\vert\,\FF_1
			\bigg]^{\frac{\z_1}{\z_2}}
			\bigg\vert\, \FF_0
			\bigg]^{\frac{1}{\z_1}}
			\Bigg)
			\Bigg]\notag\\
			&\leq 
			\E\Bigg[
			\log\Bigg(
			\E\bigg[
			\E\bigg[ 
			\dots
                % \nonumber \\
                % &\qquad
			\E\Big[
			% \textstyle\sum_{\s\in\cbrac{-1,1}^N}
			% \exp\brac{\sqrt{2}\,\z_{k_0}\sum_{j=0}^{k_0} (q_j-q_{j-1})^{1/2} z_j(\s)}
            Z_{0,\zeta_{k_0}^2\q}(z_0,\ldots,z_{k_0})
			\,\Big\vert\,
			\FF_{k_0-1}
			\Big]^{\frac{\z_{k_0-1}}{\z_{k_0}}}
			\dots
			\bigg\vert\,\FF_1
			\bigg]^{\frac{\z_1}{\z_2}}
			\bigg\vert\, \FF_0
			\bigg]^{\frac{1}{\z_1}}
			\Bigg)
			\Bigg]\notag\\
			&\leq 
			\frac{1}{\z_{k_0}}\log\Bigg(
			\E\Bigg[
			\E\bigg[
			\dots
			\E\Big[
			% \textstyle\sum_{\s\in\cbrac{-1,1}^N}
			% \exp\brac{\sqrt{2}\,\z_{k_0}\sum_{j=0}^{k_0} (q_j-q_{j-1})^{1/2} z_j(\s)}
                Z_{0,\zeta_{k_0}^2\q}(z_0,\ldots,z_{k_0})
			\,\Big\vert\,
			\FF_{k_0-1}
			\Big]
			\dots
			\bigg\vert\, \FF_0
			\bigg]
			\Bigg]
			\Bigg)\notag\\
			&=
			\frac{1}{\z_{k_0}}
			\log\Big(
			\E\Big[
			% \textstyle\sum_{\s\in\cbrac{-1,1}^N}
			% \exp\brac{\sqrt{2}\,\z_{k_0}\sum_{j=0}^{k_0} (q_j-q_{j-1})^{1/2} z_j(\s)}
                Z_{0,\zeta_{k_0}^2\q}(z_0,\ldots,z_{k_0})
			\Big]
			\Big)\notag\\
			&=
			\frac{1}{\z_{k_0}}
			\log
			\brac{
				2^N \exp\brac{\z_{k_0}^2 q_{k_0} N}
			}\notag\\
			&= N \sfrac{\log 2}{\z_{k_0}} + \z_{k_0} q_{k_0} N. \label{eq:InitCondLower2first.1}
		\end{align}
		Plugging  \eqref{eq:InitCondLower2first} and \eqref{eq:InitCondLower2first.1} back to \eqref{eq:InitCondPsi1} and then using $q_{k_0} = \sfrac{\log 2}{\z_{k_0}^2}$, we obtain
		\begin{align}
			\Psi(\q)
			\geq q_M - \sum_{j=k_0+1}^M \z_j(q_j-q_{j-1}) - \sfrac{\log 2}{\z_{k_0}} - \z_{k_0} q_{k_0}
			= q_M - \sum_{j=k_0+1}^M \z_j(q_j-q_{j-1}) - 2\sqrt{q_{k_0}\log 2}.
		\end{align}
		Since
		\begin{align}\label{eq:InitCondLowerEnd-1}
			- \sfrac{\log 2}{\z_{k_0}} + \int_{\z_{k_0}}^1 \q(u)\, \d u
			&=
			- \sfrac{\log 2}{\z_{k_0}} 
			+ \sum_{j=k_0}^M (\z_{k_0+1}-\z_{k_0}) q_{k_0}\notag\\
			&=
			- \sfrac{\log 2}{\z_{k_0}} 
			- \z_{k_0} q_{k_0} 
			+ q_M
			- \sum_{j=k_0+1}^M \z_j(q_j-q_{j-1})\notag\\
			&=
			q_M - \sum_{j=k_0+1}^M \z_j(q_j-q_{j-1}) - 2\sqrt{q_{k_0}\log 2},
		\end{align}
		we have shown that
		\begin{align}\label{eq:InitCondLowerEnd}
			\Psi(\q)
			\geq
			- \sfrac{\log 2}{\z_{k_0}} + \int_{\z_{k_0}}^1 \q(u)\, \d u.
		\end{align}
		Thus, to prove \eqref{eq:InitCondClaim2}, we need to find an upper bound which matches \eqref{eq:InitCondLowerEnd}.

		\noindent \textbf{Upper bound:} If $k \leq M-1$, we define the Gibbs measure corresponding to level $k$ as
		\begin{align}
			\mu_{\q,k}\brac{\s} &\coloneqq \frac{1}{\tilde Z_{\q,k}} \exp\brac{\textstyle\sqrt{2}\sum\limits_{j=0}^{k} (q_j-q_{j-1})^{1/2} z_j(\s)} \quad \forall\,\s\in\cbrac{-1,1}^N,\notag\\
			\tilde Z_{\q,k}&\coloneqq \tilde Z_{\q,k}(z_0,\dots,z_{k}) \coloneqq \sum_{\s\in\cbrac{-1,1}^N} \exp\brac{\textstyle\sqrt{2}\sum\limits_{j=0}^{k} (q_j-q_{j-1})^{1/2} z_j(\s)}. 
		\end{align}
		We have
		\begin{align}
			&\cEX{
				Z_{0,\q}(z_0, \dots, z_M)^{\zeta_{M}}
			}{
				\FF_{M-1}
			}\notag\\
			&=
			\tilde Z_{\q,k}^{\zeta_{M}}
			\cEX{
				\textstyle
				\brac{
					\sum_{\sigma\in\{-1,1\}^N}
					\mu_{\q,k}(\sigma)
					\exp\brac{
						\sqrt{2} \sum_{j=k+1}^M (q_j - q_{j-1})^{1/2} z_j(\sigma)
					}
				}^{\zeta_{M}}
			}{
				\FF_{M-1}
			}\notag\\
			&=    
			\tilde Z_{\q,k}^{\zeta_{M}}
			\cEX{
				\E_{\mu_{\q,k}}
				\left[
				\exp\brac{\textstyle
					\sqrt{2} \sum_{j=k+1}^M (q_j - q_{j-1})^{1/2} z_j(\sigma)
				}
				\right]^{\zeta_{M}}
			}{
				\FF_{M-1}
			}.\label{eq:lowerBoundMuKJensenBla}
		\end{align}
		Then, by applying Jensen's inequality to $\E_{\mu_{\q,k}}$,
		\begin{align}
			&\E_{\mu_{\q,k}}
			\left[
			\exp\brac{\textstyle
				\sqrt{2} \sum_{j=k+1}^M (q_j - q_{j-1})^{1/2} z_j(\sigma)
			}
			\right]^{\zeta_{M}}\notag\\
			& \geq
			\E_{\mu_{\q,k}}
			\left[
			\exp\brac{\textstyle
				\sqrt{2}\,\zeta_{M}\sum_{j=k+1}^M (q_j - q_{j-1})^{1/2} z_j(\sigma)
			}
			\right]\notag\\
			&=
			\sum_{\sigma\in\{-1,1\}^N}
			\mu_{\q,k}(\sigma)
			\exp\brac{\textstyle
				\sqrt{2}\,\zeta_{M}\sum_{i=k+1}^{M} (q_i - q_{i-1})^{1/2} z_i(\sigma)
			}.\label{eq:lowerBoundMuKJensenBla2}
		\end{align}
		Inserting \eqref{eq:lowerBoundMuKJensenBla2} into \eqref{eq:lowerBoundMuKJensenBla} and then using that $\FF_{M-1} = \s(z_0,\dots,z_{M-1})$ and that $z_M$ is independent of $\FF_{M-1}$, we get
		\begin{align}
			&\cEX{
				Z_{0,\q}(z_0, \dots, z_M)^{\zeta_{M}}
			}{
				\FF_{M-1}
			}\notag\\
			&\geq
			\tilde Z_{\q,k}^{\zeta_{M}}
			\cEX{
				\sum_{\sigma\in\{-1,1\}^N}
				\mu_{\q,k}(\sigma)
				\exp\brac{\textstyle
					\sqrt{2}\,\zeta_{M}\sum_{i=k+1}^{M} (q_i - q_{i-1})^{1/2} z_i(\sigma)
				}
			}{
				\FF_{M-1}
			}\notag\\
			&=
			\tilde Z_{\q,k}^{\zeta_{M}}
			\sum_{\sigma\in\{-1,1\}^N}
			\mu_{\q,k}(\sigma)
			\exp\brac{\textstyle
				\sqrt{2}\,\zeta_{M}\sum_{i=k+1}^{M-1} (q_i - q_{i-1})^{1/2} z_i(\sigma)
			}
			\EX{
				\eee^{
					\sqrt{2}\,\zeta_{M}(q_M - q_{M-1})^{1/2} z_M(\sigma)
				}
			}\notag\\
			&=
			\exp\brac{\zeta_{M}^2\brac{q_M-q_{M-1}}N}
			\tilde Z_{\q,k}^{\zeta_{M}}
			\sum_{\sigma\in\{-1,1\}^N}
			\mu_{\q,k}(\s)
			\exp\brac{\textstyle
				\sqrt{2}\,\zeta_{M}\sum_{i=k+1}^{M-1} (q_i - q_{i-1})^{1/2} z_i(\sigma)
			}.\label{eq:lowerBoundMuKJensen}
		\end{align}
		Inserting \eqref{eq:lowerBoundMuKJensen} into \eqref{eq:InitCondPsi1} and then proceeding analogously to \eqref{eq:lowerBoundMuKJensen} up to level $k$, we get that
		\begin{align}
			&\Psi(\q) \notag\\
			&\leq
			q_M - \z_M(q_M-q_{M-1})\notag\\
			&\quad - \lim_{N\uparrow \infty} \frac{1}{N}
			\E\Bigg[
			\log\Bigg(
			\E\bigg[
			\E\bigg[ 
			%                 \dots\notag\\
			% &\qquad
			\dots 
			\E\Big[(\tilde Z_{\q,k})^{\zeta_{M-1}}
			\E_{\mu_{\q,k}}
			\left[
			\exp\brac{\textstyle
				\sqrt{2}\,\zeta_{M}\sum_{i=k+1}^{M-1} (q_i - q_{i-1})^{1/2} z_i(\sigma)
			}
			\right]^{\frac{\z_{M-1}}{\z_M}}\notag\\
			&\qquad\qquad\qquad\qquad\qquad\qquad\qquad
			\,\Big\vert\,
			\FF_{M-2}
			\Big]^{\frac{\z_{M-2}}{\z_{M-1}}}
			\dots
			\bigg\vert\,\FF_1
			\bigg]^{\frac{\z_1}{\z_2}}
			\bigg\vert\, \FF_0
			\bigg]^{\frac{1}{\z_1}}
			\Bigg)
			\Bigg]\notag\\
			&\leq
			q_M - \sum_{j=k+1}^{M}
			\z_j(q_j-q_{j-1}) 
			- \lim_{N\uparrow \infty} \frac{1}{N} X^{(k)}_{N},\label{eq:nestedExpectation3}
		\end{align}
		where
		\begin{align}
			& X^{(k)}_{N} \coloneqq
			\E\Bigg[
			\log\Bigg(
			\E\bigg[
			\E\bigg[ 
			%                 \dots\notag\\
			% &\qquad
			\dots 
			\E\Big[(\tilde Z_{\q,k})^{\zeta_{k}}
			\Big\vert\,
			\FF_{k-1}
			\Big]^{\frac{\z_{k-1}}{\z_{k}}}
			\dots
			\bigg\vert\,\FF_1
			\bigg]^{\frac{\z_1}{\z_2}}
			\bigg\vert\, \FF_0
			\bigg]^{\frac{1}{\z_1}}
			\Bigg)
			\Bigg].
		\end{align}
		In the case $k=M$, $X^{(M)}_{N}$ is the nested expectation in \eqref{eq:InitCondPsi1}, i.e., we did not need to apply the steps in \eqref{eq:nestedExpectation3} in this case.
		Let $a=(a_0,\dots,a_k)\in\R_+^{k+1}$ with $\norm{a}_2^2 < 2\log 2$ and $(\bz(\sigma))_{\sigma\in\{-1,1\}^N} = (z_0(\sigma),\ldots,z_k(\sigma))_{\sigma\in\{-1,1\}^N}$. We set
		\begin{align}
			T_a(N) \coloneqq \cbrac{\s\in\cbrac{-1,1}^N \colon 
				N^{-1}\bz(\s) \in [a_0,\infty)\times\cdots\times [a_k,\infty)}.
		\end{align}
		Then, for any $k=1,\dots, M$, 
		\begin{align}
			&X^{(k)}_{N}
            \nonumber \\
			& =
			\E\Bigg[
			\brac{
				\1_{\abs{T_a(N)} \geq 1}
				+ \1_{\abs{T_a(N)}= 0}
			}
			\log\Bigg(
			\E\bigg[
			\E\bigg[ 
			\dots 
			\E\Big[(\tilde Z_{\q,k})^{\zeta_{k}}
			\Big\vert\,
			\FF_{k-1}
			\Big]^{\frac{\z_{k-1}}{\z_{k}}}
			\dots
			\bigg\vert\,\FF_1
			\bigg]^{\frac{\z_1}{\z_2}}
			\bigg\vert\, \FF_0
			\bigg]^{\frac{1}{\z_1}}
			\Bigg)
			\Bigg]\notag\\
			&\geq
			\E\Bigg[ \1_{\abs{T_a(N)} \geq 1}
			\log\Bigg(
			\E\bigg[
			\E\bigg[ 
			\dots 
			\E\Big[
			\brac{ 
				\textstyle\sum_{\s\in\cbrac{-1,1}^N}
				\1_{\s\in T_a(N)}
				\exp\brac{
					\sqrt{2}
					\sum_{j=0}^k
					(q_j-q_{j-1})^{1/2}
					z_j(\s)
				}
			}^{\z_k}\notag\\
			&\qquad\qquad\qquad\qquad\qquad
			\,\Big\vert\,
			\FF_{k-1}
			\Big]^{\frac{\z_{k-1}}{\z_{k}}}
			\dots
			\bigg\vert\,\FF_1
			\bigg]^{\frac{\z_1}{\z_2}}
			\bigg\vert\, \FF_0
			\bigg]^{\frac{1}{\z_1}}
			\Bigg)
			\Bigg]\notag\\
			&\quad 
			+ \E\Bigg[
			\1_{\abs{T_a(N)} = 0}
			\log\Bigg(
			\E\bigg[
			\E\bigg[ 
			%                 \dots\notag\\
			% &\qquad
			\dots 
			\E\Big[(\tilde Z_{\q,k})^{\zeta_{k}}
			\Big\vert\,
			\FF_{k-1}
			\Big]^{\frac{\z_{k-1}}{\z_{k}}}
			\dots
			\bigg\vert\,\FF_1
			\bigg]^{\frac{\z_1}{\z_2}}
			\bigg\vert\, \FF_0
			\bigg]^{\frac{1}{\z_1}}
			\Bigg)
			\Bigg]. \label{eq:InitCondLower1}
		\end{align}
		For $\s\in T_a$ and $j=0,\dots,k$, we have $z_j(\s)\geq a_j$ so the first summand of the last line of \eqref{eq:InitCondLower1} satisfies
		\begin{align}
			&\E\Bigg[ \1_{\abs{T_a(N)} \geq 1}
			\log\Bigg(
			\E\bigg[
			\E\bigg[ 
			\dots 
			\E\Big[
			\brac{ 
				\textstyle\sum_{\s\in\cbrac{-1,1}^N}
				\1_{\s\in T_a(N)}
				\exp\brac{
					\sqrt{2}
					\sum_{j=0}^k
					(q_j-q_{j-1})^{1/2}
					z_j(\s)
				}
			}^{\z_k}\notag\\
			&\qquad\qquad\qquad\qquad\qquad
			\,\Big\vert\,
			\FF_{k-1}
			\Big]^{\frac{\z_{k-1}}{\z_{k}}}
			\dots
			\bigg\vert\,\FF_1
			\bigg]^{\frac{\z_1}{\z_2}}
			\bigg\vert\, \FF_0
			\bigg]^{\frac{1}{\z_1}}
			\Bigg)
			\Bigg]\notag\\
			&\geq
			\sqrt{2}\,
			\PR{
				\abs{T_a(N)}\geq 1
			}
			\sum_{j=0}^k (q_j-q_{j-1})^{1/2} a_j N\notag\\
			&\quad
			+
			\E\Bigg[
			\1_{\abs{T_a(N)} \geq 1}
			\log\Bigg(
			\E\bigg[
			\E\bigg[ 
			\dots 
			\E\Big[\abs{T_a(N)}^{\zeta_{k}}
			\Big\vert\,
			\FF_{k-1}
			\Big]^{\frac{\z_{k-1}}{\z_{k}}}
			\dots
			\bigg\vert\,\FF_1
			\bigg]^{\frac{\z_1}{\z_2}}
			\bigg\vert\, \FF_0
			\bigg]^{\frac{1}{\z_1}}
			\Bigg)
			\Bigg]\notag\\
			&\geq
			\sqrt{2}\,
			\PR{
				\abs{T_a(N)}\geq 1
			}
			\sum_{j=0}^k (q_j-q_{j-1})^{1/2} a_j N.\label{eq:InitCondLower1.5}
		\end{align}
		For the second summand of the last line of \eqref{eq:InitCondLower1}, we estimate, for any $\tilde\s \in \cbrac{-1,1}^N$,
		\begin{align}
			&\E\Bigg[
			\1_{\abs{T_a(N)} = 0}
			\log\Bigg(
			\E\bigg[
			\E\bigg[ 
			\dots 
			\E\Big[(\tilde Z_{\q,k})^{\zeta_{k}}
			\Big\vert\,
			\FF_{k-1}
			\Big]^{\frac{\z_{k-1}}{\z_{k}}}
			\dots
			\bigg\vert\,\FF_1
			\bigg]^{\frac{\z_1}{\z_2}}
			\bigg\vert\, \FF_0
			\bigg]^{\frac{1}{\z_1}}
			\Bigg)
			\Bigg]\notag\\
			&\geq
			\E\Bigg[
			\1_{\abs{T_a(N)} = 0}
			\log\Bigg(
			\E\bigg[
			\E\bigg[ 
			\dots 
			\E\Big[
			\exp\brac{\sqrt{2}\,\z_k
				\textstyle\sum_{j=0}^k
				(q_j-q_{j-1})^{1/2}
				z_j(\tilde\s)
			}
			\Big\vert\,
			\FF_{k-1}
			\Big]^{\frac{\z_{k-1}}{\z_{k}}}
            \nonumber \\
            &\qquad\qquad\qquad\qquad\qquad\qquad\qquad\qquad\qquad\qquad\qquad\qquad\qquad
			\dots
			\bigg\vert\,\FF_1
			\bigg]^{\frac{\z_1}{\z_2}}
			\bigg\vert\, \FF_0
			\bigg]^{\frac{1}{\z_1}}
			\Bigg)
			\Bigg]   
            \notag\\
			&\geq
			\E\Bigg[
			\1_{\abs{T_a(N)} = 0}
			\sfrac{1}{\z_1}
			\log\Bigg(
			\E\bigg[
			\exp\brac{
				\sqrt{2}\,\z_1
				\textstyle\sum_{j=0}^k
				(q_j-q_{j-1})^{1/2}
				z_j(\tilde\s)
			}
			\bigg\vert\, \FF_0
			\bigg]
			\Bigg)
			\Bigg], \label{eq:InitCondLower2}
		\end{align}
		using Jensen's inequality for the concave functions $x\mapsto x^{\frac{\z_{k-1}}{\z_k}}$, \dots, $x\mapsto x^{\frac{\z_1}{\z_2}}$ in the last step. 
		By the independence of $z_1, \dots, z_k$ and $\FF_0$ and 
        the fact that that $z_k(\sigma)\sim \N(0,N)$ 
        we get,
		\begin{align}
			&\E\bigg[
			\exp\brac{
				\sqrt{2}\,\z_1
				\textstyle\sum_{j=0}^k
				(q_j-q_{j-1})^{1/2}
				z_j(\tilde\s)
			}
			\bigg\vert\, \FF_0
			\bigg]\notag\\
			&= \exp\brac{\sqrt{2}\,\z_1 q_0^{1/2} z_0(\tilde\s)}
			\,\EX{\exp\brac{
					\sqrt{2}\,\z_1
					\textstyle\sum_{j=1}^k
					(q_j-q_{j-1})^{1/2}
					z_j(\tilde\s)
			}}\notag\\
			&=\exp\brac{\sqrt{2}\,\z_1 q_0^{1/2} z_0(\tilde\s)}
			\,\exp\brac{\z_1^2 (q_M - q_0)N}.\label{eq:InitCondLower3.1}
		\end{align}
		Inserting \eqref{eq:InitCondLower3.1} into \eqref{eq:InitCondLower2} gives
		\begin{align}
			&\E\Bigg[
			\1_{\abs{T_a(N)} = 0}
			\log\Bigg(
			\E\bigg[
			\E\bigg[ 
			\dots 
			\E\Big[(\tilde Z_{\q,k})^{\zeta_{k}}
			\Big\vert\,
			\FF_{k-1}
			\Big]^{\frac{\z_{k-1}}{\z_{k}}}
			\dots
			\bigg\vert\,\FF_1
			\bigg]^{\frac{\z_1}{\z_2}}
			\bigg\vert\, \FF_0
			\bigg]^{\frac{1}{\z_1}}
			\Bigg)
			\Bigg]\notag\\
			&\geq
			\z_1 (q_M-q_0)\, \PR{\abs{T_a(N)} = 0} N + \EX{\1_{\abs{T_a(N)} = 0} \, q_0^{1/2} z_0(\tilde\s)}.
			\label{eq:InitCondLower4.1}
		\end{align}
		The last summand of the last line of \eqref{eq:InitCondLower4.1} satisfies
		\begin{align}
			&\EX{\1_{\abs{T_a(N)} = 0} \, q_0^{1/2} z_0(\tilde\s)}\notag\\
			&= \EX{\1_{\abs{T_a(N)} = 0} \1_{ z_0(\tilde\s)\geq -N^{2/3}} \, q_0^{1/2} z_0(\tilde\s)} + \EX{\1_{\abs{T_a(N)} = 0} \1_{ z_0(\tilde\s)<-N^{2/3}} \, q_0^{1/2} z_0(\tilde\s)},
		\end{align}
		where
		\begin{align}
			\EX{\1_{\abs{T_a(N)} = 0} \1_{ z_0(\tilde\s)\geq -N^{2/3}} \, q_0^{1/2} z_0(\tilde\s)} 
			&\geq - q_0^{1/2} N^{2/3}\, \PR{\abs{T_a(N)} = 0,  z_0(\tilde\s)\geq -N^{2/3}}\notag\\
			&\geq - q_0^{1/2} N^{2/3}\, \PR{\abs{T_a(N)} = 0},\label{eq:InitCondLower3}
		\end{align}
		and, since $z_0(\tilde\s)$ is negative,
		\begin{align}
			\EX{\1_{\abs{T_a(N)} = 0} \1_{ z_0(\tilde\s)<-N^{2/3}} \, q_0^{1/2} z_0(\tilde\s)}
			&\geq
			q_0^{1/2}\, \EX{\1_{ z_0(\tilde\s)<-N^{2/3}} \, z_0(\tilde\s)}\notag\\
			&=
			q_0^{1/2} \int_{-\infty}^{-N^{2/3}} \sfrac{\d y}{\sqrt{2\pi}N} y \eee^{-\frac{y^2}{2N}}\notag\\
			&=
			-\sfrac{1}{\sqrt{2\pi}} \eee^{-\frac{N^{2/3}}{2}}.\label{eq:InitCondLower4}
		\end{align}
		Combining the estimates in \eqref{eq:InitCondLower1}--\eqref{eq:InitCondLower4}, we see that
		\begin{align}
			\frac{1}{N} X^{(k)}_{N}\geq
			\sqrt{2}\,
			\PR{
				\abs{T_a(N)}\geq 1
			}
			\sum_{j=0}^k (q_j-q_{j-1})^{1/2} a_j 
			+
			\z_1 (q_M-q_0)\, \PR{\abs{T_a(N)} = 0} 
			+ o_N(1).
		\end{align}
		Note that with respect to $\mathbb{P}(\cdot)$, $(\bz(\sigma))_{\sigma\in\{-1,1\}^N} = (z_0(\sigma),\ldots,z_k(\sigma))_{\sigma\in\{-1,1\}^N}$ is a binary $(k+1)$-dimensional branching random walk where the increment distributes as a $(k+1)$-dimensional standard Gaussian distribution, i.e., all of the components are independent and each component distributed as a standard Gaussian variable. 
        Denote by 
		\[\mathcal{R}_N = 
		\left\{N^{-1}\bz(\sigma):\sigma\in\{-1,1\}^N\right\}
		\] 
		the normalized range of $(\bz(\sigma))_{\sigma\in\{-1,1\}^N}$. In \cite{biggins1978asymptotic}, Biggins characterized the limiting shape of $\cR_N$ by the ball centered at $0$ with radius $\sqrt{2\log 2}$
        \[
        \mathcal{B}(0,\sqrt{2\log 2})
        =
        \{x:\norm{x}_2\leq \sqrt{2\log 2}\}.
        \]
        In particular, 
        Biggins showed in \cite[Theorem A(ii)]{biggins1978asymptotic} that for every interior point $x$ of $\mathcal{B}(0,\sqrt{2\log 2})$, there exists a sequence $x_N\in\mathcal{R}_N$ such that
		$x_N\rightarrow x$
		when $N$ goes to infinity. 
        On the other hand, since $\norm{a}_2<\sqrt{2\log 2}$, the set
		\begin{align*}
			\mathcal{S}
            =
            (\mathrm{int}\,\mathcal{B}(0,\sqrt{2\log 2}))
			\cap 
			\left(
			[a_0,\infty)\times\cdots\times [a_k,\infty)
			\right)
		\end{align*}
        is non-empty. Pick $x\in \mathcal{S}$, i.e., there exists $\rho>0$ such that $\mathcal{B}(x,\rho)\subseteq \mathcal{S}$.
		Then, Biggins' shape theorem implies in particular that almost surely $\cB(x,\rho)\cap \cR_N\neq\varnothing$ for $N$ sufficiently large. Thus,
		\begin{align}
            \lim_{N\rightarrow\infty}
			\PR{
				\abs{T_a(N)}\geq 1
			} 
            &=
            \lim_{N\rightarrow\infty}
            \PR{
            \cR_N\cap ([a_0,\infty)\times \cdots \times[a_k,\infty))
            \neq \varnothing
            }
            \nonumber \\
            &\geq 
            \lim_{N\rightarrow\infty}
            \PR{
            \cR_N\cap \cB(x,\rho)
            \neq \varnothing
            }
            =1.
            \label{eq:claim T=0}
		\end{align}
		\begin{remark}
			In Corollary 3.4.4 of \cite{alban_influence_2025}, one of the authors provided an alternative proof of \eqref{eq:claim T=0} without appealing to Biggins' shape theorem.
		\end{remark}
		Then, \eqref{eq:claim T=0} yields that
		\begin{align}\label{eq:LimXnkLowerBound}
			\lim_{N\uparrow \infty} \frac{1}{N} X^{(k)}_{N} \geq\sqrt{2}\,
			\sum_{j=0}^k (q_j-q_{j-1})^{1/2} a_j.
		\end{align}
		Taking the supremum over all $a=(a_0,\dots,a_k)\in\R_+^{k+1}$ with $\norm{a}_2^2 < 2\log 2$.
		we have
		\begin{align}
			&\sup\brac{
				\textstyle\sum_{j=0}^k (q_j-q_{j-1})^{1/2} a_j \colon a=(a_0,\dots,a_k)\in\R_+^{k+1},\, \norm{a}_2^2 < 2\log 2
			}\notag\\
			&=
			\max
			\brac{
				\textstyle\sum_{j=0}^k (q_j-q_{j-1})^{1/2} a_j \colon a=(a_0,\dots,a_k)\in\R_+^{k+1}, \,\norm{a}_2^2 \leq 2\log 2
			}\notag\\
			&=\sqrt{2 q_k \log 2},
		\end{align}
		since by the Cauchy--Schwartz inequality, the maximizer $a^*$ satisfies
		\begin{align}
			a^*_j =\sqrt{2\log 2}\, \frac{(q_j-q_{j-1})^{1/2}}{q_k^{1/2}}, \quad j=0,\dots,k.
		\end{align}
		Thus,
		\begin{align}
			\Psi(\q)
			\leq
			q_M - \sum_{j=k+1}^{M}
			\z_j(q_j-q_{j-1}) 
			- 2\sqrt{q_k \log 2}.\label{eq:InitCondFinalLowerBound}
		\end{align}
		By \eqref{eq:InitCondLowerEnd-1}, we have found a matching upper bound to \eqref{eq:InitCondLowerEnd}. Thus, \eqref{eq:InitCondClaim2} holds for $\q \in \qmt$ with $\q_{k_0} = \sfrac{\log 2}{\zeta_{k_0}^2}$.
	\end{proof}
	
	\section{Proof of Theorem~\ref{thm:main}} \label{sec:finitedim}
	
	% \todo{I think this is still required...
	% 	\begin{enumerate}
	% 		\item Start with a Gaussian comparison to argue that we can add regularity to $A$. In particular, we want $A$ to be convex and Lipschitz. There might be some issues for the boundary case. In particular, we might need another argument for the case where you have a jump at $1$.
	% 		\item I want to look at the finite $M$ HJE and show that the limiting free energy is a sub/supersolution. This is done by applying the comparison principle from Theorem 1.2 in Chen and Xia '24. Need to say a word on the regularity of the limiting FE.
	% 		\item By Gaussian comparison, the liminf/limsup of the FE is lower bound/upper bound by the solution of the HJE. 
	% 		\item (Extension from finite $M$) Apply Theorem 4.6 to show that the limit of these solutions is the HL formula
	% 		\item Prove that the HL formulae are the same.
	% 	\end{enumerate}
	% }
	
	%We recall the notation in Section~\ref{sec:derivative}. 
	Fix $M\in\N$. In this section, we assume that the $\zeta$'s are equidistant, i.e., we assume $\zeta_k = k/(M+1)$ for $k=0,\ldots,M+1$. This restriction can be done without loss of generality because $Q_2\subseteq Q_1$ by H\"{o}lder's inequality so Proposition~\ref{prop:LipschitzFn} implies that the free energies for general $\q\in Q_2$ can be approximated by step functions in $\cup_{M=1}^\infty \qmeq$ which was defined in \eqref{eq:Qequidist}. Also, we recall the notation in Section~\ref{sec:derivative} where we defined the closed convex cone
	\begin{align*}
		C_\leq^{(M)} = \cbrac{q=(q_0,\dots,q_M)\in \R^{M+1} \colon 0 \leq q_0 \leq \dots \leq q_M}
	\end{align*}
	and the free energy notation $F_N(t,q): \R_{\geq 0} \times C_{\leq}^{(M)}$ defined in \eqref{eq:F_N(t,q)}. Similarly, we can identify the initial condition with $q \in C_{\leq}^{(M)}$ by defining $$\Psi(q) = \Psi(\q_M) \text{ where } \q_M = \sum_{m=0}^{M} q_m \1_{[m/(M+1), (m+1)/(M+1))}.$$
	When we view $F_N$ and $\Psi$ as a function on $\R_{\geq 0}\times Q_2$ and $Q_2$, we adopt the notation $F_N(t,\q)$ and $\Psi(\q)$, respectively.
	
	To prove Theorem~\ref{thm:main}, we want to apply the Hamilton--Jacobi theory for mean-field disordered systems (see, for example \cite{chen_hamilton-jacobi_2023,Mourrat2023FreeEnergyUpperBound}). This theory requires monotonicity of the initial condition $\Psi$. Note that $Q_2$ is a closed convex cone on the Hilbert space $L_2\bkl{[0,1), \R}$. Its \emph{dual cone} is the closed convex cone defined by
	\begin{align}
		Q_2^* \coloneqq \gkl{\pp \in L_2\bkl{[0,1),\R}\colon \textstyle\int_0^1 \pp(u) \q(u) \,\d u \geq 0 \:\forall\,\q\in Q_2}.
	\end{align}
	Similarly, the dual cone ${C_\leq^{(M)}}^\ast$ of $C_\leq^{(M)}$ is defined by
	\begin{align*}
		{C_\leq^{(M)}}^\ast = 
		\gkl{
			p\in \R^{M+1} : p\cdot q\geq 0 \ \forall\, q\in C_\leq^{(M)}
		}.
	\end{align*}
	We say that a function $G:Q_2\rightarrow\R$ is \emph{$Q_2^\ast$-increasing} if for all $\q_1,\q_2\in Q_2$, 
	\begin{align*}
		\int_0^1 \pp(u)\q_1(u)\dd{u}
		\geq 
		\int_0^1 \pp(u)\q_2(u)\dd{u}, 
		\quad \text{for all} \quad \pp\in Q_2
	\end{align*}
	implies that $G(\q_1)\geq G(\q_2)$. Likewise, a function $J:C_\leq^{(M)}\rightarrow\R$ is \emph{${C_\leq^{(M)}}^\ast$-increasing} if for all $q_1,q_2\in C_\leq^{(M)}$, for all $p\in C_\leq^{(M)}$, $p\cdot q_1 \geq p\cdot q_2$
	implies that $J(q_1)\geq J(q_2)$. 
	\begin{lemma}\label{lem:PsiCStarIncreasing}
		For all $M\in\N$, the function $q\mapsto\Psi(q)$ defined on $C_\leq^{(M)}$ is ${C_\leq^{(M)}}^\ast$-increasing. Moreover, 
		the function $\q\mapsto\Psi(\q)$ defined on $Q_2$ is $Q_2^\ast$-increasing.
	\end{lemma}
	\begin{proof}  
		Fix $M\in\N$. Let $q_0,q_1\in C_\leq^{(M)}$ be such that 
		for all $p\in C_\leq^{(M)}$, $p\cdot q_1 \geq p\cdot q_0$.
		Our goal is to prove that $\Psi(q_1)\geq \Psi(q_0)$. We first argue that we can assume that $q_0,q_1\in C_<^{(M)}$. Fix $\varepsilon>0$. For all $q_0,q_1\in C_\leq^{(M)}$, define $\tilde q_0$ and $\tilde q_1$ by $\tilde q_{i,k} =  q_{i,k}+\frac{k\varepsilon}{M+1}$ for $i=0,1$ and $k=0,\ldots,M+1$. By the definition, one sees that $\tilde q_0,\tilde q_1\in C_<^{(M)}$. Suppose that $\Psi$ defined on $C_<^{(M)}$ is ${C_\leq^{(M)}}^\ast$-increasing. Note that by our construction, we also have for all $p\in C_\leq^{(M)}$, $p\cdot \tilde q_1 \geq p\cdot \tilde q_0$, so we have $\Psi(\tilde q_1)\geq \Psi(\tilde q_0)$. On the other hand, by the inequality $(x-y)_+\leq \abs{x-y}$ for all $x,y\in \R$, we have $\abs{\Psi( q_i)-\Psi(\tilde q_i)}<\varepsilon$ for all $i=0,1$. Therefore, we have $\Psi(q_1)\geq \Psi(q_0)-2\varepsilon$. Since $\varepsilon>0$ is arbitrarily chosen, the proof is completed.
		
		From now on, we assume that $q_0,q_1\in C_<^{(M)}$.
		For each $\lambda \in [0,1]$,
		we set $q_\lambda = \lambda q_1 + (1-\lambda) q_0$. 
		Since $C_M^{(M)}$ is convex, $q_\lambda\in C_<^{(M)}$.
		% We set  
		% \begin{align}\label{eq:PsiCIncrDefPsiTilde}
			%     \tilde\Psi \colon C_<^{(M-1)} &\to \R,\notag\\
			%     (q_0,\dots,q_{M-1}) &\mapsto \Psi\kl{ \textstyle\sum_{j=0}^M q_j \1_{\left[\frac{j}{M}, \frac{j+1}{M}\right)}},
			% \end{align}
		% recalling that 
		% \begin{align}
			%     C_<^{(M-1)} = \cbrac{q=(q_0,\dots,q_{M-1})\in \R^{M} \colon 0 < q_0 < \dots < q_{M-1}}.
			% \end{align}
		By the fundamental theorem of calculus, denoting by $\partial_\lambda\Psi(q_{\lambda})$ the derivative of the function $\lambda\mapsto \Psi(q_\lambda)$ evaluated at $\lambda$, we have
		\begin{align}\label{eq:PsiCIncr20}
			\Psi(q_1) - \Psi(q_0)
			%     &=
			%     \tilde\Psi\kl{q^{(1)}_0,\dots,q^{(1)}_{M-1}} - \tilde\Psi\kl{q^{(0)}_0,\dots, q^{(0)}_{M-1}}
			=
			\int_0^1 \partial_\lambda\Psi(q_{\lambda}) \,\d{\lambda}.
		\end{align}
		Here, by the chain rule, 
		\begin{align}\label{eq:PsiCIncr21}
			\partial_\lambda\Psi(q_{\lambda}) = (q_1-q_0)\cdot \nabla\Psi(q_{\lambda}).
		\end{align}
		We have the following lemma on the monotonicity of gradient, whose proof will be provided right below the current proof.
		\begin{lemma}\label{lem:GradientPsiIncreasing}
			Let $M\in \N$. For all $q\in C_<^{(M)}$, the gradient $\nabla \Psi (q)$ lies in $C_\leq^{(M)}$.
		\end{lemma}
		Then, by Lemma~\ref{lem:GradientPsiIncreasing}, \eqref{eq:PsiCIncr21} and our assumption of $q_0$ and $q_1$, the integrand in \eqref{eq:PsiCIncr20} is non-negative. Therefore, we obtain that
		\begin{align*}
			\Psi(q_1) - \Psi(q_0)
			=
			\int_0^1 \partial_\lambda\Psi(q_{\lambda}) \,\d{\lambda} \geq 0.
		\end{align*}
		
		It remains to prove the second statement of this lemma. Let $\q_0,\q_1 \in Q_2$ be such that
		\begin{align}\label{eq:Q2StarIncreasing2}
			\int_0^1 \pp(u) (\q_1(u)-\q_0(u)) \,\d u \geq 0
			\quad \text{for all} \quad
			\pp\in Q_2,
		\end{align}
		and our goal is to show that \eqref{eq:Q2StarIncreasing2} implies $\Psi(\q_1) \geq \Psi(\q_0)$.
		By Lemma~\ref{lem:DensityQM} and adopting the similar argument in the first paragraph of this proof, there exists $M\in\N$ such that $q_0,q_1\in C_<^{(M)}$ and for all $i=0,1$, we have $\norm{\tilde \q_i - \q_i}_2 < \varepsilon$
		where $\tilde \q_i(u)=\sum_{k=0}^M\tilde q_{i,k}\1_{u\in[k/(M+1),(k+1)/(M+1))}$. Then, by Lipschitz continuity of $\Psi$ provided by Corollary~\ref{cor:convexLipschitz}, we have
		\begin{align*}
			\Psi(\q_1)-\Psi(\q_0)
			\geq -2\varepsilon + \Psi(\tilde\q_1)-\Psi(\tilde\q_0) = -2\varepsilon + \Psi(\tilde q_1) - \Psi(\tilde q_0).
		\end{align*}
		As in \eqref{eq:PsiCIncr20}, we can write 
		\begin{align*}
			\Psi(\tilde q_1) - \Psi(\tilde q_0) 
			= 
			\int_0^1 \partial_\lambda\Psi(\tilde q_{\lambda}) \,\d{\lambda}
		\end{align*}
		where $\tilde{q}_\lambda = (1-\lambda) \tilde{q}_0 + \lambda \tilde{q}_1$. Similar to \eqref{eq:PsiCIncr21}, we have
		\begin{align*}
			\partial_\lambda\Psi(\tilde q_{\lambda})
			=
			(\tilde{q}_1-\tilde{q}_0)\cdot \nabla\Psi(\tilde{q}_{\lambda})
			=
			(M+1)\int_0^1 (\tilde{\q}_1(u)-\tilde{\q}_0(u))\pp_\lambda(u) \,\d{u}
		\end{align*}
		where $\pp_\lambda(u) = \sum_{i=0}^M \partial_i\Psi(\tilde q_\lambda)\1_{u\in[k/(M+1),(k+1)/(M+1))}$. Note that $\q_\lambda\in Q_2$ and $|\partial_i\Psi(\tilde q_\lambda)| \leq \frac{1}{M + 1}$ for all $i$ (see \eqref{eq:partial_q_Psi}),   so 
		\begin{align*}
			\norm{\pp_\lambda}_2 \leq \frac{1}{M+1}.
		\end{align*}
		Therefore, by the Cauchy--Schwarz inequality and the assumptions on $\q_0,\q_1$, we have
		\begin{align*}
			&\int_0^1 (\tilde{\q}_1(u)-\tilde{\q}_0(u))\pp_\lambda(u) \,\d{u}
			\nonumber \\
			&=
			\int_0^1 (\tilde{\q}_1(u)-\q_0(u))\pp_\lambda(u) \,\d{u}
			\nonumber \\
			&+
			\int_0^1 (\q_1(u)-\q_0(u))\pp_\lambda(u) \,\d{u}
			+
			\int_0^1 (\q_1(u)-\tilde{\q}_0(u))\pp_\lambda(u) \,\d{u}
			\geq \frac{-2\varepsilon}{M+1}.
		\end{align*}
		In conclusion, we have
		\begin{align*}
			\Psi(\q_1)-\Psi(\q_0) \geq -4\varepsilon.
		\end{align*}
		Since $\varepsilon>0$ is arbitrary, the proof is completed.
	\end{proof}
        We now prove Lemma~\ref{lem:GradientPsiIncreasing}.
	\begin{proof}[Proof of Lemma~\ref{lem:GradientPsiIncreasing}]
		Let $M\in\N$ and $q=(q_0,\dots,q_M) \in C_<^{(M)}$. Defining 
        \[
        \q\coloneqq\sum_{j=0}^{M} q_j \1_{\left[\frac{j}{M+1}, \frac{j+1}{M+1}\right)}\in Q^{(M)}\subseteq Q_1,
        \] applying the formula for $\Psi(\q)$ provided by Theorem~\ref{thm:InitCond} yields
		\begin{align}
			\Psi(q)
			=
			\Psi(\q)
			&= 
			-\log 2 + \int_0^1 \brac{\q(u) - \sfrac{\log 2}{u^2}}_+ \d u
			\notag\\
			&=
			-\log 2 + \int_{x_*(\q)}^1 \q(u) - \sfrac{\log 2}{u^2} \,\d u
			=
			- \sfrac{\log 2}{x_*(\q)} + \int_{x_*(\q)}^1 \q(u) \,\d u,
		\end{align}
		where $x_*(\q) = \inf\gkl{x \in (0,1]\colon \q(x) > \sfrac{\log 2}{x^2}}$, setting $x_*(\q) = 1$ if $q_{M} \leq \log 2$. We compute the partial derivatives of $\Psi$ at $q=(q_0,\dots,q_{M})$:
		\begin{enumerate}
			\item If $x_*(\q) = 1$, then $\Psi(q_0,\dots,q_{M}) = - \log 2$, so 
			$\partial_{q_j}\Psi(q) = 0$ for all $j = 0,\dots, M$.
			\item If $x_*(\q) = \sfrac{k}{M+1}$ for some $k\in \gkl{1,\dots,M}$, then
			\begin{align*}
				\Psi(q)
				=-(\log 2)\sfrac{M+1}{k} + \sum_{j=k}^{M} \sfrac{q_j}{M+1}.
			\end{align*}
			By differentiating the formula above, we obtain 
			\begin{align*}
				\partial_{q_j}\Psi(q_0,\dots,q_{M}) =
				\begin{cases}
					0, & \text{ if } j\in \{0,\ldots,k-1\},\\
					\sfrac{1}{M+1}, & \text{ if } j\in \{k,\ldots,M\}.
				\end{cases}
			\end{align*}
			\item If $x_*(\q) \in \kl{\sfrac{k}{M+1},\sfrac{k+1}{M+1}}$ 
			for some $k\in \gkl{1,\dots,M}$, then
			$x_*(\q) =\sqrt{\frac{\log 2}{q_k}}$. This gives
			\begin{align}
				\Psi(q)
				&=
				- \sqrt{q_k \log 2} + \int_{\sqrt{\frac{\log 2}{q_k}}}^1 \q(u) \,\d u \notag\\
				&=
				- \sqrt{q_k \log 2} + q_k \kl{\sfrac{k+1}{M+1} - \sqrt{\sfrac{\log 2}{q_k}}} + \sum_{j=k+1}^{M}   \sfrac{q_j}{M+1}\notag\\
				&= - 2\sqrt{q_k \log 2} + q_k \sfrac{k+1}{M+1} + \sum_{j=k+1}^{M}   \sfrac{q_j}{M+1}.
			\end{align}
			Thus,
			\begin{align}\label{eq:partial_q_Psi}
				\partial_{q_j}\Psi(q) =
				\begin{cases}
					0, & \text{ if } j\in \{0,\ldots,k-1\},\\
					\sfrac{k+1}{M+1} - \sqrt{\sfrac{\log 2}{q_k}} = \sfrac{k+1}{M+1} - x_*(\q), & \text{ if } j= k,\\
					\sfrac{1}{M+1}, & \text{ if } j \in \{k+1,\ldots,M\}.
				\end{cases}
			\end{align}
		\end{enumerate}
		In particular, in each case, we have $\nabla \Psi(q) \in C_{\leq}^{(M)}$ as desired.
		%\begin{equation*}
		%%\label{eq:GradientPsiInCone}
		%    \nabla \Psi(q) \in C_{\leq}^{(M)}
		%\end{equation*}
		%as desired.
	\end{proof}
	
	Next, we state and prove a Gaussian comparison lemma that allows us to assume that $A$ is convex and Lipschitz in this section, which gives us the required regularity conditions to apply Theorem 1.2 in \cite{chen_hamilton-jacobi_2023}. Recall the definition of the enriched free energy in \eqref{eq:enrichedFE} for a generic covariance $A$: 
	\begin{align}\label{eq:FE_Cov_A}
		F_N\left(
		(H_N^{A}(\sigma))_{\sigma\in\{-1,1\}^N}
		\right)
		&= - \frac{1}{N} \E
		\log\kl{\textstyle\sum_{\alpha \in \N^M} v_\a \sum_{\sigma \in \{-1,1\}^N}   
			\exp\bkl{H_N^{A}(t,q,\s,\a)}
		}.
	\end{align}
	We use the notation $ F_N^{A} := F_N\left(
	(H_N^{A}(\sigma))_{\sigma\in\{-1,1\}^N}
	\right)$ to denote the free energy associated with the covariance function $A$.
	\begin{lemma}
		\label{lem:GP}
		Fix $M,N\in \N$. Let $A_1$ and $A_2$ be right-continuous and increasing functions such that $A_1(0)=A_2(0)=0$ and $A_1(1)=A_2(1)=1$.
		Suppose that $A_1(x)\leq A_2(x)$ for all $x\in [0,1]$.
		Then, for all $(t,q)\in \R_{\geq 0}\times C_\leq^{(M)}$, we have
		\begin{align*}
			F_N^{A_1}(t,q)\leq F_N^{A_2}(t,q).
		\end{align*}
	\end{lemma}
	\begin{proof}
		Throughout the proof, we fix the randomness in $v_\alpha$ and $(Y_q(\sigma,\alpha) )$ and will prove a bound conditionally on these random objects. Let $\tilde F_N$ be the random free energy as defined in \eqref{eq:FE_Cov_A} without the expected value. 
		So if we differentiate $\tilde F_N$ with respect to $H_N^A(\sigma_1)$, we obtain
		\begin{align*}
			\pdv{\tilde F_N}{H_N^A(\sigma_1)}
			=
			-\frac{1}{N}
			\frac{
				\textstyle
				\sum_{\alpha \in \N^M} v_\a    
				\exp\bkl{H_N^{A}(t,q,\s_1,\a)}
			}{
				\textstyle\sum_{\alpha \in \N^M} v_\a \sum_{\sigma \in \{-1,1\}^N}   
				\exp\bkl{H_N^{A}(t,q,\s,\a)}
			}.
		\end{align*}
		Now, for $\s_2\neq \s_1$, we have
		\begin{align*}
			\pdv{\tilde F_N}{H_N^A(\sigma_1)}{H_N^A(\sigma_2)}
			&=
			\frac{1}{N}
			\frac{
				\textstyle
				\kl{
					\sum_{\alpha \in \N^M} v_\a    
					\exp\bkl{H_N^{A}(t,q,\s_1,\a)}
				}
				\kl{
					\sum_{\alpha \in \N^M} v_\a    
					\exp\bkl{H_N^{A}(t,q,\s_2,\a)}
				}
			}{
				\kl{
					\textstyle\sum_{\alpha \in \N^M} v_\a \sum_{\sigma \in \{-1,1\}^N}   
					\exp\bkl{H_N^{A}(t,q,\s,\a)}
				}^2}
			\nonumber \\
			&\geq 0.
		\end{align*}
		Then, by Lemma~3.2~in~\cite{BK2} and by the fact the covariance of $(H_N^{A_1}(\sigma))_{\sigma\in\{-1,1\}^N}$ is bounded from above by the one of $(H_N^{A_2}(\sigma))_{\sigma\in\{-1,1\}^N}$ as in the statement of the lemma, we obtain
		\begin{align*}
			&\cEX{
				\tilde{F}_N\left(
				(H_N^{A_1}(\sigma))_{\sigma\in\{-1,1\}^N}
				\right)
			}{ v_\alpha, (Y_q(\sigma,\alpha) )_{ \sigma,\alpha }  }
            \nonumber \\
			&\leq
			\cEX{
				\tilde{F}_N\left(
				(H_N^{A_2}(\sigma))_{\sigma\in\{-1,1\}^N}
				\right)
			}
			{v_\alpha, (Y_q(\sigma,\alpha) )_{ \sigma,\alpha }},
		\end{align*}
		which yields
		\begin{align*}
			F_N^{A_1}(t,q)\leq F_N^{A_2}(t,q),
		\end{align*}
		and the proof is completed.
	\end{proof}
	
	Let $A$ be a generic covariance function $A$ satisfying the assumption in Theorem~\ref{thm:main}. Then, $A$ is upper bounded by $\mathrm{id}(x)=x$. For the lower bound, we use the following lemma.
	\begin{lemma}
		\label{lem:A low bnd}
		Suppose that $A$ is non-negative and increasing and has a finite left-derivative at $1$.
		There exists $\tilde{A}$ that is increasing, convex and Lipschitz with $\tilde{A}(0)=0$ and $\tilde{A}(1)=A(1)$ such that $\tilde{A}(x)\leq A(x)$ for all $x\in [0,1]$.
	\end{lemma}
	\begin{proof}
		By applying the Taylor expansion of $A$ at $1$, we have 
		\begin{align*}
			A(x) = A(1) + A'(1)(x-1) + \epsilon(x)(x-1), \quad x\uparrow 1
		\end{align*}
		where $A'(1)$ stands for the left-derivative of $A$ at $1$ and $\lim_{x\uparrow 1}\epsilon(x) = 0$. In particular, there exists $\delta \in (0 ,\frac{2}{3 A'(1)} )$ such that $|\epsilon(x)|< A'(1)/2$ for all $x\in (1-\delta,1]$. Thus, for $x\in (1-\delta,1]$, we have
            \begin{align*}
                A(x) \geq A(1) + \frac{3A'(1)}{2}(x-1) \geq A(1) + \frac{1}{\delta} (x-1).
            \end{align*}
        Therefore, we take 
        \[
        \tilde{A}(x)=(A(1) + \delta^{-1} (x-1))_+ = \begin{cases}
            0, & x \in [0, 1-\delta],\\
            A(1) + \frac{1}{\delta} (x-1),  & x \in (1-\delta, 1],
        \end{cases} 
        \]
        and the proof is completed.
	\end{proof}

        \begin{remark}
        The condition on the finite left-derivative in Lemma~\ref{lem:A low bnd} is where the corresponding hypothesis in Theorem~\ref{thm:main} is used. Steps towards removing this condition are explained in Section~\ref{sec:rem case}.
        \end{remark}
	Then, by Lemma~\ref{lem:GP} and Lemma~\ref{lem:A low bnd},
	for all $(t,q)\in \R_{\geq 0}\times C_\leq^{(M)}$, we have
	\begin{align}
		\liminf_{N\rightarrow\infty} F_N^{\tilde{A}}(t,q)\leq
		\limsup_{N\rightarrow\infty} F_N^{A}(t,q)\leq
		\limsup_{N\rightarrow\infty} F_N^{\mathrm{id}}(t,q). \label{eq:sandwich.0}
	\end{align}
	By \eqref{eq:sandwich.0}, it is sufficient to show that for $A$ Lipschitz and convex, for all $(t,q)\in\R_{\geq 0}\times C_\leq^{(M)}$, $\lim_{N\rightarrow\infty} F_N^A(t,q)$ exists and it is given by a formula that is independent of the choice of $A$. Thus, from now on, we assume that $A$ is Lipschitz and convex.

	We start with a proposition that states the limit infimum of $F_N^A(t,q)$ is the viscosity supersolution of a Hamilton--Jacobi equation. The proof of the proposition is postponed to Section~\ref{sec:supersolution}. 
	\begin{proposition}
		\label{prop:supersolution}
		Fix $M\in\N$. Define $\HHH_M^A:\R^{M+1}\rightarrow [0,\infty)$ by 
		\begin{align*}
			\HHH_M^A(q) 
			= 
			\inf
			\left\{
			\frac{1}{M+1}
			\sum_{m=0}^M \ar(p_m)
			\,:\,
			p=(p_0,\ldots,p_m)\in C_\leq^{(M)} \cap (q+ {C_\leq^{(M)}}^\ast)
			\right\},
		\end{align*}
		where 
		the extension $\ar\colon \R \to \R_{\geq 0}$ of $A$ is defined by
		\begin{align}
			x \mapsto \begin{cases}
				0, & \text{ if } x <0,\\
				A(x), & \text{ if } x \in [0,1],\\
				1 + |A|_{\mathrm{Lip}}(x-1), & \text{ if } x >0.\\
			\end{cases}
		\end{align}
		Then, for all $(t,q)\in \R_{\geq 0}\times C_\leq^{(M)}$, the limit infimum $\liminf_{N\rightarrow\infty}F_N^A(t,q)$ is a viscosity supersolution of the Cauchy problem of the Hamilton--Jacobi equation
		\begin{align*}
			\begin{cases}
				\partial_t f - \HHH_M^A(\nabla f) = 0, & \forall\,(t,q)\in\R_{+}\times C_\leq^{(M)}, \\
				f(0,q) = \Psi(q), & \forall\, q\in C_\leq^{(M)}.
			\end{cases}
		\end{align*}
		%where $\q_M = \sum_{m=0}^{M} q_m \1_{[m/(M+1), (m+1)/(M+1))}$.
	\end{proposition}
	\begin{remark}
		Note that $\ar$ is a convex and Lipschitz function on $\R$ which has the same Lipschitz constant as $A$. 
		The choice of extension $\ar$ might seem to be arbitrary, but as the proof continues, it will be clear that the formula only depends on $x\in [0,1]$. 
	\end{remark}
	
	We state another proposition which provides an upper bound of the limiting free energy. The proof is postponed to Section~\ref{sec:FI subsol}.
	\begin{proposition}
		\label{prop:subsolution}
		For all $M\in\N$, the limit supremum $(t,q)\mapsto \limsup_{N\rightarrow\infty}F_N^{\mathrm{id}}(t,q)$ is a viscosity subsolution of the Hamilton--Jacobi equation
		\begin{align*}
			\begin{cases}
				\partial_t  f - \div f = 0, & \forall\, (t,q)\in\R_{+}\times C_\leq^{(M)}, \\
				f(0,q) = \Psi(q), & \forall\, q\in C_\leq^{(M)}.
			\end{cases}
		\end{align*}
	\end{proposition}
	
	We proceed with the proof of Theorem~\ref{thm:main}. For $M\in\N$, by Theorem 1.2 in \cite{chen_hamilton-jacobi_2023}, the equations in Proposition~\ref{prop:supersolution} and Proposition~\ref{prop:subsolution} admit unique viscosity solutions, denoted by $f^A_M$ and $f^{\mathrm{id}}_M$, respectively. Moreover, by Lemma~\ref{lem:GP} and the comparison principle provided by (2) in Theorem 1.2 in \cite{chen_hamilton-jacobi_2023}, we have that for all $(t,q)\in\R_{+}\times C_\leq^{(M)}$,
	\begin{align}
		f_M^A(t,q)\leq \liminf_{N\rightarrow\infty} F_N^A(t,q)
		\leq \limsup_{N\rightarrow\infty} F_N^{\mathrm{id}}(t,q)
		\leq
		f_M^{\mathrm{id}}(t,q). \label{eq:sandwich.2}
	\end{align}
	Now, let $(t,\q)\in\R_{+}\times Q_2$. Let $q_M\in C_\leq^{(M)}$ be defined by
	\begin{align*}
		q_M^i = (M+1)\int_{i/(M+1)}^{(i+1)/M+1}\q(u)\,\d{u}
	\end{align*}
	for $i=0,\ldots,M+1$. Then by (3) of Theorem 4.7 in \cite{chenHamiltonJacobiEquations2025}, the limits of $f^A_M(t,q_M)$ and $f^{\mathrm{id}}_M(t,q_M)$ in the uniform topology 
	\begin{align*}
		f^A(t,\q) \coloneqq \lim_{M\rightarrow\infty} f^A_M(t,q_M)
		\quad \text{and}\quad 
		f^{\mathrm{id}}(t,\q) \coloneqq \lim_{M\rightarrow\infty} f^{\mathrm{id}}_M(t,q_M)
	\end{align*}
	are given by the Hopf formulae
	\begin{align}
		f^A(t,\q) 
		&= \sup_{\pp\in Q_\infty}\inf_{\yy\in Q_\infty}
		\left\{
		\Psi(\yy) + \int_0^1 \pp(u)(\q(u)-\yy(u))\,\d{u}
		+ t\int_0^1 A_{\mathrm{ext}}(\pp(u))\,\d{u}
		\right\},
		\label{eq:Hopf.A} \\
		f^{\mathrm{id}}(t,\q) 
		&= 
		\sup_{\pp\in Q_\infty}\inf_{\yy\in Q_\infty}
		\left\{
		\Psi(\yy) + \int_0^1 \pp(u)(\q(u)-\yy(u))\,\d{u}
		+ t\int_0^1 \pp(u)\,\d{u}
		\right\}.
		\label{eq:Hopf.id}
	\end{align}
	We argue in the rest of the proof that both \eqref{eq:Hopf.A} and \eqref{eq:Hopf.id} are equal to Formula~\ref{eq:main}. Note that by \eqref{eq:sandwich.2}, this also implies the limit of $F_N(t,\q)$ exists and equals Formula~\ref{eq:main}, which completes the proof.
	
	Recall the convex dual $\Psi^*$ of $\Psi$ is defined by
	\begin{align}\label{eq:aslkfjeCD}
		\Psi^*(\pp) = \sup_{\q\in Q_2}\kl{\int_0^1 \pp(u) \q (u)\,\d u - \Psi(\q)},
	\end{align}
	for all $\pp\in Q_2$. The formula for $\Psi$ in Theorem~\ref{thm:InitCond} also allows us to explicitly compute $\Psi^*$.
	\begin{proposition}\label{prop:HJCremPsiCOnvexDual}
		For $\Psi\colon Q_1 \to \R$ as in Theorem~\ref{thm:InitCond} and $\pp \in Q_2$,
		\begin{align}\label{eq:HJCremPsiConvexDual2}
			\Psi^*(\pp) = 
			\begin{cases}
				\infty, & \text{ if } \norm{\pp}_1 \geq 1 \
                \text{ or } \norm{\pp}_\infty > 1 \\%\text{ or there exists } u \in [0,1) \text{ with } \pp(u) > 1,\\
				\frac{\log 2}{1-\norm{\pp}_1}, & \text{ otherwise.}
			\end{cases}
		\end{align}
	\end{proposition}
	The proof of Proposition~\ref{prop:HJCremPsiCOnvexDual} is deferred to Section~\ref{sec:proof of Prop:HJConvexDual}.	By Proposition~\ref{prop:HJCremPsiCOnvexDual}, the Hopf formula in \eqref{eq:Hopf.A} can be rewritten as
	\begin{align}
		f^A(t,\q) 
		&= \sup_{\substack{\pp\in Q_\infty \\ \norm{\pp}_\infty\leq 1,\norm{\pp}_1<1}} 
		\left\{
		\int_0^1 \pp(u)\q(u)\,\d{u}
		+ t\int_0^1 A_{\mathrm{ext}}(\pp(u))\,\d{u}
		-
		\frac{\log 2}{1-\norm{\pp}_1}
		\right\}
		\nonumber \\
		&=
		\sup_{\substack{\pp\in Q_\infty \\ \norm{\pp}_\infty\leq 1,\norm{\pp}_1<1}} 
		\left\{
		\int_0^1 \pp(u)\q(u)\,\d{u}
		+ t\int_0^1 A(\pp(u))\,\d{u}
		-
		\frac{\log 2}{1-\norm{\pp}_1}
		\right\}, \label{eq:HJ.A.2}
	\end{align}
	which in particular, shows that the formula does not depend on the choice of $A_{\mathrm{ext}}$.
	Similarly, \eqref{eq:Hopf.id} can be rewritten as 
	\begin{align}
		f^{\mathrm{id}}(t,\q)
		=
		\sup_{\substack{\pp\in Q_\infty \\ \norm{\pp}_\infty\leq 1,\norm{\pp}_1}<1} 
		\left\{
		\int_0^1 \pp(u)\q(u)\,\d{u}
		+ t\int_0^1 \pp(u)\,\d{u}
		-
		\frac{\log 2}{1-\norm{\pp}_1}
		\right\}. \label{eq:HJ.id.2}
	\end{align}
	We prove that both \eqref{eq:HJ.A.2} and \eqref{eq:HJ.id.2} equal \eqref{eq:main}. 
	We start with the lower bound for \eqref{eq:HJ.A.2}. Let $\lambda \in [0,1)$.
	We set $\pp^{(\lambda)} = \1_{[1-\lambda,1)}$. In the case $\lambda=0$, this is to be understood as $\pp^{(0)}\equiv 0$. 
	Since $\pp^{(\lambda)} \in Q_2$ with $\norm{\pp}_\infty = 1$ and $\norm{\pp}_1=\lambda$, we get
	\begin{align}\label{eq:compareViscosFE1}
		\sup_{\substack{\pp \in Q_\infty,\\\norm{\pp}_\infty\leq 1, \norm{\pp}_1 = \lambda}} \kl{t \int_0^1 A(\pp(u)) \,\d u}
		\geq t \int_0^1 A(\pp^{(\lambda)}(u)) \,\d u = \lambda t,
	\end{align}
	using that $A(0)=0$ and $A(1) = 1$. Therefore, $f^A(t,\q)$ admits the lower bound
	\begin{align*}
		f^A(t,\q)
		\geq 
		\sup_{\lambda\in [0,1)}
		\left\{
		\lambda t - \int_{1-\lambda}^1 \q(u) \d u - \frac{\log 2}{1-\lambda}
		\right\}.
	\end{align*}
	It remains to prove an upper bound for \eqref{eq:HJ.id.2}. Note that for any $\lambda \in [0,1]$, $\pp \in Q_\infty$ such that $\norm{\pp}_\infty \leq 1$ and $\norm{\pp}_1 = \lambda$ and $\q\equiv q$ constant, we have
	\begin{align}
		\int_0^{1-\lambda} \pp(u)q\d{u}
		=
		\int_{1-\lambda}^1 (1-\pp(u))q\d{u}. \label{eq:constant q}
	\end{align}
	Now, for any $\q$ non-negative and increasing, we have
	\begin{align}
		\int_0^{1-\lambda} \pp(u)\q(u)\d{u}
		\leq 
		\int_0^{1-\lambda}\pp(u)\q(1-\lambda)\d{u}
		\label{eq:general q.1}
	\end{align}
	and 
	\begin{align}
		\int_{1-\lambda}^1 (1- \pp(u))\q(u)\d{u}
		\geq 
		\int_{1-\lambda}^1(1-\pp(u))\q(1-\lambda)\d{u}.
		\label{eq:general q.2}
	\end{align}
	Therefore, combining \eqref{eq:constant q}, \eqref{eq:general q.1} and \eqref{eq:general q.2} yields
	\begin{align*}
		\int_0^{1-\lambda} \pp(u)\q(u)\d{u}
		\leq 
		\int_{1-\lambda}^1 (1- \pp(u))\q(u)\d{u},
	\end{align*}
	which is equivalent to 
	\begin{align}
		\int_0^1 \pp(u)\q(u)\d{u}
		\leq
		\int_{1-\lambda}^1 \q(u)\d{u}.
	\end{align}
	Therefore, $f^I(t,\q)$ admits the upper bound
	\begin{align*}
		f^I(t,\q)
		\leq 
		\sup_{\lambda\in [0,1)}
		\left\{
		\lambda t - \int_{1-\lambda}^1 \q(u) \d u - \frac{\log 2}{1-\lambda}
		\right\},
	\end{align*}
	which also yields that $f^A(t,\q)=f^{\mathrm{id}}(t,\q)$ and they both equal to Formula \eqref{eq:main}.
	
	\subsection{Proof of Proposition~\ref{prop:supersolution}} \label{sec:supersolution}
	
	We adopt a similar argument as in \cite{Mourrat2023FreeEnergyUpperBound}. Let $(t,q)\mapsto \tilde F(t,q)$ be a subsequential limit of $F_N$. For simplicity, we omit specifying the subsequence along which the convergence of $(t, q) \mapsto F_N(t, q)$ to $\tilde F$ holds.
	
	Let $\phi(t,q)\in C^\infty((0,\infty)\times C_\leq^{(M)})$ be a smooth test function such that $\tilde F - \phi$ has a local minimum at $(t_\star,q_\star)\in (0,\infty)\times C_\leq^{(M)}$. Our goal is to show that 
	\begin{align}
		(\partial_t \phi - \HHH_M^A(\nabla\phi))(t_\star,q_\star) \geq 0.
		\label{eq:sup sol}
	\end{align}
	Define for every $(t,q)\in (0,\infty)\times C_\leq^{(M)}$, $\tilde \phi(t,q) = \phi(t,q) - (t-t_\star)^2 - \abs{q-q_\star}^2$. 
	We can find $(t_N, q_N) \in (0, \infty) \times C_\leq^{(M)}$ such that, for $N$ sufficiently large, it is a local minimum of $F_N - \tilde\phi$, and
	\[
	\lim_{N \to \infty} (t_N, q_N) = (t_\star, q_\star).
	\]

	To prove \eqref{eq:sup sol}, we claim that for all $(t,q)\in (0,\infty)\times C_\leq^{(M)}$, we have 
	\begin{align}
		(\partial_t F_N - \HHH_M^A(\nabla F_N))(t,q) \geq 0. \label{eq:F_N sup sol}
	\end{align}
	Also, we claim that
	\begin{align}
		\HHH_M^A(\nabla F_N)(t_N,q_N) \geq \HHH_M^A(\nabla\tilde \phi)(t_N,q_N). \label{eq:Hamcomp}
	\end{align}
	If \eqref{eq:F_N sup sol} and \eqref{eq:Hamcomp} both hold, then by the fact that $\partial_t F_N(t_N,q_N) = \partial_t \phi(t_N,q_N)$ (as $t_N$ is a critical point and an interior point) and the fact that $\tilde\phi$ is smooth and $\HHH_M^A$ is continuous, we have
	\begin{align*}
		(\partial_t \phi - \HHH_M^A(\nabla \phi))(t_\star,q_\star)
		&=
		\liminf_{N\rightarrow\infty}
		(\partial_t \tilde \phi - \HHH_M^A(\nabla \tilde \phi))(t_N,q_N)
		\nonumber \\
		&\geq 
		\liminf_{N\rightarrow\infty}
		(\partial_t F_N - \HHH_M^A(\nabla F_N))(t_N,q_N)
		\geq 0.
	\end{align*}
	We start with proving \eqref{eq:F_N sup sol}. By Proposition~\ref{prop:PartialDeriv}, we have
	\begin{align}
		&(\partial_t F_N - \HHH_M^A(\nabla F_N))(t,q)
		\nonumber \\
		&=
		\E\bbga{A\kl{\sfrac{\s \wedge \tilde\s}{N}}}_{t,\q,2}
		-
		\sum_{m=1}^M
		\frac{1}{M}
		A
		\left(
		\E\bbga{ \1_{\a\wedge \tilde\a = k}(\a,\tilde\a)\, \sfrac{\s \wedge \tilde\s}{N}}_{t,\q,2}
		\right)
		\nonumber \\
		&\geq 
		\E\bbga{A\kl{\sfrac{\s \wedge \tilde\s}{N}}}_{t,\q,2}
		-
		\E\bbga{A\kl{\sfrac{\s \wedge \tilde\s}{N}}}_{t,\q,2} 
		\label{eq:jensen}
		\nonumber \\
		&=0,
	\end{align}
	where we apply Jensen's inequality and the convexity of $A$ to obtain \eqref{eq:jensen}. 
	
	It remains to show \eqref{eq:Hamcomp}, where we follow the same argument as in (5.16) in \cite{Mourrat2023FreeEnergyUpperBound}. By Lemma~\ref{lem:increasing}, we know that $\nabla \tilde{F}_N\in C_\leq^{(M)}$. Thus, by the definition of $\HHH_M^A$, it suffices to show that
	\begin{align}
		\nabla F_N - \nabla\tilde \phi \in  {C_\leq^{(M)}}^*. \label{eq:dual_cone_A}
	\end{align}
	Since $(t_N,q_N)$ is a local minimum of $F_N - \phi$, we have for every $q'\in C_\leq^{(M)}$ that 
	\begin{align}
		(q' - q_N)\cdot (\nabla F_N - \nabla\tilde \phi)(t_N,q_N) \geq 0. \label{eq:grad_inner_prod}
	\end{align}
	Since $C_\leq^{(M)}$ is a convex cone, we can substitute $q' - q_N$ by $q'$ in \eqref{eq:grad_inner_prod}.
	By (4.11) of Lemma~4.4 in \cite{Mourrat2023FreeEnergyUpperBound}, we conclude that \eqref{eq:dual_cone_A} is true.
	
	\subsection{Proof of Proposition~\ref{prop:subsolution}}
	\label{sec:FI subsol}
	
	Let $(t,q)\mapsto \tilde F^{\mathrm{id}}(t,q)$ be a subsequential limit of $F_N^{\mathrm{id}}$. For simplicity, we omit specifying the subsequence along which the convergence of $(t, q) \mapsto F_N^{\mathrm{id}}(t, q)$ to $\tilde F^\mathrm{id}$ holds.
	
	Let $(t_\star,\q_\star)\in (0,\infty)\times Q^{(M)}$ and $\phi\in C^\infty((0,\infty)\times C_\leq^{(M)})$ be such that $\tilde F^{\mathrm{id}} - \phi$ has a local maximum at $(t_\star,q_\star)$. We claim that
	\begin{align*}
		(\partial_t\phi - \div\phi )(t_\star,q_\star)
		\leq
		0,
	\end{align*}
	where $\div\phi$ is the divergence of $\phi$.
	Define for every $(t,q)\in (0,\infty)\times C_\leq^{(M)}$, $\tilde \phi(t,q) = \phi(t,q) + (t-t_\star)^2 + \abs{q-q_\star}^2$. 
	We can find $(t_N, q_N) \in (0, \infty) \times C_\leq^{(M)}$ such that, for $N$ sufficiently large, it is a local maximum of $F_N^{\mathrm{id}} - \tilde\phi$, and
	\[
	\lim_{N \to \infty} (t_N, q_N) = (t_\star, q_\star).
	\]
	By Proposition~\ref{prop:PartialDeriv}, we have that for all $(t,q)\in  (0,\infty)\times Q^{(M)}$,
	\begin{align}
		&(\partial_t F_N^{\mathrm{id}} - \div F_N^\mathrm{id})(t,q)
		\nonumber \\
		&=
		\E\bbga{\sfrac{\s \wedge \tilde\s}{N}}_{t,\q,2}
		-
		\sum_{m=1}^M
		\frac{1}{M}
		\E\bbga{ \1_{\a\wedge \tilde\a = k}(\a,\tilde\a)\, \sfrac{\s \wedge \tilde\s}{N}}_{t,\q,2}
		\nonumber \\
		&=
		\E\bbga{\sfrac{\s \wedge \tilde\s}{N}}_{t,\q,2}
		-
		\E\bbga{\sfrac{\s \wedge \tilde\s}{N}}_{t,\q,2} 
		\nonumber \\
		&=0.
	\end{align}
	On the other hand, since $(t_N,q_N)$ is a local maximum of $F_N^{\mathrm{id}} - \tilde \phi$, we have for every $q'\in C_\leq^{(M)}$,
	\begin{align}
		(q' - q_N)\cdot (\nabla F_N^{\mathrm{id}} - \nabla\tilde \phi)(t_N,q_N) \leq 0. \label{eq:grad_inner_prod_I}
	\end{align}
	Since $C_\leq^{(M)}$ is a convex cone, we can substitute $q' - q_\star$ by $q'$ in \eqref{eq:grad_inner_prod}, so by (4.11) of Lemma 4.4 in \cite{Mourrat2023FreeEnergyUpperBound}, 
	\begin{align}
		\nabla\tilde \phi - \nabla F_N \in  {C_\leq^{(M)}}^*. \label{eq:dual_cone}
	\end{align}
	In particular, by (4.12) of Lemma 4.4 in \cite{Mourrat2023FreeEnergyUpperBound}, we have
	\begin{align}
		(\div \tilde\phi
		-
		\div F_N^{\mathrm{id}}
		)(t_N,q_N)
		=
		\sum_{m=1}^M \partial_{q_m} \tilde \phi(t_N,q_N) - \partial_{q_m}F_N(t_N,q_N) \geq 0.
	\end{align}
	Thus, as $\partial_t F_N^{\mathrm{id}}(t_N,q_N) = \partial_t \tilde \phi(t_N,q_N)$ and $\tilde\phi$ is smooth, we conclude that
	\begin{align*}
		(\partial_t\phi - \div \phi)(t_\star,q_\star)
		=
		\limsup_{N\rightarrow\infty}
		(\partial_t \tilde\phi - \div \tilde\phi)(t_N,q_N)
		\leq 
		\limsup_{N\rightarrow\infty}
		(\partial_t F_N^{\mathrm{id}} - \div F_N^{\mathrm{id}})(t_N,q_N)
		=0,
	\end{align*}
	and the proof is completed.
	
	\subsection{Proof of Proposition~\ref{prop:HJCremPsiCOnvexDual}}
	\label{sec:proof of Prop:HJConvexDual}
	Let $\pp\in Q_2$.
	By Theorem~\ref{thm:InitCond} and the definition of convex dual in \eqref{eq:aslkfjeCD},
	\begin{align}
		\Psi^*(\pp) 
		=  \log 2 + \sup_{\yy\in Q_2} h_\pp(\yy),
		\label{eq:Psi*}
	\end{align}
	where
	\begin{align}\label{eq:Psi*-1}
		h_{\pp}(\yy) \coloneqq \int_0^1 \pp(u) \yy(u)\, \d u  - \int_0^1 \brac{\yy(u) - \sfrac{\log 2}{u^2}}_+ \d u.
	\end{align}
	If $\norm{\pp}_\infty > 1$, then there exists $u_1 = u_1(\pp) \in [0,1)$ such that $\pp(u) > 1$ for all $u \geq u_1$ since $\pp$ is increasing. We set $\tilde\yy = y\1_{[u_1,1)}$ with $y\geq \frac{\log 2}{u_1^2}$. We have 
	\begin{align}\label{eq:Psi*0}
		\Psi^*(\pp)
		&= \log 2 + \sup_{\yy\in Q_2} h_\pp(\yy)\notag\\
		&\geq \log 2 +  h_\pp(\tilde\yy)\notag\\
		&=
		\int_{u_1}^1 \pp(u)y \,\d u 
		- \int_{u_1}^1  y - \sfrac{\log 2}{u^2}\, \d u + \log 2 \notag \\
		&\geq y\int_{u_1}^1(\pp(u) - 1) \,\d u.
	\end{align}
	Since $\pp$ is increasing, $\pp(u) > 1$  for all $u\in [u_1,1)$, so the integral $\int_{u_1}^1(\pp(u) - 1) \,\d u$ in the last line of \eqref{eq:Psi*0} is strictly positive. Thus, the last line of \eqref{eq:Psi*0} can be arbitrarily large as $y\rightarrow\infty$.
	
	Note that for $\yy\in Q_2$, since $\yy$ is increasing,
	\begin{align}
		\int_0^1 \brac{\yy(u) - \sfrac{\log 2}{u^2}}_+ \d u = \int_{u_*(\yy)}^1  \yy(u) - \sfrac{\log 2}{u^2} \,\d u = \log 2 - \sfrac{\log 2}{u_*(y)}   + \int_{u_*(\yy)}^1  \yy(u) \,\d u,
	\end{align}
	where \begin{align}\label{eq:Psi*0.75}
		u_*(\yy) 
		= \sup\gkl{u\in (0,1)\colon \yy(u) < \sfrac{\log 2}{u^2}},
	\end{align}
	so
	\begin{align}\label{eq:Psi*0.5}
		h_\pp(\yy) = \int_0^1 \pp(u) \yy(u)\, \d u  - \int_{u_*(\yy)}^1 \yy(u)\, \d u  + \sfrac{\log 2}{u_*(\yy)} - \log 2.
	\end{align}
	
	If $\norm{\pp}_1\geq 1$, we set $\tilde\yy \equiv y$, where $y\geq \log 2$. Then $u_*(\tilde\yy) = \sqrt{\frac{\log 2}{y}}$. 
	Thus, by \eqref{eq:Psi*} and \eqref{eq:Psi*0.5},
	\begin{align}
		\Psi^*(\pp)
		= \log 2 + \sup_{\yy\in Q_2} h_\pp(\yy)
		&\geq \log 2 +  h_\pp(\tilde\yy)\notag\\
		%&= \int_0^1 \pp(u) y \,\d u  - \int_0^1 \brac{y - \sfrac{\log 2}{u^2}}_+ \d u  + \log 2 \notag\\
		&= \int_0^1 \pp(u) y \,\d u  - \int_{u_*(\tilde\yy)}^1 y \, \d u  + \sfrac{\log 2}{u_*(\tilde\yy)} \notag\\
		%&= y \,\norm{\pp}_1  - \kl{y-\sqrt{y\log 2} + \log 2 - \sqrt{y\log 2}} + \log 2 \notag\\
		&=  y \,(\norm{\pp}_1 - 1) + 2\sqrt{y\log 2}.
	\end{align}
	Since $\norm{\pp}_1\geq 1$, the lower bound above can be arbitrarily large as $y\rightarrow\infty$.

	Finally, suppose that both $\norm{\pp}_1<1$ and $\pp\leq 1$ on $[0,1)$. We first establish the lower bound by choosing the constant path $\tilde\yy \equiv \sfrac{\log 2}{(1-\norm{\pp}_1)^2}$. Then,
	\begin{align}\label{eq:Psi*2}
		u_*(\tilde\yy) 
		%= \sup\left\{u\in [0,1): \tilde\yy(u) < \sfrac{\log 2}{u^2}\right\} 
		= \sfrac{\sqrt{\log 2}}{\sqrt{\sfrac{\log 2}{(1-\norm{\pp}_1)^2}}}
		= 1-\norm{\pp}_1.
	\end{align}
	By \eqref{eq:Psi*}, \eqref{eq:Psi*0.5} and \eqref{eq:Psi*2},
	\begin{align}\label{eq:Psi*1}
		\Psi^*(\pp)
		\geq 
		h_\pp(\tilde \yy) &=
		\sfrac{\log 2}{(1-\norm{\pp}_1)^2} \int_0^1 \pp(u)\, \d u
		-
		\int_{u_*(\tilde\yy)}^1 \sfrac{\log 2}{(1-\norm{\pp}_1)^2} \, \d u + 
		\sfrac{\log 2}{1-\norm{\pp}_1}.
	\end{align}
	It remains to show the matching upper bound.
	If $\norm{\pp}_1=0$, then the monotonicity of $\pp$ implies that $\pp\equiv 0$. In this case, obviously, $\Psi^*(\pp)= \log 2 = \frac{\log 2}{1-\norm{\pp}_1}$. From now on, we assume $\norm{\pp}_1 \in (0,1)$.
	
	Note that 
	\begin{align}\label{eq:Psi*3}
		\sup_{y\in Q_2} h_{\pp}(\yy) 
		= 
		\max
		\gkl{
			\sup_{\yy\in Q_2,u_*(\yy) = 1} h_{\pp}(\yy),
			\sup_{\yy\in Q_2,u_*(\yy) < 1} h_{\pp}(\yy)
		},
	\end{align}
	where $u_*(\yy)$ is defined in \eqref{eq:Psi*0.75}. By this definition, for each $\yy\in Q_2$ with $u_*(\yy)=1$, we have $\yy(u) \leq \log 2$ for all $u\in [0,1)$.
	This and \eqref{eq:Psi*0.5} imply
	\begin{align}\label{eq:Psi*3.4}
		h_{\pp}(\yy) = \int_0^1 \yy(u) \pp(u)\, \d u 
		\leq \norm{\pp}_1\log 2 .
	\end{align}
	Thus,
	\begin{align}\label{eq:Psi*4}
		\sup_{\yy\in Q_2,u_*(\yy) = 1} h_p(\yy) \leq  \norm{\pp}_1 \log 2.
	\end{align}
	On the other hand, for $\yy\in Q_2$ with $u_*(\yy)<1$, we rewrite  \eqref{eq:Psi*0.5} as
	\begin{align}\label{eq:Psi*4.3}
		h_{\pp}(\yy)
		&=
		\int_0^1 \yy(u)\pp(u)\, \d u
		-
		\int_{u_*(\yy)}^1 \yy(u)\, \d u - \log 2 + \sfrac{\log 2}{u_*(\yy)} \\
		&=
		\int_0^{u_*(\yy)} \yy(u) \pp(u)\, \d u
		+ 
		\int_{u_*(\yy)}^1 \yy(u)(\pp(u)-1)\, \d u 
		- \log 2
		+
		\sfrac{\log 2}{u_*(\yy)} \label{eq:Psi*4.3.2}.
	\end{align}
	By the definition of $u_*$ in \eqref{eq:Psi*0.75}, it holds for $u\in[0,u_*(\yy))$ that $\yy(u) < \frac{\log 2}{u_*(\yy)^2}$. Also, $\yy(u) \geq \frac{\log 2}{u_*(\yy)^2}$ for $u \in(u_*(\yy),1)$. Since we assumed that $\pp\leq 1$ on $[0,1)$, applying this to the right-hand side of \eqref{eq:Psi*4.3.2} gives, 
	\begin{align}\label{eq:Psi*4.5}
		h_p(\yy)
		&\leq 
		\int_0^{u_*(\yy)} \sfrac{\log 2}{u_*(\yy)^2} \pp(u)\, \d u
		+ 
		\int_{u_*(\yy)}^1 \sfrac{\log 2}{u_*(\yy)^2} (\pp(u)-1)\, \d u 
		- \log 2
		+
		\sfrac{\log 2}{u_*(\yy)} \notag\\
		&=
		\sfrac{\log 2}{u_*(\yy^2)}
		\left(\norm{\pp}_1 - (u_*(\yy)-1)^2\right) 
		=
		g(u_*(\yy)),
	\end{align}
	where for $x\in (0,1)$,
	\begin{align}
		g(x) \coloneqq \sfrac{\log 2}{x^2}
		\kl{\norm{\pp}_1 - (x-1)^2}.
	\end{align}
	Note that $g$ attains its maximum on $(0,1)$ in $1-\norm{\pp}_1$ with maximum value $\sfrac{\norm{\pp}_1 \log 2}{1-\norm{\pp}_1}$.
	Thus, taking the supremum over $\yy$ in \eqref{eq:Psi*4.5} gives 
	\begin{align}\label{eq:Psi*5}
		\sup_{\yy\in Q_2,u_*(\yy)<1} 
		h_{\pp}(\yy) 
		\leq \sfrac{\norm{\pp}_1 \log 2}{1-\norm{\pp}_1}.
	\end{align}
	Combining \eqref{eq:Psi*3}, \eqref{eq:Psi*4} and \eqref{eq:Psi*5}, we get that
	\begin{align}
		\sup_{y\in Q_2} h_{\pp}(\yy)
		\leq 
		\max
		\left\{
		\norm{\pp}_1\log 2,
		\sfrac{\norm{\pp}_1\log 2}{1-\norm{\pp}_1}
		\right\}
		= \sfrac{\norm{\pp}_1\log 2}{1-\norm{\pp}_1},
	\end{align}
	where the equality in the last step follows from the fact that $1-\norm{\pp}_1\in (0,1)$. Thus,
	\begin{equation*}
		\Psi^*(\pp) =  \sup_{y\in Q_2} h_{\pp}(\yy) + \log 2 \leq \sfrac{\norm{\pp}_1\log 2}{1-\norm{\pp}_1} + \log 2 = \sfrac{\log 2}{1-\norm{\pp}_1}.\qedhere
	\end{equation*}
	
	\section{Outlook}
	\label{sec:outlook}

    \subsection{Removing the Left Derivative Hypothesis}\label{sec:rem case}

    We claim that to remove the left derivative hypothesis in Theorem~\ref{thm:main}, it will suffice to study the random energy model (REM). Recall that the REM corresponds to the special case when the covariance function $A(x) = \1_{x = 1}$, and it's 
 Hamiltonian takes a particularly simpler form:
		\[
		H_N^{\1_{x = 1}}(\sigma) = \sqrt{N} z_{\sigma}
		\]	
		where the $(z_{\sigma})_{\sigma \in \Sigma}$ are independent Gaussians. Suppose that we can show that 
        \begin{equation}\label{eq:leftderiv1}
         \liminf_{N\rightarrow\infty} F_N^{\1_{x = 1}}(t,q) \geq \liminf_{N\rightarrow\infty} F_N^{\mathrm{id}}(t,q).
        \end{equation}
    Then by Lemma~\ref{lem:GP} and the fact that $ \1_{x = 1} \leq A \leq \mathrm{id}$ for any convex $A$, we conclude that
    \begin{equation}\label{eq:leftderiv2}
    F_N^{\1_{x = 1}}(t,q)  \leq F_N^{A}(t,q)  \leq F_N^{\mathrm{id}}(t,q) .
    \end{equation}
    Combining \eqref{eq:leftderiv1} and \eqref{eq:leftderiv2} and applying Theorem~\ref{thm:main} to $A = \mathrm{id}$ will give that 
    \[
    \lim_{N\rightarrow\infty} F_N^{\1_{x = 1}}(t,q) = \lim_{N\rightarrow\infty}  F_N^{A}(t,q) =  \lim_{N\rightarrow\infty}  F_N^{\mathrm{id}}(t,q) = \sup_{\lambda\in [0,1)}
			\left\{
			\lambda t 
			+ \int_{1-\lambda}^1 \q(u) \d u
			- \frac{\log 2}{1-\lambda}
			\right\}.
    \]
    
    Unfortunately, computing the limit of $F_N^{\1_{x = 1}}(t,q) $ cannot be deduced from Theorem~\ref{thm:main} by a simple continuity argument. We sketch the main barrier below.
    
    Let $A_p(x) = x^p$. Notice that as $p \to \infty$, $A_p(x) \to \1_{x = 1}$ pointwise. We have that the Hamiltonian can be written as
	\[
	H_{N,p}(x) = \sqrt{N} \sum_{i = 1}^N \bigg( A_p\Big( \frac{i}{N} \Big) - A_p\Big( \frac{i - 1}{N} \Big) \bigg)^{1/2} z_{\restr{\s}{i}}.
	\]
	If we consider the free energy
	\begin{align}
		F_N(t,\q_M;p) 
		&\coloneqq - \frac{1}{N} \EX{\textstyle
			\log\kl{ \sum_{\alpha \in \N^M} v_\a \sum_{\sigma \in \{-1,1\}^N}   
				\exp\bkl{H_{N,p}(t,\q_M,\s,\a)}}
		}
        \nonumber
	\end{align}
	where
	\begin{align*}
		H_{N,p}(t,\q_M,\s,\a)&\coloneqq\sqrt{2t}\, H_{N,p}(\sigma) - Nt  + \sqrt{2}\,Y_{\q_M}(\sigma,\alpha) - N q_M.
	\end{align*}
	Then we see that uniformly in $t,\q_M$, we have that the Lipschitz constant of $p \mapsto F_N(t,\q_M;p) $ is bounded from above by $p$. As usual, let $\langle \cdot \rangle_p$ denote the Gibbs average with respect to $H_{N,p}(t,\q_M,\s,\a)$. This is because of an integration by parts,
	\begin{align*}
		&\partial_p F_N(t,\q_M;p) 
        \nonumber \\
        &= -\frac{1}{N}\E\langle \partial_p 	H_{N,p}(x)  \rangle_p
		\\&= -\frac{p}{2N}\E \bigg\langle \sqrt{N} \sum_{i = 1}^N \bigg( A_p\Big( \frac{i}{N} \Big) - A_p\Big( \frac{i - 1}{N} \Big) \bigg)^{-1/2}  \bigg( A_{p-1}\Big( \frac{i}{N} \Big) - A_{p-1}\Big( \frac{i - 1}{N} \Big) \bigg) z_{\restr{\s}{i}}  \bigg\rangle_p
		\\&\leq \frac{p}{2}\E \bigg\langle  A_{p - 1}\Big( \frac{\sigma^1 \wedge \sigma^2}{N} \Big) - A_{p - 1}\Big(1 \Big) \bigg\rangle_p 
	\end{align*}
	so
	\[
	|	\partial_p F_N(t,\q_M;p) | \leq \frac{p}{2}.
	\]
	This is unfortunate because by the simple bound we do not get a uniform Lipschitz constant, so taking $p \to \infty$ will not work. Instead, one will have to compute the limit of $F_N^{\1_{x = 1}}(t,q)$ directly, which can potentially be done using the explicit computation of Section~\ref{sec:initalCondition}.

	\subsection{Outside the Weak Correlation Regime} \label{sec:outside}

	We now explore what happens for non-convex covariance functions. In the CREM, recall from \eqref{eq:3FECREMFormula} that the limit free energy is given by 	\begin{equation}\label{eq:TRUE_FE}
		f^{A}_{CREM}(t, 0) = t \hat A(x(t)) - (1-x(t)) \log 2 - 2 \sqrt{t \log 2} \int_0^{x(t)} \sqrt{\hat A'(x)}\, \d x, 
	\end{equation}
	where $x(t) = \sup\kl{x\in (0,1)\colon \hat A'(x) \geq \frac{\log 2}{t}}$. This formula holds even for $A$ that are not necessarily convex. However, there appear to be some technical barriers that prevent an extension of the simple variational formula to this setting. In fact, we will show in this section that we can explicitly solve the variational formula for a simple non-convex covariance function and prove that the variational formula is a strict lower bound.

	\iffalse
	The following result shows that the limiting free energy is always a supersolution if $A$ is concave. 
	\begin{proposition}
		Let  $F_N$ be the enriched free energy of the CREM with concave covariance function $A$, see \eqref{eq:3FECREMFormula}. Let $\Psi$ be as in Theorem~\ref{thm:InitCond}.
		Then, the unique viscosity solution $f$ of \eqref{eq:HJ2} with $f(0,\cdot) = \Psi$ from Proposition~\ref{prop:HopfFormulaCREMNeu} satisfies 
		\begin{align}\label{eq:freeEnEqualsViscosSol}
			f(t,\q) &\geq \lim_{N\uparrow\infty} F_N(t,\q),
		\end{align}
		for all $t\geq 0$ and $\q\in Q_2$.
	\end{proposition}
	\fi
	To this end, we consider the a simple two speed non-convex function, which is parameterized by $(\theta,c)$ 
	\[
	A_{\theta,c}(x) = \theta_1 x \1( x \leq c ) + [\theta_2 (x - c) + \theta_1 c] \1( x > c )
	\]
	where $\theta_1 := \theta > 1$ is the slope on $[0,c]$ and $\theta_2 := \frac{1 - \theta_1 c}{1 - c}$ is the slope on $[c,1]$. 
	Since $A$ must be increasing, we require that $\theta_1 c < 1$. When applied to the piecewise linear covariance function $A_{\theta,c}(x)$ the limit in \eqref{eq:TRUE_FE} simplifies to the following:
	\begin{lemma}
		\label{lem:f*}
		For the non-linearity given by $A_{\theta,c}$ we have
		\[
    f^{A_{\theta,c}}_{CREM}(t,0) = \begin{cases}
			- \log 2&t \leq \frac{\log 2}{\theta_1}\\
			c \theta_1 t - (1-c) \log 2 - 2c \sqrt{\theta_1 (\log 2) t}& \frac{\log 2}{\theta_1} < t \leq \frac{\log 2}{\theta_2}\\
			t - 2 \sqrt{\log 2 t} (c \sqrt{\theta_1} + (1-c ) \sqrt{\theta_2})& \frac{\log 2}{\theta_2} < t 
		\end{cases}.
		\]
	\end{lemma}

	We compare this with the solution for the variational formula. For the piecewise linear non-convex covariance function, we define the naive extension of the variational formula by
	\begin{align}
		f_{var}^{A_{\theta,c}}(t,\q)=\sup_{\pp \in Q_2} \inf_{\yy\in Q_2} \brac{\Psi(\yy) + \int_0^1 \pp(u) \cdot (\q(u)-\yy(u)) \dd u + t \int_0^1 A(\pp(u)) \dd u} \label{eq:var_prob}
	\end{align}
	where
	\begin{align}
		\Psi(\q)  = -\log 2 + \int_0^1 \brac{\q(u) - \sfrac{\log 2}{u^2}}_+ \dd u. 
	\end{align}
	This can be explicitly solved, which is the main result of the following lemma.
	\begin{lemma}\label{lem:var_simplified}
		For the non-linearity given by $A_{\theta,c}$ we have
		\[
    f_{var}^{A_{\theta,c}}(t,0) = \sup_{\pp \in Q_2 \cap \{ \|\pp\|_{\infty} \leq 1 \}}  \bigg(  -\frac{\log 2}{1 - \| \pp \|_1} + t \int_0^1 A_{\theta,c}(\pp(u)) \dd u \bigg).
		\]
	\end{lemma}
	\begin{proof}
		Notice that
		\begin{equation}\label{eq:var_forumla_2speed}
			f_{var}^{A_{\theta,c}}(t,0)=\sup_{\pp \in Q_2} \inf_{\yy\in Q_2} \brac{ \int_0^1 \brac{\yy(u) - \sfrac{\log 2}{u^2}}_+ \dd u - \int_0^1 \pp(u) \cdot \yy(u) \dd u + t \int_0^1 A_{\theta,c}(\pp(u)) \dd u - \log 2}.
		\end{equation}
		We solve this variational problem explicitly. Let $u_*(\yy) = \sup \{ u \in [0,1] \coloneqq \yy(u) < \frac{\log 2}{u^2} \}$. We can write \eqref{eq:var_forumla_2speed} as
		\begin{align*}
        f_{var}^{A_{\theta,c}}(t,0)= \sup_{\pp \in Q_2} \inf_{\yy\in Q_2} \brac{ \int_{u_*}^1 (1 - \pp(u) ) \yy(u)  \dd u - \int_0^{u_*} \pp(u)  \yy(u) \dd u + t \int_0^1 A_{\theta,c}(\pp(u)) \dd u - \frac{\log(2)}{u_*} }.
		\end{align*}
		First of all, notice that if $\pp(u) > 1$ for any $u$, then the minimum is attained at $-\infty$, so we can consider the case when $\| \pp \|_1 < 1$. Notice that $(1 - \pp(u)) \geq 0$ and $\pp(u) \geq 0$ on $Q_2$, so the function is minimized if $\yy(u)$ is as big as possible on $[0,x_*]$ and as small as possible on $[x_*,1]$. Since $\yy$ must be increasing, this implies that $\yy$ must be constant. For constant values of $\yy$, we have
		\[
		u_*(y) = \sqrt{\frac{\log 2}{y} } \wedge 1.
		\]
		Therefore, for any fixed $p \in  Q_2$ such that $\{ \|\pp\|_{\infty} \leq 1 \}$, we have that
		\begin{align*}
			&f_{var}^{A_{\theta,c}}(t,0) \nonumber \\
            &=  \inf_{y \geq 0} \brac{\int_{x_*}^1 (1 - \pp(u) ) y  \dd u - \int_0^{x_*} \pp(u)  y \dd u + t \int_0^1 A(\pp(u)) \dd u - \frac{\log(2)}{u_*}}
			\\
            &= \min\bigg( 
            \inf_{y \leq \log 2} \brac{ - y\|\pp \|_1 + t \int_0^1 A_{\theta,c}(\pp(u)) \dd u - \log 2},
            \nonumber \\
            & \qquad\qquad\qquad\qquad\qquad\qquad\qquad\qquad
            \inf_{y \geq \log 2} y  - y \|\pp \|_1 - 2\sqrt{y \log 2} + t \int_0^1 A_{\theta,c}(\pp(u)) \dd u   
            \bigg)
			\\&= \min\brac{ -\log 2 (1 + \|\pp \|_1) + t \int_0^1 A_{\theta,c}(\pp(u)) \dd u -  \log 2, -\frac{\log 2}{1 - \| \pp \|_1} + t \int_0^1 A_{\theta,c}(\pp(u)) \dd u }
			\\&=  -\frac{\log 2}{1 - \| \pp 
            \|_1} + t \int_0^1 A(\pp(u)) \dd u,
		\end{align*}
        and the proof is completed.
	\end{proof}
	
	Using the simplified formula in Lemma~\ref{lem:var_simplified}, we can explicitly solve the variational formula for the two-piece covariance function. 
	
	\begin{lemma}\label{lem:solution-two-speed}
		For any $c \in [0,1]$ and $0 < \theta < \frac{1}{c}$, the variational problem \eqref{eq:var_prob} with non-linearity given by $A_{\theta,c}$ has a unique solution given by
		\begin{equation}
			f_{var}^{A_{\theta,c}}(t,0) =  \begin{cases}
				- \log 2, &t \leq \frac{\log 2}{\theta_1},\\
				\theta_1 t - 2 \sqrt{ \ln2 \theta_1 t }, & \frac{\log 2}{\theta_1} < t \leq \frac{\log 2}{(1-c^2)\theta_1}, \\
				\theta_1 ct - \frac{\ln 2}{1 - c}, & \frac{\log 2}{(1-c^2)\theta_1} < t \leq \frac{\ln 2}{(1 - c)^2 \theta_2}, \\
				t - 2\sqrt{\ln 2 \theta_2 t }, & \frac{\ln 2}{(1 - c)^2 \theta_2}  < t.  
			\end{cases}\label{eq:solution-two-speed}
		\end{equation}
	\end{lemma}
	\begin{proof}
		Our goal is to compute
		\[
		f_{var}^{A_{\theta,c}}(t,0) = \sup_{\pp \in Q_2 \cap \{ \|\pp\|_{\infty} \leq 1 \}}  \bigg(  -\frac{\log 2}{1 - \| \pp \|_1} + t \int_0^1 A(\pp(u)) \dd u \bigg).
		\]
		\textit{1. Upper Bound:} By Jensen's inequality, since $p \in Q_2$ is increasing, we have
		\[
		\sup_{\pp \in Q_2 \cap \{ \|\pp\|_{\infty} \leq 1 \}}  \bigg(  -\frac{\log 2}{1 - \| \pp \|_1} + t \int_0^1 A(\pp(u)) \dd u \bigg) \leq \sup_{\pp \in \{ \| \pp\|_1 \leq 1 \}}  \bigg(  -\frac{\log 2}{1 - \| \pp \|_1} + t A_{\theta,c}( \| \pp \|_1 ) \bigg).
		\]
		We have reduced the problem to a one-dimensional optimization problem 
		\begin{align}
			f_{var}^{A_{\theta,c}}(t,0) \leq h(\lambda) &= \sup_{0 \leq \lambda \leq 1}  \bigg(  -\frac{\log 2}{1 - \lambda} + t A_{\theta,c}( \lambda )   \bigg) \nonumber
			\\&= \sup_{0 \leq \lambda \leq 1}  \bigg(  -\frac{\log 2}{1 - \lambda} + t   c \theta_1 \lambda \1(\lambda < c) + t( c \theta_1  + (\lambda - c)\theta_2  ) \1(\lambda \geq c)  \bigg) .\label{eq:upbound_2speed}
		\end{align}
		The critical points of this function satisfy
		\[
		h'(\lambda) = -\frac{\log 2}{(1 - \lambda)^2} + t ( \theta_1 \1( \lambda \in  [0,c) ) +  \theta_2 \1( \lambda \in  [c,1] ) ).
		\]
		Solving this maximization problem by using the critical point condition or substituting the boundary on the $4$ cases gives the representation in \eqref{eq:solution-two-speed}. 
		
		We will see that using Jensen's does not cost anything, and this inequality is, in fact, an equality. We will show in the next part of the proof that the maximizer is attained when $\pp$ is supported on a region of $[0,1]$ where $A_{\theta,c}$ is affine. 
		\\\\
		\textit{2. Lower Bound:} We now argue that the upper bound is sharp. First of all, if we restrict our maximizers to the set of step functions of the form
		\[
		\pp_{\xi_1, \xi_2, \lambda}( x ) = \lambda \1(x \in [1-\xi_2, 1-\xi_1]) + \1(x \in [1-\xi_1, 1] ),
		\]
		where $\|\pp\|_1 = \xi_1 + \lambda(\xi_2 - \xi_1) = \lambda$ and $0 < \xi_1 < \xi_2 < 1$, then we get an obvious lower bound
		\begin{align*}
			\sup_{\pp \in Q_2 \cap \{ \|\pp\|_{\infty} \leq 1 \}} t\int_0^1 A(\pp(u)) \dd u  &\geq \sup_{ p_{\xi_1, \xi_2, \lambda} } \bigg(  t\int_0^1 A(\pp_{\xi_1, \xi_2, d}(u)) \dd u \bigg)\\
			&\geq \sup_{ \substack{ \xi_1 + c(\xi_2 - \xi_1) = \lambda \\ 0 < \xi_1 < \xi_2 < 1 } } \bigg( (\xi_2 - \xi_1)\theta_1 ct + \xi_1 t \bigg)
			\\&= \sup_{ \frac{\lambda - c}{1-c} \wedge 0 \leq \xi_1 \leq \lambda } \big( (  \lambda - \xi_1) \theta_1 + \xi_1 \big)t,
		\end{align*}
		which can be solved explicitly,
		\[
		\sup_{ \frac{\lambda - c}{1-c} \wedge 0 \leq \xi_1 \leq \lambda } \big( (  \lambda - \xi_1) \theta_1 + \xi_1 \big)t = \sup_{0 \leq \lambda \leq 1} ( \lambda \theta_1 t \1(\lambda < c) + t( c \theta_1  + (\lambda - c)\theta_2  ) \1(\lambda \geq c) ) = \sup_{0 \leq \lambda \leq 1} tA_{\theta,c}(\lambda).
		\]
		Therefore, we have the lower bound
		\[
		f_{var}^{A_{\theta,c}}(t,0) \geq  \sup_{0 \leq \lambda \leq 1}  \bigg(  -\frac{\log 2}{1 - \lambda} + t A_{\theta,c}( \lambda )   \bigg),
		\]
		which is identical to \eqref{eq:upbound_2speed}.
	\end{proof}
	
	By Lemma~\ref{lem:f*} and Lemma~\ref{lem:solution-two-speed}, we see that the free energy $f^{A_{\theta,c}}_{CREM}(t,0)$ is bounded strictly from above by $f_{var}^{A_{\theta,c}}(t,0)$. 
	\begin{proposition}\label{prop:one-sided}
		For any $c \in [0,1]$ and $0 < \theta < \frac{1}{c}$, and covariance function $A_{\theta,c}$ we have
		\[
		\lim_{N\uparrow\infty} F_N(t,0) = f^{A_{\theta,c}}_{CREM}(t,0) < f_{var}^{A_{\theta,c}}(t,0) . 
		\]
	\end{proposition}

    The fact that we have a one-sided bound in Proposition~\ref{prop:one-sided} is not too surprising. The function $A_{\theta,c}$ is concave, so the proof of Proposition~\ref{prop:supersolution} can be adapted to show that $f_{var}^{A_{\theta,c}}(t,0)$ is always a subsolution, that is
    \[
    \lim_{N\uparrow\infty} F_N(t,0) \leq f_{var}^{A_{\theta,c}}(t,0).
    \]
    However, as seen in Proposition~\ref{prop:one-sided}, the inequality is not sharp. There are some gaps in the current theory mainly caused by the lack of monotonicity in the non-convex case, so more work to extend the theory to this setting remains open. 
	
	\section*{Acknowledgments}
	
	The authors thank Pascal Maillard for his encouragement to pursue this line of research and for pointing out Biggins' shape theorem in \cite{biggins1978asymptotic}. The authors also thank Lisa Hartung, Jean-Christophe Mourrat, and Victor Issa for stimulating discussions.
	AA acknowledges support from ANR-DFG grant REMECO (No. 446173099) and TRR 146 (No.  233630050).
	FHH acknowledges support from Eliran Subag’s ISF grant (No. 2055/21) and ERC grant (No. 101165541, Horizon Europe). JK acknowledges the support of the Natural Sciences and Engineering Research Council of Canada (NSERC) [RGPIN-2020-04597, DGECR-2020-00199], the Canada Research Chairs programme, and the Ontario Research Fund.

	%\bibliographystyle{abbrv}
	%\bibliography{CREM_and_HJ}

	% \appendix
	
 %    \iffalse
 %    \section{Embedding into a Higher Dimensional Space}\label{app:embed}

 %    \todo{Mention the following somewhere?}
    
 %    \jk{
 %    In this section, we explain why we cannot embed the enriched CREM cannot be embedded into $\R^{2^{N+1}}$ as in the proof of Theorem~\ref{thm:main}: There, the reference measure $P_N$ is the uniform distribution on the embedding of $\leaves{N}$, so it is a discrete probability measure with $2^N$ point masses. To obtain a contradiction, assume that $P_N = P_1^{\otimes 2^{N+1}}$, where $P_1$ is a probability measure on $\R$. Since $P_N$ is discrete, $P_1$ must also be a discrete measure. Let $\ell\in \N$ be the number of point masses of $P_1$.  Then, $P_N$ has $\ell^{2^{N+1}}$~point masses, which contradicts the fact that $P_N$ has $2^N$ point masses. Thus, the embedding of the enriched CREM into $\R^{2^{N+1}}$ as in the proof of Theorem~\ref{thm:main} does not fit the setting of~\cite{ChenMourrat2023nonconvexVectorSpin}, so some work has to be done to establish the variational formula for the limiting free energy.
 %    }
 %    \fi

\end{document}